\documentclass{amsart}
\usepackage{graphicx}
\usepackage{amssymb}
\usepackage{amsfonts}
\usepackage{hyperref}
\usepackage{mathrsfs}
\swapnumbers
\newtheorem{theorem}{Theorem}[section]
\newtheorem{lemma}[theorem]{Lemma}
\newtheorem{corollary}[theorem]{Corollary}
\newtheorem{proposition}[theorem]{Proposition}
\theoremstyle{definition}
\newtheorem{definition}[theorem]{Definition}
\newtheorem{assumption}[theorem]{Assumption}
\newtheorem{remark}[theorem]{Remark}
\newtheorem{example}[theorem]{Example}

\numberwithin{equation}{section}
\theoremstyle{plain}

\numberwithin{equation}{section} 
\numberwithin{figure}{section} 
\theoremstyle{plain}
\theoremstyle{plain}
\theoremstyle{remark}
\newtheorem*{acknowledgement*}{Acknowledgement}
\theoremstyle{example}

\usepackage[margin=3cm]{geometry}

\newcommand{\cA}{{\mathcal A}}
\newcommand{\cB}{{\mathcal B}}
\newcommand{\cC}{{\mathcal C}}
\newcommand{\cD}{{\mathcal D}}

\newcommand{\cF}{{\mathcal F}}
\newcommand{\cG}{{\mathcal G}}
\newcommand{\cH}{{\mathcal H}}

\newcommand{\cL}{{\mathcal L}}

\newcommand{\cS}{{\mathcal S}}

\newcommand{\cX}{{\mathcal X}}
\newcommand{\cY}{{\mathcal Y}}
\newcommand{\cZ}{{\mathcal Z}}
\newcommand{\te}{{\theta}}

\newcommand{\Om}{{\Omega}}
\newcommand{\om}{{\omega}}
\newcommand{\ve}{{\varepsilon}}
\newcommand{\del}{{\delta}}

\newcommand{\sig}{{\sigma}}
\newcommand{\al}{{\alpha}}

\newcommand{\ka}{{\kappa}}
\newcommand{\la}{{\lambda}}

\newcommand{\bbC}{{\mathbb C}}
\newcommand{\bbE}{{\mathbb E}}

\newcommand{\bbN}{{\mathbb N}}
\newcommand{\bbP}{{\mathbb P}}
\newcommand{\bbR}{{\mathbb R}}

\newcommand{\bbZ}{{\mathbb Z}}
\newcommand{\bbI}{{\mathbb I}}

\begin{document}
\title[]{Statistical properties of  Markov shifts (part I)}  
 \vskip 0.1cm
 \author{Yeor Hafouta}
\address{
Department of Mathematics, The University of Florida}
\email{yeor.hafouta@mail.huji.ac.il}%

\thanks{ }
\dedicatory{  }
 \date{\today}

\maketitle
\markboth{Y. Hafouta}{ } 
\renewcommand{\theequation}{\arabic{section}.\arabic{equation}}
\pagenumbering{arabic}

\begin{abstract}
We prove central limit theorems, Berry-Esseen type theorems, almost sure invariance principles, large deviations and Livsic type regularity for partial sums of the form $S_n=\sum_{j=0}^{n-1}f_j(...,X_{j-1},X_j,X_{j+1},...)$, where $(X_j)$ is an inhomogeneous  Markov chain satisfying some mixing assumptions and $f_j$ is a sequence of sufficiently regular functions. Even though the case of non-stationary chains and time dependent functions $f_j$ is more challenging, our results seem to be new already for stationary Markov chains. They also seem to be new for non-stationary Bernoulli shifts (that is when $(X_j)$ are independent but not identically distributed). This paper is the first one in a series of two papers. In \cite{Work} we will prove local limit theorems including developing the related reduction theory in the sense of \cite{DolgHaf LLT, DS}.

Our results apply to  Markov shifts in random dynamical environment, 
products of random non-stationary positive matrices and other operators, 
random Lyapunov exponents, non-autonomous non-uniformly expanding transformations, as well as several processes arising in statistics and applied probability like linear processes, inhomogeneous  iterated random functions and GARCH processes. Most of these examples seem to only be treated in literature for iid $X_j$ and here we are able to drop both the stationarity and the independence assumptions. However, even in the classical setup of Anosov maps, subshifts of finite type and Gibbs-Markov maps our results seem to be new when working with measures of maximal entropy since we can consider observables which are only H\"older continuous on average. 

Our proofs are based on conditioning on the future instead of the regular conditioning on the past that is used to obtain similar results when $f_j(...,X_{j-1},X_j,X_{j+1},...)$ depends only on $X_j$ (or on finitely many variables). In particular we generalize the Berry-Esseen theorem in \cite{DolgHaf PTRF 1} to functions which depend on the entire path of the chain, and the results in \cite{DolgHaf PTRF 2} about Markov chains to more general chains. Moreover, we obtain results that solely depend on regularity properties of $f_j$ and mixing rates, without assuming any form of ellipticity. 

Our results are significant for both practitioners from statistics and applied probability and theorists in probability theory, ergodic theory and dynamical systems (e.g. we generalize \cite{HennionAoP97} from iid matrices to non-stationary Markovian ones and get optimal rates in the setup of \cite{KestFurs61}).  We expect many other applications of our abstract results, for instance, to Markovian inhomogeneous random walks on $\text{GL}_d(\bbR)$, but in order not to overload the paper this will be discussed in future works, together with the local limit theorems mentioned above. 
\end{abstract}


\section{Introduction}
Let $(Y_j)$ be an independent sequence of zero mean square integrable random variables, and let $S_n=\sum_{j=1}^nY_j$. The classical CLT states that if $\lim_{n\to\infty}\sigma_n=\infty, \sigma_n=\|S_n\|_{L^2}$ then $S_n/\sigma_n$ converges in distribution to the standard normal law if and only if the Lindeberg condition\footnote{Namely that $\lim_{n\to\infty}\sig_n^{-2}\sum_{j=1}^n\bbE[Y_j^2\bbI(|Y_j|\geq \varepsilon \sig_n)]=0$ for all $\varepsilon>0$.} holds. In particular, when $Y_j$ are identically distributed and non-constant then the weak limit is Gaussian.
Note that when $\sig_n\not\to\infty$ then by Kolmogorov's three-series theorem $S_n$ converges almost surely (and also in $L^2$) and so there is no weak limit in general. 
\subsection{CLT rates for independent summands}

The CLT is only an asymptotic result, and in order to make it useful in applications some convergence rate is needed. It is customary to quantify the convergence of $S_n/\sig_n, \sig_n=\|S_n\|_{L^2}$ to the standard normal law by the quantity
$$
\Delta_n:=\sup_{t\in\bbR}\left|\bbP(S_n\leq t\sig_n)-\Phi(t)\right|
$$
where $\Phi(t)$ is the standard normal distibution function.
Then the classical Berry-Esseen theorem (see \cite{Berry, Esseen}) in the iid case states that when $Y_1\in L^3$ then $\Delta_n=O(\sig_n^{-1})=O(n^{-1/2})$. 
 In general, the rate $O(\sig_n^{-1})$ is optimal. Indeed, by a classical result of Esseen \cite{Ess56} in the iid case $\Delta_n=o(n^{-1/2})$ if and only if $\bbE[Y_1^3]=0$ and $Y_1$ does not take values on a lattice\footnote{i.e. a set of the form $a+b\bbZ$ for some $a\in\bbR$ and $b>0$.}.   
 
 In this paper we are interested in the behavior of non-stationary processes. In the context of independent summands this means that $Y_j$ are not identically distributed. In that case the classical result of Berry and Esseen shows that $\Delta_n$ is at most of order $\sig_n^{-3}\sum_{j=1}^n\bbE[|Y_j|^3]$. Similarly to the iid case, when dropping the requirement of identical distribution   when $\sigma_n\to\infty$ then $\Delta_n=o(\sig_n^{-1})$ if and only if $(Y_j)$ is not reducible to a lattice valued sequence and $\sum_{j=1}^n\bbE[Y_j^3]=o(\sig_n^2)$ (see \cite{DolgHaf PTRF 1}). 
All of the above results are obtained by using Fourier analysis methods applied with the Fourier transform of the measure induced by  the law of $S_n/\sig_n$ (i.e. the characteristic function of $S_n/\sig_n$).

\subsection{Weakly dependent summands}
Independence is a very strong and unrealistic assumption. Many real life models are based on weakly dependent variable $Y_k$ instead of independent ones. In the sections below we will briefly review the literature and explain our main results in three main cases: stationary systems, sequential dynamical systems (i.e. purely non-autonomous systems) and 
 random dynamical systems (i.e. random variables in random environments).

\subsubsection{\textbf{Stationary sequences}}
The literature on limit theorems for sums of the form $S_n=\sum_{j=0}^{n-1}f\circ T^j$ for sufficiently regular functions $f$ and sufficiently fast mixing dynamical system $T$ is vast.
In his  seminal paper \cite{Nag61} Nagaev developed an approach which by now is refereed to as the Nagaev-Guivaech method (or the spectral method) and proved that if $X_j$ is a stationary sufficiently well (elliptic) mixing Markov chain and $Y_j=f(X_j)$ for some measurable function $f$ such that $0<\|Y_1\|_{L^2}\leq \|Y_1\|_{L^3}<\infty$ then 
$$
\Delta_n=O(n^{-1/2}).
$$
Note that the CLT itself is due to Dobrushin \cite{Dob56}, see the next section.
Since then optimal CLT rates $O(n^{-1/2})$ (aka Berrry Esseen theorems) were obtained for many classes of stationary processes $Y_k$ under some mixing (weak dependence) and moment assumptions on $Y_k$, see \cite{HH, GH, GO, Jirak,Rio, RE} for a few general approaches for chaotic dynamical systems, Markov chains, Bernoulli shifts and bounded $\phi$-mixing sequences.  Such result have applications in other areas of probability and statistics like products of random matrices (see \cite{PelMat1}), iterated function systems and other processes arising in statistics and applied probability \cite{[15],Jirak}, and many others.
 Of course, there are many other results in literature but it is beyond the scope of this paper to  provide a full list. In the stationary setting our results seem to be new as the setting of functions that depend on the entire path of a Markov chain was not treated, but for sufficiently regular functions we expect such results to follow from \cite{HH}. However,  note that for Markov measures (including measures of maximal entropy (MME)) on subshifts of finite type, Gibbs Markov maps or  Anosov maps (via symbolic representations) our results apply to functions $f$ which are only H\"older continuous on average. This was not treated in literature and does not seem to immediately follow from existing results.

\subsection{Nonstationary sequences}
Traditionally, in literature most results concerning limit theorems are obtained for stationary sequences, which can be viewed as an autonomous dynamical system generated by a single deterministic map preserving the probability law generated by the process. One of the current challenges in the field of stochastic processes and dynamical systems  is to better understand non-stationary processes, namely
random and time-varying dynamical systems, in particular to develop novel probabilistic techniques to prove limit theorems. This direction of research, the ambition of which is to approach more the real by taking in account a time dependence inherent\footnote{e.g. external forces affect the local laws of physics, the uncertainty principle etc.} in some phenomena, has recently seen an enormous amount of activity. Many difficulties and questions emerge from this non-stationarity and time dependence. Let us mention for example the existence of many open questions about the establishment of quenched and sequential limit theorems  for systems with random or non-autonomous dynamics. The study of these systems opens new interplays between probability theory and dynamical systems, and leads to interesting insights in other areas of science. In what follows we  discuss the progress that has been done in recent years concerning non-autonomous dynamical systems and its interplay with our main results.

\subsubsection{\textbf{Sequential dynamical systems and non-stationary processes}}
A sequential dynamical system is formed by composition of different maps $T_j$. The dynamics is described by the time dependent orbits of a point $x$,
$$
x,\,\,\, T_0x,\,\,\, T_1\circ T_0x,\,\,\,T_2\circ T_1\circ T_0x,...
$$
In this setup the goal is to prove limit theorems for Birkhoff sums of the form
$$
S_nf=\sum_{j=0}^{n-1}f_j\circ T_{j-1}\circ\cdots\circ T_1\circ T_0
$$
considered as random variables on an appropriate probability space. Note that given a sequence of random variables $(X_j)$ with values in spaces $\cX_j$ it induces a natural sequence of left shifts $T_j:\cY_j\to\cY_{j+1}$ on the shifted path spaces $\cY_j:=\{(x_{j+k})_{k\in\bbZ}, x_m\in\cX_m\}$. Thus the theory of compositions of different maps coincides with the theory of nonstationary sequences of random variables.

The ``story" here begins with Dobrushin's CLT.
In \cite{Dob56} Dobrushin provided sufficient conditions for the CLT for sufficiently well contracting bounded Markov chains $(Y_j)$, where some growth conditions on $\|Y_j\|_{L^\infty}$ is allowed. This seems to be the first CLT beyond the independent case. We refer to \cite{PelCLT, SV} for a modern presentation and strengthening Dobrushin's CLT.
Since the, the central limit theorem was studied for many classes of non-stationary sequences and time dependent dynamical systems. We refer to  \cite{Bk95,CR,CLBR,LD1,Nonlin, HNTV,  NSV12,NTV ETDS 18, CLT3, CLT2} for a very partial list.

Concerning optimal CLT rates, the first result of this kind beyond the case when the variance of $S_n$ grows linearly fast seems to appear in \cite{DolgHaf PTRF 1}, where Berry-Esseen theorems were obtained for summands of the form $Y_j=f_j(X_j,X_{j+1})$ for uniformly bounded functions $f_j$ and uniformly elliptic inhomogeneous Markov chain $X_j$. In \cite{DolgHaf PTRF 2} we, in particular, extended this result for uniformly elliptic finite state Markov chains  and H\"older continuous functions $f_j=f_j(...,X_{j-1},X_j,X_{j+1},...)$ of the entire path of the chain $(X_m)$.

In this paper we will prove central limit theorems with optimal rates and large deviations for sequences of random variables of the form $Y_j=f_j(...,X_{j-1},X_j,X_{j+1},...), j\geq0$, where $(X_k)$ is a sufficiently well mixing inhomogeneous  Markov chain and $f_j$ are sufficiently regular functions. Even though the main difficulties arise due to time dependence of the functions $f_j$ and non-stationarity of the chain, there seem to be very little results in literature already in the case of a single function $f_j=f$ and a stationary chain beyond the case when $f$ depends only on finitely many coordinates as discussed above. While it has its own theoretical interest, we note that the dependence on the entire path of the chain is inherent in many application like products of random matrices and other operators,
random Lyapunov exponents, non-uniformly expanding transformations, as well as several processes arising in statistics and applied probability like linear processes,  iterated random functions and GARCH sequences, see Section \ref{App}. To demonstrate this natural phenomenon we recall that stationary iterated random function are defined in recursion by $Y_k=G(Y_{k-1},X_k)$ for some measurable function $G(y,x)$ such that $G(\cdot,X_0)$ is contracting on average. Then $Y_k$ depends on $X_j, j\leq k$. Such processes have a wide range of applications in applied probability, see \cite{[15]}, when the case of iid $X_j$ is considered. Here we can drop the independence and the stationarity assumptions and consider inhomogenuous Markov chains instead, which is a more realistic model for random noise than the iid setting. Moreover, we can consider time dependent functions $G_k$ such that $Y_k=G_k(Y_{k-1},X_k)$ which we believe is more realistic. Similarly,  as opposed to almost all results in literature we can consider non-stationary Markov dependent products of random matrices etc. We refer to Section \ref{App} for several other examples.

From a ``dynamical" point of view, compared with \cite{Jirak} we are able to consider Markov shifts instead of Bernoulli shifts (although with exponential approximation coefficients). Already the case when $X_j$ are independent but not identically distributed is not covered in \cite{Jirak}, and so even this case is new. As noted above, from a dynamical perspective we prove optimal CLT rates for H\"older on average observables $f_j$ which are not covered in literature already in the stationary case for Anosov maps and subsfhits (although we need to consider Markov measures like MME). Let us also mention that our setup compliments many recent results for sequential chaotic dynamical systems, see \cite{CR, DolgHaf PTRF 2} and references therein. Indeed, our results falls withing this framework of a sequential dynamical system.

From a ``Markovian" point of view our results extend the results in \cite{Nag61} to functions that depend on the entire path already in the stationary case and are not necessarily uniformly H\"older continuous. As noted before, it seems like this was not directly treated in literature even in the stationary case. 
In the non-stationary case our results extend \cite{DolgHaf PTRF 1} to functions that depend on the entire path of the chain and for more general chains which are not necessarily elliptic. Compared with the Markovian case in \cite{DolgHaf PTRF 2} where dependence on the entire path is allowed, we can consider more general chains (not necessarily finite state or elliptic) and functions $f_j=f_j(...,X_{j-1},X_j,X_{j+1},...)$ which are only H\"older on average in an appropriate sense. In fact, all that we need that $\sup_j\|f_j-\bbE[f_j|X_{j+k}; |k|\leq r]\|_{L^p}$ decays exponentially fast as $r\to\infty$ for appropriate $p$'s, which is much weaker than H\"older continuity on average. 

Finally, let us discus some other applications. For instance, we are able to provide optimal CLT rates in the Markovian case in the CLT for products of positive matrices in the CLT of Furstenberg and Kesten \cite{KestFurs61}. Our results also extend the results in \cite{HennionAoP97} for positive matrices from the iid case to Markovian non-stationary matrices, using a different approach. In a sense, our approach is closer to \cite{KestFurs61}, although philosophically it is also close in spirit to \cite{HennionAoP97} since both use projective metrics. Other examples in ergodic theory concern random Lyapunov exponent of Markov dependent hyperbolic matrices, see Section \ref{Lyp Sec}.
As noted before, our results can also be useful for practitioners in statistics and applied probability since we are able to capture more general iterative processes  that are generated by an inhomogeneous Markov chain (see Section \ref{Iter}). It seems that all the results in literature concern only the iid case (see \cite{Jirak} for the most general result for such applications).
As noted above, we strongly believe that working with iid driving systems is not realistic (both the independence and the stationary), and here we are able to consider non-independent and nonstationary processes.

\subsubsection{\textbf{Random dynamical systems}}
Here we focus our attention on our applications to Markov shifts in random dynamical environment (see Section \ref{LDP}). 
One can view this setup as a special case of a random dynamical system (RDS). 
RDS are motivated by real life phenomenon of random noise which make a given system non-stationary in nature. 
Ergodic theory of RDS has attracted a lot of attention in the past decades, see \cite{Arnold98, Cong97, Crauel2002, Kifer86, LiuQian95, KiferLiu}.
We refer to  the introduction of \cite[Chapter 5]{KiferLiu} for a historical discussion and applications to, for instance, statistical physics, economy and meteorology etc.
The literature on statistical properties (i.e. limit theorems) of random dynamical systems  exploded in recent years. Let us mention only a few results which are most relevant to our work. In \cite{Cogburn} central limit theorems were studied for Markov chains in random dynamical environment (as opposed to Markov shifts).
In \cite{Kifer1998} central limit theorems were studied for a variety of random dynamical systems, while in \cite{Kifer1996} large deviations were obtained.
In \cite{DavorNonlin, DavorCMP, DavorTAMS} central limit theorems, large deviations and almost sure invariance principle were obtained. Berry-Esseen theorems were obtained in \cite{DH1, HK, YH YT}. In the past two decades the number of  papers on the asymptotic behavior of random dynamical systems has exploded, and so again we will not make an attempt to even provide a full list. Our contribution to the theory of random dynamical systems is that we can consider functionals which depend on the entire path of the Markov chain (in the random environment), which includes applications to many other natural processes in random environment. 

\section{Preliminaries and main results}\label{Sec1}

\subsection{Mixing and approximation conditions}
Let $(X_j)_{j\in\bbZ}$ be a Markov chain defined on some probability space $(\Omega,\cF,\bbP)$. For all $-\infty\leq k\leq\ell\leq\infty$ let us denote by $\cF_{k,\ell}$ the $\sig$-algebra generated by $X_s$ for all finite $k\leq s\leq\ell$. 
Let $1\leq q\leq p\leq\infty$, and recall that (see \cite[Ch. 4]{Brad}) the (reverse) $\varpi_{q,p}$ mixing (weak dependence) coefficients associate with the chain are given by 
$$
\varpi_{q,p}(n):=\sup_{j}\varpi_{q,p}(\cF_{j+n,\infty},\cF_{-\infty,j})
 $$
 where for every sub-$\sigma$-algebras $\cG,\cH$ of $\cF$,
 $$
\varpi_{q,p}(\cG,\cH)=\sup\{\|\bbE[g|\cG]-\bbE[g]\|_{L^p}:g\in L^q(\cH): \|g\|_{L^q}\leq1\}. 
 $$
 Note that $\varpi_{q,p}$ is decreasing in $q$ and increasing in $p$ and that
\begin{equation}\label{mix1}
 \|\bbE[g(...,X_{j-1},X_{j})|X_{j+n},X_{j+n+1},...]-\bbE[g(...,X_{j-1},X_{j})]\|_{L^p}\leq \|g(...,X_{j-1},X_{j})\|_{L^q}\varpi_{q,p}(n)  
\end{equation}
for all $j,n,$ and measurable functions $g$ on $\prod_{k\leq j}\cX_k$.
 In what follows we will always work under the following assumptions, with appropriate $p$ and $q$.
 \begin{assumption}
For some $1\leq q, p\leq\infty$ we have   
\begin{equation}\label{mix}
\lim_{n\to\infty}\varpi_{q,p}(n)=0.   
\end{equation}
\end{assumption}
Note that in Theorem \ref{RPF} we will, in particular, show that $\varpi_{q,p}(n)$ decays exponentially fast under \eqref{mix} if $q\leq p$. Some of our results  will also require
\begin{assumption}
 There exist $1\leq p<q\leq\infty$, $c>0$ and $\gamma\in(0,1)$ such that for all $n\in\bbN$,   
\begin{equation}\label{mix2}
\varpi_{q,p}(n)\leq c\gamma^n.   
\end{equation} 
\end{assumption}
 \begin{remark}
Recall that (see \cite[Ch.4]{Brad}), the more familiar $\rho,\phi$ and $\psi$ mixing coefficients can be written as  
 $$
\varpi_{2,2}(\cG,\cH)=\rho(\cG,\cH),\,\,\varpi_{\infty,\infty}(\cG,\cH)=2\phi(\cG,\cH),
\,\,
\varpi_{1,\infty}(\cG,\cH)=\psi(\cG,\cH)
 $$
 where 
 $$
\rho(\cG,\cH)=\sup\left\{\text{corr}(g,h): g\in L^2(\cG), h\in L^2(\cH)\right\},
 $$
 $$
\phi(\cG,\cH)=\sup\left\{|\bbP(B|A)-\bbP(B)|: A\in\cG, B\in\cH, \bbP(A)>0\right\},
 $$
 and 
  $$
\psi(\cG,\cH)=\sup\left\{\left|\frac{\bbP(A\cap B)}{\bbP(A)\bbP(B)}-1\right|: A\in\cG, B\in\cH,\, \bbP(A)\bbP(B)>0\right\}.
 $$
 Note that both $\rho$ and $\psi$ are symmetric but $\phi$ is not. Set
 $$
 \rho(n)=\varpi_{2,2}(n),\phi_R(n)=\frac12\varpi_{\infty,\infty}(n)\,\text{ and }\,\psi(n)=\varpi_{1,\infty}(n).
 $$
 Thus when $q=p=2$ condition \eqref{mix} means that the chain is $\rho$-mixing while condition \eqref{mix} when $p=q=\infty$ means that the chain is reverse $\phi$-mixing (see \cite{BradMix}), while when condition \ref{mix} holds with $q=\infty$ and $p=1$ the chain is  $\psi$-mixing. Note that $\varpi_{q,p}(\cdot,\cdot)\leq\varpi_{\infty,1}(\cdot,\cdot)=\psi(\cdot,\cdot)$ and so this is the strongest type of mixing among the above. Moreover,  (see \cite{BradMix}),
$$
\rho(\cG,\cH)\leq 2\sqrt{\phi(\cG,\cH)}
$$
and so $\rho(n)\to 0$ if $\phi_R(n)\to 0$. 
\end{remark}

Let $\cX_j$ be the state space of $X_j$ and let $\cY_j=\cdots\cX_{j-1}\times\cX_{j}\times\cdots\cX_{j+1}\cdots$ be the infinite product. Of course, as sets all $\cY_j$ are identical, but for notational convenience we will keep the subscript $j$ and write $\cY_j=\{(x_{j+k})_{k\in\bbZ}: x_s\in\cX_s\}$. This will come in handy when presenting the approximation coefficients $v_{j,p,\delta}$ defined in 
Let $T_j:\cY_{j}\to\cY_{j+1}$ be the left shift defined below. Let $T_j(x)=(x_{j+k+1})_{k\in\bbZ}$ if $x=(x_{j+k})_{k\in\bbZ}$. Let us denote by $\mu_j$ the law of the random $\cY_j$ valued variable $(...,X_{j-1},X_j,X_{j+1},...)$. Then $(T_j)_*\mu_j=\mu_{j+1}$. Again, both $\mu_j$ and $T_j$ depend on $j$ only because of the different labeling of the coordinates in $\cY_j$.   For $n\in\bbN$ set
$$
T_j^n=T_{j+n-1}\circ\cdots\circ T_{j+1}\circ T_j:\cY_j\to \cY_{j+n}.
$$
Let us fix some $\delta\in(0,1)$, $b,a\geq 1$. Given a measurable function $g:\cY_j\to\bbC$ let 
$$
\|g\|_{j,a,b,\delta}=\|g\|_{L^a(\mu_j)}+v_{j,b,\delta}(g)
$$
where 
\begin{equation}\label{approx}
 v_{j,b,\delta}(g)=\sup_r \delta^{-r}\|g-\bbE[g|\cF_{j-r,j+r}]\|_{L^b(\mu_j)}.   
\end{equation}
Note that $\|\cdot\|_{j,a,b,\delta}$ is increasing in  both $a$ and $b$ and that
$$
\|g(...,X_{j-1},X_j,X_{j+1},...)-\bbE[g(...,X_{j-1},X_j,X_{j+1},...)|X_{j-r},...,X_{j+r}\|_{L^b}\leq v_{j,b,\delta}(g)\delta^r.
$$
Let us denote by $\cB_{j,a,b,\delta}$ the Banach space of all measurable functions $h:\cY_j\to\bbC$ such that $\|h\|_{j,a,b,\delta}<\infty$.
\begin{remark}\label{Rem Hold}
One particular instance that $v_{j,\infty,\delta}(g)<\infty$ is when all $\cX_j$ are metric spaces with metric $d_j$, normalized in size such that $\text{diam}(\cX_j)\leq1$ and $g:\cY_j\to\bbR$ is H\"older continuous with respect to the metric $\rho_j$ on $\cY_j$ given by 
\begin{equation}\label{rho metric}
\rho_j(x,y)=\sum_{k\in\bbZ}2^{-|k|}d_{j+k}(x_{j+k}, y_{j+k}),\,\, x=(x_{j+k}),\, y=(y_{j+k})  
\end{equation}
Here we take $\delta=2^{-\alpha}$ where $\alpha$ is the H\"older exponent of $g$. If $g$ is only H\"older continuous on average, that is 
$$
|g(x)-g(y)|\leq (C(x)+C(y))(\rho_j(x,y))^\alpha, \,\,C(\cdot)\in L^b(\mu_j)
$$
for some $\alpha\in(0,1]$ and $b>0$
then $v_{j,b,\delta}(g)<\infty$.  In both cases by the minimization property of conditional expectations we can just replace $\bbE[g|\cF_{j-r,j+r}]$ in the definition of $v_{j,b,\delta}$ by $g_j(c,X_{j-r},...,X_{j+r},d)$ for appropriate points $c\in\prod_{\ell<j-r}\cX_\ell$ and $d\in\prod_{\ell>j+r}\cX_{\ell}$.     
\end{remark}

\subsection*{Limit theorems}
 Let $f_j:\cY_j\to\bbR$ be measurable functions. Denote 
$$
S_nf=\sum_{k=0}^{n-1}f_k(...,X_{k-1},X_k,X_{k+1},...)=\sum_{k=0}^{n-1}f_k\circ T_0^k(...,X_{-1},X_0,X_1,...).
$$
Suppose that $f_j\in L^2(\mu_j)$ and let $\sigma_n=\sqrt{\text{Var}(S_nf)}$ and for all $t\in\bbR$ denote
$$
F_n(t)=\bbP(S_nf-\bbE[S_nf]\leq t\sig_n)=\bbP((S_nf-\bbE[S_nf])/\sig_n\leq t)
$$
where the second equality holds when $\sig_n>0$.
Let 
$$
\Phi(t)=\frac{1}{\sqrt {2\pi}}\int_{-\infty}^t e^{-\frac12x^2}dx
$$
be the standard normal distribution function.
Recall that the (self-normalized) central limit theorem (CLT) means that for every real $t$, 
$$
\lim_{n\to\infty}F_n(t)=\Phi(t).
$$
Our main results are optimal CLT rates  for the sequence of random variables $(S_nf)_{n=1}^\infty$ under appropriate mixing conditions and assumptions of the form $\sup_{j}\|f_j\|_{j,a,b,\delta}$ for appropriate $a,b$ and $\delta$. Note that in the generality of our setup even the CLT was not discussed before, an issue that will also be addressed in this paper. We will also prove some large deviations type results. The local CLT will be addressed in \cite{Work}.

\begin{remark}
The fact that the functions $f_j$ are allowed to depend on $j$ and on the entire path of the chain allows more flexibility than the classical situation where $f_j$ depends only on $X_j$. In Section \ref{App} we will provide many examples where this kind of dependence arises. For the meanwhile let us note that this setup includes certain sequence of random variables having the form $Z_n=F_n(X_0,...,X_{n-1},X_{n})$. Indeed, let us write 
$$
Z_n=\sum_{j=0}^{n}f_j(X_0,...,X_j)
$$
where $f_j(X_0,...,X_{j-1})=F_{j}(X_0,...,X_{j})-F_{j-1}(X_0,...,X_{j-1})$ and $F_{-1}:=0$.
Now the condition $\sup_{j\geq0}\|f_j\|_{j,a,b,\delta}$ holds if 
$$
\sup_{j\geq 0}\|F_j(X_0,...,X_{j})-F_{j-1}(X_0,...,X_{j-1})\|_{L^a}<\infty
$$
and 
$$
\sup_{j\geq 0}v_{j,b,\delta}(F_j)=\sup_j\sup_{r\leq j}\delta^{-r}\|F_j-\bbE[F_j|\cF_{j-r,j}]\|_{L^b}<\infty
$$
where we view $F_j$ as a function on $\cY_j$ which depends only on the coordinates $x_{j+k}$ for $-j\leq k\leq 0$. This means that our results apply when $F_j$ and $F_{j-1}$ are consistent in the sense that the are not too far in the $L^a$ norm and when $F_j$ depends weakly on the ``past" with exponentially decaying memory. In fact, several of our examples in Section \ref{App} fit this or a similar framework (i.e. the logarithms of products of random positive matrices, iterated random functions and linear processes), but to make the paper reader friendly we prefer to introduce these examples one by one.

Finally, remark than in Assumption \ref{Ass2} below we allow that $\|f_j\|_{j,a,s,\delta}=O((j+1)^\zeta), \zeta>0$ when  $\sig_n^2\geq c_1n^{2(b/a)(1+\zeta)+2\zeta+\varepsilon}$ for all $n$ large enough, where $b=\frac{a-3}{3a}$.
 Thus in the above context we get limit theorems for $Z_n$ when $\|F_j-F_{j-1}\|_{L^a}=O((j+1)^\zeta)$ and $v_{j,s,\delta}(F_j)=O((j+1)^\zeta)$.
\end{remark}

\subsection{Moments and mixing type assumptions needed for optimal CLT rates}

Our results concerning optimal CLT rates  will require that \textbf{one} of the following assumptions hold.

\begin{assumption}\label{Ass1}
 $f_j(...,X_{j-1},X_j,X_{j+1},...)$ depends only on $X_{j+k}, k\geq 0$.
Let $p,q\geq 1$ and $s\geq 2$ be such that  $\frac1p=\frac1{s}+\frac{1}q$ (so $q>p$). Suppose that \eqref{mix2} holds with these $q$ and $p$ and
there exists $\delta\in(0,1)$ such that 
$\sup_{j}\|f_j\|_{j,\infty,s,\delta}<\infty$. Moreover,
$\sig_n\to \infty$. Under this assumption we set $k=\infty$.
\end{assumption}

\begin{assumption}\label{Ass2}
Let $p,q\geq 1$ and $s\geq 3$ be such that  $\frac1p=\frac1{s}+\frac{1}q$ (so $q>p$).
Suppose that \eqref{mix2} holds with these $q$ and $p$. Let $a>s$. 
Suppose that
there exist $\delta\in(0,1)$ and $c_0,\zeta>0$, $a>k_0\geq 3$ such that 
$\|f_j\|_{j,a,s,\delta}\leq c_0(j+1)^{\zeta}, j\geq0$. Moreover,
there exist $c_1,n_0>0$ and $\varepsilon>0$ such that $\sig_n^2\geq c_1n^{(1+\zeta)\frac{2k_0}{a-k_0}+2\zeta+\varepsilon}$ for all $n\geq n_0$, where  when  $a=\infty$ we set $\frac{2k_0}{k_0-3}=0$. Under this assumption we set $k=k_0$.
\end{assumption}

\begin{remark}
Let  
$$
\varepsilon_0=\sup\left\{\varepsilon>0:\,\liminf_{n\to\infty}\frac{\sig_n^2}{n^{\varepsilon+2\zeta+\frac{2k_0(1+\zeta)}{a-k_0}}}>0\right\}.
$$
Then we can always take $\varepsilon<\varepsilon_0$ which is arbitrarily close to $\varepsilon_0$.
\end{remark}

The advantage of Assumption \ref{Ass2} compared with Assumption \ref{Ass1} is that it allows  $\|f_j\|_{j,a,s,\delta}$ to grow in $j$ and to depend on the past $X_k, k<j$, but the disadvantage is that it requires the variance to grow at least as fast as small power on $n$ (depending on $a$ and $\zeta$) and that the rates we obtain under Assumption \ref{Ass2} are of order $n^{\varepsilon/2-w}$ where $w$ is arbitrarily small.  Note that when $\zeta=0$ and $a=\infty$ then $n^{\ve_0}$ is essentially the growth rate of the variance and so in these circumstances and so we get arbitrarily close to optimal rates.

For stationary chains and a single function $f_j=f$, unless $\text{Var}(S_nf)$ is bounded the limit $\sigma^2=\lim_{n\to\infty}\frac{\sig_n^2}{n}$ exists and it is positive. The same holds true for Markov shifts in random dynamical environments (see Section \ref{LDP} and Theorem \ref{VarRDS}). Moreover, for small perturbations of stationary Markov chains $\sig_n^2$ grows linearly fast unless $\sig_n$ is bounded, see Section \ref{LinSec}. 
 Thus, in these circumstances Assumption \ref{Ass2}  allows us to consider functions $f_j$ such that $\|f_j\|_{j,\infty,s,\delta}=O((j+1)^{1/2-w}), w>0$ or when $p<a<\infty$,\, $\|f_j\|_{j,a,s,\delta}=O((j+1)^{\frac12(1-9/a)-w}), w>0$. When $\zeta=0$ and $a$ is large we get rates of order $O(n^{1/2-w})$ for $w=w(a)\to 0$ as $a\to\infty$.

Our next (optional) assumption requires the following notation. Given a finite set $B\subset \bbN_0:=\mathbb N\cup\{0\}$ we write 
$$
S_Bf=\sum_{j\in B}f_j(...,X_{j-1},X_j,X_{j+1},...).
$$
\begin{definition}
Let $A>1$. A variance partition of $\bbN_0$ corresponding to $A$ is 
 a partition $B_{1},B_2,...$ of $\bbN_0$ into intervals in the integers such that $B_j$ is to the left of $B_{j+1}$,\, $A\leq \text{Var}(S_{B_j}f)\leq 2A$
for all $j$ and $k_n=\max\{k: B_k\subset [0,n-1]\}$ satisfies
$A^{-1}\sig_n^2\leq k_n\leq A\sig_n^2$. A variance partition of $N_n=\{0,1,...,n-1\}$ is defined similarly.
\end{definition}

\begin{assumption}\label{Ass1.1}
Let $p,q\geq 2$ and $s\geq 2$ be such that  $\frac1p=\frac1{s}+\frac{1}q$ (so $q>p$). 
Let us assume that there exists $3\leq k_0<a<p/2$ such that $1/k_0=1/p+1/a$.
Suppose that \eqref{mix2} holds with these $q$ and $p$ and
there exists $\delta\in(0,1)$ such that 
$\sup_{j}\|f_j\|_{j,p,s,\delta}<\infty$. Moreover,
$\sig_n\to \infty$ and for all $n$ and $A$ large enough there exists a variance partition corresponding to $A$ such that $\sup_{\ell}\max_{B\subset B_\ell}\|S_{B}\|_{L^p}<\infty$ (where $B$ is a sub-interval whose left end point is the same as $B_\ell$).  Under this assumption we set $k=k_0$.
\end{assumption}

For unbounded functions without growth rates on the variance and without the uniform control over $\|S_{B_\ell}\|_{L^p}$ we will impose the following assumption.

\begin{assumption}\label{DomAss}
Suppose that \eqref{mix} holds with $p=\infty$ and some $1\leq q\leq p$.
Let $k\geq3$. There is a constant $C>0$ such that
\begin{equation}\label{1 cond}
\bbE[|f_j|^k|X_{j+1},X_{j+2},...]\leq C 
\end{equation}
almost surely and  there exist $\cF_{j-r,j+r}$ measurable functions $F_{j,r}$ such that
all $m\geq 0$ and $r\geq m$ we have,
\begin{equation}\label{2 cond}
\bbE[|f_{j+m}-F_{j+m,r}|^k|X_{j},X_{j+1},...]\leq C\delta^{rk}
\end{equation}
almost surely.
\end{assumption}
When $\sup_{j\geq 0}\|f_j\|_{j,\infty,\infty,\delta}<\infty$ then 
conditions \eqref{1 cond} and \eqref{2 cond}  hold  since we can take $F_{j,r}=\bbE[f_j|\cF_{j-r,j+r}]$. Note that by integrating the left hand sides of \eqref{1 cond} and \eqref{2 cond} and using the minimization property of conditional expectations the above assumption implies that $\sup_j\|f_j\|_{j,k,k,\delta}<\infty$. 
We refer to Appendix \ref{Discussion} for a detailed discussion on conditions \eqref{1 cond} and \eqref{2 cond} beyond the case  $\sup_{j\geq 0}\|f_j\|_{j,\infty,\infty,\delta}<\infty$. For instance, we can consider Markov chains satisfying the two sided Doeblin condition \eqref{Doeblin} which means that the laws of $X_{j+1}$ given $X_j$ are uniformly equivalent to the law of $X_j$, and functions $f_j$ with $\sup_{j}v_{j,\infty,\delta}(f_j)<\infty$ which satisfy an appropriate third moment condition. 


 The need in one of Assumptions \ref{Ass1}, \ref{Ass2}, \ref{Ass1.1} or \ref{DomAss}  is that each one of them guarantees that appropriate complex perturbations of the operators $\cL_j$ which map a function on $\prod_{k\geq j}\cX_k$ to a function on $\prod_{k\geq j+1}\cX_k$ given by 
 $$
\cL_j g(x)=\bbE\left[g(X_j,X_{j+1},...)|(X_{j+k})_{k\geq1}=x\right]
 $$
 are of class $C^k$ in the parameter that represents the perturbation when considered as a map between $\cB_{j,q,p,\delta}$ to $\cB_{j+1,q,p,\delta}$.

\subsubsection{A moment condition for block decompositions}
Recall that by \cite{Berry,Esseen} already for independent random variables $Y_k$ the optimal CLT rate $O(\sig_n^{-1})$ is known  when 
\begin{equation}\label{Comp}
L_{n,3}:=\sum_{j=0}^{n-1}\bbE[|Y_j-\bbE[Y_j]|^3]=O(\sig_n^2)    
\end{equation}
where $\sig_n^2=\text{Var}(Y_0+...+Y_{n-1})$.
When $\sig_n^2$ grows linearly fast this condition holds when $\sup_j\|Y_j\|_{L^3}<\infty$. The purpose of the following assumption is to address this type of comparison between the sum of the third absolute moments of the individual summands and the variance itself $\sig_n^2$. Like in the classical case, we will not need this assumption when $V:=\liminf_{n\to\infty}\frac{\sig_n^2}{n}>0$.

\begin{assumption}\label{Stand MomAss}
We have $\lim_{n\to\infty}\sig_n=\infty$. Let $k$ be like in one of Assumptions \ref{Ass1}, \ref{Ass2}, \ref{Ass1.1} or \ref{DomAss}, depending on the case. Let $k$ be like in one of Assumptions \ref{Ass1}, \ref{Ass2}, \ref{Ass1.1} or \ref{DomAss}. Suppose that
 either $V:=\liminf_{n\to\infty}\frac1n\text{Var}(S_n)>0$ or that for some finite $3\leq k_0\leq k$
for all $n$ and  $A>1$ large enough there exists a variance partition $(B_{\ell,n})$ of $\{0,1,...,n-1\}$ such that
\begin{equation}\label{LypAss}
\tilde L_{k_0,n}:=\sum_{j=1}^{k_n}\bbE\left[|S_{B_{j,n}f}-\bbE[S_{B_{j,n}f}]|^{k_0}\right]=O(\sig_n^2).
\end{equation}
\end{assumption}
When $k_0=3$ Assumption \ref{Stand MomAss} is very similar in spirit to \eqref{Comp}, except that we need to consider the sums of third moments along  blocks $B_j$. The reason is that our methods are based on a block partition argument which is crucial in overcoming the fact that $\sig_n^2$ can grow sublinearly fast. 
\begin{remark}
   In fact, by taking a closer look at the arguments in the proof of Theorems \ref{BE} and \ref{ThWass} and the proof of the main results in \cite{H} without Assumption \ref{Stand MomAss} we obtain CLT rates of order $\max(\sig_n^{-1}, \sigma_n^{-3}\tilde L_{3,n})$. However we are mostly interested in the case of optimal rates by means of $\sig_n$ and so the details are omitted. 
\end{remark}

We can verify Assumption \ref{Stand MomAss} in various situations. Before we present some sufficient conditions
let us recall that $f_j=f_j(...,X_{j-1},X_j,X_{j+1},...)$ is a reverse martingale difference (with respect to the reverse filtration $\cF_{j,\infty}$) if $f_j$ depends only on $X_{j+k},k\geq0$ and 
$\bbE[f_j(X_j,X_{j+1},...)|\cF_{j+1,\infty}]=0$ almost surely, for all $j\geq0$.  Recall also that $f_j$ is a forward martingale difference (with respect to the filtration $\cF_{-\infty,j}$) if it depends only on $X_{j+k}, k\leq0$ and  $\bbE[f_j(...,X_{j-1},X_j)|\cF_{-\infty,j-1}]=0$ almost surely.

\begin{remark}
It is not very hard to construct examples of reverse martingale differences in our setup. For instance, suppose $f_j(X_j,X_{j+1},...)=g_j(X_j)h_j(X_{j+1},X_{j+2},...)$ for some functions $g_j$ and $h_j$. 
The reversed martingale condition together with the Markov property means that $\bbE[g_j(X_j)|X_{j+1}]=0$, almost surely. For independent $X_j$ this only means that $\bbE[g_j(X_j)]=0$, while in general one can just replace $g_j$ with $g_j-\bbE[g_j|X_{j+1}]$.
Similarly, the forward martingale difference condition  holds when $f_j(...,X_{j-1},X_{j})=g_j(X_j)h_j(...,X_{j-2},X_{j-1})$ and $\bbE[g(X_j)|X_{j-1}]=0$.

Also, note that in the martingale case $\liminf_{n\to\infty}\frac{\sig_n^2}{n}>0$ if 
$\sum_{j=0}^{n-1}\bbE[f_j^2]\geq cn$ for some $c>0$ and all $n$ large enough. This is the case when $\inf_j\bbE[f_j^2]>0$. In these circumstances,  Assumption \ref{Stand MomAss} holds.
\end{remark}

We can verify Assumption \ref{Stand MomAss} in the following circumstances.
\begin{proposition}\label{MomMom}
Suppose $\lim_{n\to\infty}\sig_n=\infty$. Then Assumption \ref{Stand MomAss} holds in the following cases:
\vskip0.1cm 
(1) if \eqref{mix} holds with $p=\infty$ and some $1\leq q\leq p$ and $\sup_j\|f_j\|_{j,\infty,\infty,\delta}<\infty$ for some $\delta\in(0,1)$ the Assumption \ref{Stand MomAss} holds with every finite $k_0$.
 \vskip0.1cm
 (2) Assumption \ref{Stand MomAss} holds with $k_0=4$ if $f_j$ is a reversed martingale difference, 
  \eqref{mix} holds with some $1\leq q\leq p$, $\sup_j\|f_{j}^2\|_{j,q,p,\delta}<\infty$,
  and there exists a constant $C>0$ such that 
 \begin{equation}\label{Lyp Cond}
\sum_{j=0}^{n-1}(\bbE[G_\ell^2]+\|G_\ell\|_{L^u})\leq C\sig_n^2.     
 \end{equation}
 where $G_\ell=f_\ell^2-\bbE[f_\ell^2]$ and $u$ is the conjugate exponent of $p$.
 \vskip0.1cm
 (3) Assumption \ref{Stand MomAss} holds with $k_0=4$ if $f_j$ is a forward martingale difference, 
 \eqref{mix} holds with some $1\leq p\leq q$, $\sup_j\|f_{j}^2\|_{j,p,q,\delta}<\infty$,
 and \eqref{Lyp Cond} holds with $u$ being the conjugate exponent of $q$.
 \vskip0.1cm
 (4) Assumption \ref{Stand MomAss} holds with $k_0=4$ if $\mu_j(f_j)=0$ for all $j$,  and $\varpi_{p,q}(n)\to 0$ for some conjugate exponents $q,p$ with $p\leq 2$
 and for some $\delta\in(0,1)$ we have
$$
\sup_\ell(\|f_\ell\|_{L^4}^2+v_{\ell,2,\delta}(f_\ell^2)+\|f_\ell\|_{L^1}+v_{\ell,q,\delta}(f_\ell))<\infty.
$$
For a fixed $n$ let  $B_1,...,B_{k_n}$ be a block partition of $N_n$ like in Assumption \ref{Stand MomAss} except that \eqref{LypAss} is not assumed to hold.
Set 
$$
U_j=U_{j,n}=\sum_{\ell\in B_j}\bbE[f_\ell^4]+\left(\sum_{\ell\in B_j}\bbE[f_\ell^2]\right)^2+
\sum_{\ell\in B_j}\left(
v_{\ell,2,\delta}(f_\ell^2)+\|f_\ell\|_{L^{3p}}^3+v_{\ell,q,\delta}(f_\ell^3)\right).
$$
Let $V_j=\min(U_j,U_j^{3/4})$. Then Assumption \ref{Stand MomAss} (i.e. \eqref{LypAss}) holds if 
\begin{equation}\label{Lyp Cond 2}
\sum_{j=1}^{k_n}V_j=O(\sig_n^2).     
 \end{equation}
\end{proposition}
 Proposition \ref{MomMom} follows from Proposition \ref{Mom prop} which deals with high order moments. 
 The proof of Proposition \ref{Mom prop} appears in Sections \ref{MomSec1} and \ref{MomSec2}.

\subsection*{A discussion about the conditions of Proposition \ref{MomMom}}
 Condition \eqref{Lyp Cond 2}  is similar in spirit to \eqref{Lyp Cond} but it also involves approximation coefficients $v_{k,a,\delta}, a=2,q$ of appropriate powers of $f_k$.
Note that by taking $p=\infty$ (or $q=\infty$) in conditions (2) (or (3)) we have $u=1$ and then $\|G_\ell\|_{L^u}=\bbE[|G_\ell|]\leq 2\bbE[f_\ell^2]$. In that case condition \eqref{Lyp Cond} is equivalent to $\sum_{j=0}^{n-1}\bbE[G_\ell^2]=O(\sig_n^2)$    
which holds when 
\begin{equation}\label{Ly2}
\sum_{j=0}^{n-1}\bbE[f_\ell^4]=O(\sig_n^2)   
\end{equation}
which is very similar to \eqref{Comp}, replacing the third moment by the fourth.

 When $p<\infty$ (or $q<\infty$) its conjugate exponent $u$ is larger than $1$. Taking for instance $p==u=q=2$ we see that for $\rho$-mixing Markov chains in the martingale case  condition \eqref{Lyp Cond} holds when 
 $
\sum_{\ell=0}^{n-1}\|f_\ell\|_{L^{4}}^2=O(\sig_n^2)
 $
 which in some sense is also somewhat close in spirit to \eqref{Comp}, and it holds when 
 $\|f_\ell\|_{L^{4}}\leq C\|f_\ell\|_{L^2}$. Condition \ref{Lyp Cond 2} also shares resemblance with \eqref{Comp}. For instance,  when restricting to that case when $f_k(...,X_{k-1},X_k,X_{k+1},...)$ depends only on $X_{k+m}$ for $|m|\leq m_0$ for some $m_0$ and all $k$ then we can just omit the approximation coefficients $v_{k,2,\delta}$ and under the`` martingale like" condition 
 \begin{equation}\label{Mart Like}
  \sum_{k=j}^{j+n-1}\text{Var}(f_k)\leq C\left(1+\text{Var}\left(\sum_{k=j}^{j+n-1}f_k\right)\right)\,\,\text{for all }j,n   
 \end{equation} 
 condition \eqref{Lyp Cond} holds when 
 $$
\sum_{\ell=0}^{n-1}(\bbE[f_\ell^4]+\|f_\ell\|_{L^{3p}}^3)=O(\sig_n^2)
 $$
 where we recall that $\mu_j(f_j)=0$. When taking $p\leq 4/3$ this condition reduces to $\sum_{\ell=0}^{n-1}\max(\|f_\ell\|_{L^4}^4,\|f_\ell\|_{L^4}^3)=O(\sig_n^2)$.
 Under \eqref{Mart Like} the above condition holds, for instance, when 
 $\bbE[f_\ell^4]+\|f_\ell\|_{L^{3p}}^3\leq C\bbE[f_\ell^2]$ or $\max(\|f_\ell\|_{L^4}^4,\|f_\ell\|_{L^4}^3)\leq C\bbE[f_\ell^2]$ when $p\leq 4/3$.
 This  can happen  when $\sig_n^2=o(n)$ since sublinear growth very often comes from decay of $f_\ell$ to $0$ as $\ell\to\infty$ at an appropriate rate and in an appropriate sense. 

\subsection{The growth of the variance, Livsic regularity and the CLT}

In general, in order for the CLT to hold we need the individual summands to of smaller order than the variance.   In particular, we need to know when the variance is bounded. Let us begin with a characterization of this boundedness.
\begin{theorem}\label{Var them}
Suppose that either \eqref{mix} holds with some $p\geq q\geq 2$ and $\sup_j\|f_j\|_{j,q,p,\delta}<\infty$, or \eqref{mix} holds with some $q\geq p\geq 1$ and $\sup_j\|f_j\|_{j,\infty,\infty,\delta}<\infty$.  In the first case set $b=p$ and in the second case set $b=\infty$.
The following conditions are equivalent.
\vskip0.2cm
(1) $\liminf_{n\to\infty}\text{Var}(S_nf)<\infty$;
\vskip0.2cm
(2) $\sup_{n\in\bbN}\text{Var}(S_n f)<\infty$;

\vskip0.2cm
(3) we can write 
$$
f_j=\bbE[f_j(...,X_{j-1},X_j,X_{j+1},...)]+M_j+u_{j+1}\circ T_j-u_j,\,\ \mu_j-\text{a.s.}
$$
where $\sup_j\|u_j\|_{j,s,p,\delta^{1/2}}<\infty,$\;
$\sup_j\|M_j\|_{j,s,p,\delta^{1/2}}\!\!<\!\!\infty$ for all finite $s\leq b$, $M_j$ depends only on the coordinates $X_k,k\geq j$, $u_j$ and $M_j$ have zero mean and $M_j(X_j,X_{j+1},...),j\geq 0$ is a reverse martingale difference  with respect to the reverse filtration $\cG_j=\cF_{j,\infty}$ and\footnote{Note that by the martingale converges theorem we get that the sum $\sum_{k=0}^\infty M_k(X_k,X_{k+1},...)$ converges almost surely and in $L^s$.} 
$$
\sum_{j\geq 0}\text{Var}(M_j(X_j,X_{j+1},...))\!\!<\!\!\infty.
$$
Moreover, if 
\eqref{mix} holds with  $p=q=\infty$ and $\sup_j\|f_j\|_{j,\infty,\infty,\delta}<\infty$ then $\sup_j\|u_j\|_{j,\infty,\infty,\delta^{1/2}}<\infty$ and
$\sup_j\|M_j\|_{j,\infty,\infty,\delta^{1/2}}\!\!<\!\!\infty$. If $f_j(...,X_{j-1},X_j,X_{j+1},..)$ depends only on $X_j, k\geq j$ then $\delta^{1/2}$ above can be replaced by $\delta$.

\vskip0.1cm
If also  one of Assumptions \ref{Ass1}, \ref{Ass2}, \ref{Ass1.1} or \ref{DomAss} hold (expect for the variance growth) 
then all the above conditions are equivalent to the following condition: there exist measurable functions $H_j:\cY_j\to\bbR$ such that 
$$
f_j=H_{j+1}\circ T_j-H_j,\,\mu_j\text{ a.s.}
$$
In case all the above conditions hold we must have $H_j\in L^s(\mu_j)$ for all finite $s\leq b$ and, in fact, $H_j=\mu_j(H_j)+u_j+\sum_{k\geq j}M_k\circ T_j^{k-j}$.
\end{theorem}
Note that we can just take $p=q=2$ in the above theorem, which shows that for $\rho$ mixing Markov chains we get the result for square integrable functions. However, considering larger $p$'s shows that the same level of regularity is preserved in the martingale coboundary decomposition in condition (3). Note that the last part of Theorem \ref{Var them} is an appropriate version of the, so called, Livsic theory (see \cite{Livsic}) for non-stationary Markov shifts. 
\begin{remark}
Condition in (3) in Theorem \ref{Var them} can also be written as 
$$
f_j=G_{j+1}\circ T_j-G_j,\,\, G_j=\sum_{k=0}^{j-1}\mu_k(f_k)+u_j+\sum_{k\geq j}M_k\circ T_j^{k-j}.
$$
However, in general it is not true that $\sup_{j}\|G_j\|_{j,a,p,\delta}<\infty$, see \cite{BDH} for examples in the case when $X_j$ are iid (using that the dynamics of the doubling map $Tx=2x\text{ mod }1$ is coded by iid Bernoulli shift on $\{0,1\}^\bbN$). Note that (see \cite{BDH}) when $(X_j)$ is stationary and $f_j=f$ do not depend on $j$ or for Markov chains in random dynamical environment discussed in Section \ref{LDP} we can ensure that $M_j=0$ for all $j$, which in this case yields that $\sup_j\|H_j-\mu_j(H_j)\|_{j,q,p,\delta}<\infty$, namely that it has the same level of regularity as $f_j$.    
\end{remark}

Next we address the CLT. First, we describe what is essentially known in literature in our setup.
The following result follows by the discussion in \cite[Section 7.2]{HafSPA}.
\begin{theorem}
$S_nf$ obeys the CLT if \eqref{mix} holds with some $p\geq q>2$, $\sup_j\|f_j\|_{j,q,p,\delta}<\infty$ and  $\sigma_n^2\geq c\ln^{1+\varepsilon}(n)$ for some $c,\varepsilon>0$ and all $n$ large enough.
\end{theorem}
 The idea is standard. We can approximate $f_j=f_j(...,X_{j-1},X_j,X_{j+1},...)$ by $f_{n,r(n)}=\bbE[f_j|X_{j-r(n)},...,X_{j+r(n)}]$ with $r(n)=C\ln n$ for $C$ large enough. Then, for instance, one can apply Stein's method, use standard forwrad martingale approximation or use Bernstein's big block small block approach, see \cite[Section 7.2]{HafSPA}. However, without growth rates on the variance such an approximation procedure seems to fail even for independent $X_j$ since then the dependency range is of logarithmic order in $n$, while the variance might be of smaller magnitude. Without growth assumptions on $\sig_n$ beyond $\sig_n\to\infty$  we can prove the following CLT.
 \begin{theorem}\label{CLT}
(i) In the circumstances of Proposition \ref{MomMom} (iv), the CLT holds when the LHS of \eqref{Lyp Cond 2} is of order $o(\sig_n^4)$.
 \vskip0.1cm
 
 (ii)  Let \eqref{mix} hold with some $1\leq q, p\leq\infty$. Suppose  $\sig_n\to\infty$. Then  Then $S_nf$ obeys the CLT if \textbf{all} of the following conditions hold.
  \vskip0.2cm
 (1) Either $p\geq q$ and $f_j(...,X_{j-1},X_{j},X_{j+1},...)$ depends only  $X_{j+k},k\geq 0$ and  it is a reverse martingale difference with respect to the reverse filtration $\cF_{j,\infty}$ or $q\geq p$, it depends only on $X_{j+k},k\geq 0$ and it is a forward martingale difference with respect to the filtration $\cF_{-\infty,j}$.
  \vskip0.2cm
(2) $\sup_j\|f_{j}^2\|_{j,q,p,\delta}<\infty$ (with $b=\max(q,p)$ this holds true when $\sup_j\|f_{j}^2\|_{j,2b,2b,\delta}<\infty$).
  \vskip0.2cm
 (3)
 $(f_j)$ satisfies  the Lindeberg condition, that is, for every $\varepsilon>0$ we have 
 \begin{equation}\label{Lind}
\lim_{n\to\infty}\sig_n^{-2}\sum_{j=0}^{n-1}\bbE[|f_j|^2\bbI(|f_j|\geq \ve\sig_n)]=0  
 \end{equation}
 where $\bbI(A)$ denotes the indicator function of en event $A$. 
 \vskip0.2cm
 (4) In the reversed martingale case let $u$ be the conjugate exponent of $p$ while in the forward martingale case let $u$ be the conjugate exponent of $q$. 
 With $G_j=f_j^2-\mu_j(f_j^2)$ we have
 \begin{equation}\label{Special condition}
 \lim_{n\to\infty}\sigma_n^{-4}\sum_{j=0}^{n-1}(\bbE[G_j^2]+\|G_j\|_{L^u})=0.    
 \end{equation}
 \end{theorem}

Recall that for independent summands $f_j$ the Lindeberg condition \eqref{Lind} is equivalent to the CLT. In our case, when $\max(p,q)=\infty$ (so $u=1$) note that since 
$$
\sum_{j=0}^{n-1}\|Q_j\|_{L^1}\leq 2\sum_{j=0}^{n-1}\bbE[f_j^2]=2\sigma_n^2
$$
condition \eqref{Special condition} means that $\sum_{j=0}^{n-1}\bbE[Q_j^2]=o(\sigma_n^4)$. This condition holds when $\bbE[|f_j|^4]\leq \varepsilon_n\sigma_n^2\bbE[|f_j|^2]$ for $\varepsilon_n\to0$, and in particular when $\bbE[|f_j|^4]\leq C\bbE[|f_j|^2]$ for some constant $C>0$. When $p<\infty$ (or $q<\infty$), $u>1$
 and then condition \eqref{Special condition} holds when also $\sum_{j=0}^{n-1}\|f_j\|_{L^{2u}}^{2}=o(\sigma_n^4)$ which is the case when  $\|f_j\|_{L^{2u}}\leq \varepsilon_n\sigma_n\|f_j\|_{L^2}$ with $\varepsilon_n\to 0$. 
 \begin{remark}\label{CLT Rem}
 Since this paper is more focused on  CLT rates we did not try to optimize the conditions for the CLT under which our methods work. For instance, some growth rates in $j$ of either $\sup_j\|f_j\|_{j,\infty,\infty,\delta}$ or
 $\sup_j\|f_j^2\|_{j,q,p,\delta}$ may be allowed. Additionally, assumptions like $\|f_{j}\|_{j,q,u,\delta}\leq C\|f_k\|_{k,q,u,\delta}$ for $k\leq j$ and similar ones should yield the CLT without the martingale difference condition. The idea is that by combining Lemmata  \ref{Sinai} and \ref{Mart lemm} $f_j$ is cohomologous to a martingale difference and then the conditions of Theorem \ref{CLT} (ii) should be checked for these martingales, and that under such assumptions the martingale difference $M_j$ satisfies $\|M_j\|_{L^a}\leq C\|f_j\|_{L^a}$ for appropriate $a$'s. 
 \end{remark}

\subsection{Optimal CLT rates and moment estimates}
In this section we will state our results concerning optimal CLT rates (aka Berry Esseen theorems).

\begin{theorem}\label{BE}
Let \textbf{one } of Assumptions \ref{Ass1}, \ref{Ass2}, \ref{Ass1.1} or \ref{DomAss} be in force and let $k,k_0$ be as described in the assumptions. Under one of Assumption \ref{Ass1} and \ref{DomAss} let $u=0$. While under Assumption \ref{Ass2} let $u=1-\varepsilon(1-\varepsilon)/2$ and under Assumption \ref{Ass1.1} let $u=2a/p<1$.
Suppose also that Assumption \ref{Stand MomAss} holds.

\vskip0.2cm
(i)  for all finite $0\leq s\leq k_0-1$ there is a constant $C_s$ such that 
$$
\sup_{t\in\bbR}(1+|t|^s)\left|F_n(t)-\Phi(t)\right|\leq C_s\sig_n^{-(1-u)}.
$$ 
\vskip0.1cm
(ii) for all $q>\frac1{k_0}$ we have 
$
\left\|F_n-\Phi\right\|_{L^q(dx)}=O(\sig_n^{-(1-u)}).
$
\vskip0.1cm
(iii)  for all finite $1\leq s\leq k_0-1$  there is a constant $C_s$ such that
for every absolutely continuous function $h:\bbR\to\bbR$  such that 
$H_s(h):=\int \frac{|h'(x)|}{1+|x|^s}dx<\infty$ we have 
$$
\left|\bbE[h((S_n-\bbE[S_n])/\sig_n)]-\int h d\Phi\right|\leq C_s H_s(h)\sig_n^{-(1-u)}.
$$
\end{theorem}
One example of functions in (iii) above are $h(x)=x^{a}, a<s$. Then Theorem \ref{BE} (iii) provides estimates form the moments of $S_nf-\bbE[S_nf]$ by means of the variance of $S_n$ and the standard normal moments.

Under Assumption \ref{DomAss} the rates in Theorem \ref{BE} (i), say with $s=0$, are consistent with the classical Berry Esseen theorem for stationary  Markov chains with some ellipticity and functions of the form $f_j=f(X_j)$ that state the when $\|f(X_1)\|_{L^3}<\infty$ then the optimal rate $O(n^{-1/2})$ is achieved (see \cite{Nag61}). Note that in this setting ellipticity ensures \eqref{1 cond} while \eqref{2 cond} trivially holds true since $f(X_j)$ depends only on $X_j$.
Note that (see \ref{Comp}) for independent (zero mean) summands $Y_j$ the optimal CLT rate $O(\sig_n^{-1})$ are is achieved  when $\sup_j\|Y_j\|_{L^\infty}<\infty$, which is consistent with Assumption \ref{Ass1} . 

Note that for uniformly elliptic inhomogenuous Markov chains $X_j$ the optimal rates $O(\sig_n^{-1})$ were  achieved with $Y_j=f_j(X_j,X_{j+1})$ also  for uniformly bounded functions $f_j$, see \cite[Theorem]{DolgHaf PTRF 2}. This is also consistent with Assumption \ref{Ass1}, but however here we can consider functions that depend on the entire path of the chain, and we can consider more general chains which are not necessarily elliptic. Remark also that in \cite{DolgHaf PTRF 2} when restricting the results to Markov chains we obtained optimal CLT rates for uniformly elliptic \textbf{finite state} Markov chains $X_j$ and uniformly H\"older continuous functions $f_j=f_j(...,X_{j-1},X_j,X_{j+1},...)$, which always satisfy $\sup_j\|f_j\|_{j,\infty,\infty,\delta}<\infty$ for some $\delta$. Compared with these results we are able to consider much more general chains without ellipticity conditions, and, under suitable conditions, unbounded functions $f_j$ or functions exhibiting some growth in $j$. In fact, the results in \cite{DolgHaf PTRF 2} were mostly about uniformly expanding or hyperbolic systems, and in Section \ref{DSY} we will apply our results for some classes of non-uniformly expanding or hyperbolic maps and get optimal rates for H\"older on average functions $f_j$.

Next, recall that the $p$-th Wasserstien distance between two probability measures $\mu,\nu$ on $\bbR$ with finite absolute moments of order $b$ is given by 
$$
W_p(\mu,\nu)=\inf_{(X,Y)\in\mathcal C(\mu,\nu)}\|X-Y\|_{L^b}
$$
where $\mathcal C(\mu,\nu)$ is the class of all  pairs of random variables $(X,Y)$ on $\bbR^2$ such that $X$ is distributed according to $\mu$, and $Y$ is distributed according to $\nu$. 

\begin{theorem}\label{ThWass}
Let \textbf{one}  of Assumptions \ref{Ass1}, \ref{Ass2}, \ref{Ass1.1} \ref{DomAss} be in force and let $k,k_0$ be as described at the assumptions. Suppose also that Assumption \ref{Stand MomAss} holds. 
Then, for every finite $b<k_0-1$ we have 
$$
W_b(dF_n, d\Phi)=O(\sig_n^{-(1-u)})
$$
where $dG$ is the measure induced by a distribution function $G$ and $u$ is like in Theorem \ref{BE}.
\end{theorem}


Next, set
$$
S_{j,n}f=\sum_{k=j}^{j+n-1}f_k\circ T_j^{k-j}=\sum_{k=j}^{j+n-1}f_k(...,X_{k-1},X_k, X_{k+1},...).
$$
A key ingredient in the  proof of Theorem \ref{CLT} and Theorems  \ref{BE} and \ref{ThWass}  are the second part and third parts of following proposition, which we believe has its own interest.
\begin{proposition}\label{Mom prop}
(i) Let \eqref{mix} hold with some $p\geq q\geq1$ and suppose $\sup_{j}\|f_j\|_{j,q,p,\delta}<\infty$. Then
for $j\geq0$ and 
$n\!\in\!\bbN$,
$$
\|S_{j,n}f-\bbE[S_{j,n}f]\|_{L^q}\leq C_q\sqrt{n}
$$
for some constant $C_q>0$.
\vskip0.1cm
(ii) Let \eqref{mix} hold with $p=\infty$ and some $1\leq q\leq\infty$. Assume that
$\sup_{j}\|f_j\|_{j,\infty,\infty,\delta}<\infty$.  
Then for every $2\leq b<\infty$ there is a  constant  $C_b$ such that for all $j\geq 0$ and 
$n\!\in\!\bbN$,
$$
\|S_{j,n}f-\bbE[S_{j,n}f]\|_{L^b}\leq C_b\left(1+\sqrt{\text{Var}(S_{j,n}f)}\right).
$$
\vskip0.1cm
(iii) In the (forward or reversed) martingale case and under either \eqref{mix} or \eqref{mix2} 
 there is a constant $C$ such that for all $j\geq 0$ and $n\in\bbN$,
\begin{equation}\label{4 Bound}
 \|S_{j,n}f\|_{L^4}\leq C\left(1+\|S_{j,n}f\|_{L^2}+\left(\beta_{j,n}\right)^{1/4}\right)  \end{equation}
where with $G_\ell=f_\ell^2-\bbE[f_\ell^2]$ and 
$$
\beta_{j,n}=\sum_{\ell=j}^{j+n-1}(\bbE[G_\ell^2]+\|G_\ell\|_{L^u})
$$
and $u$ is the conjugate exponent of $\max(q,p)$.
\vskip0.1cm
(iv)
 Under either \eqref{mix} or \eqref{mix2}  
for all $j$ and $n$ and $\delta\in(0,1)$ we have 
$$
\bbE\left[(S_{j,n}f)^4\right]\leq \sum_{k=j}^{j+n-1}\bbE[f_k^4]+C\left(\sum_{k=j}^{j+n-1}\bbE[f_k^2]\right)^2+C\sum_{k=j}^{j+n-1}\left(\|f_\ell\|_{L^4}^2+v_{k,2,\delta}(f_k^2)+\|f_k\|_{L^{3p}}^3+v_{k,q,\delta}(f_k^3)\right)
$$
where $C=C(R,\delta,f)=2(R+(1-\delta^{1/4})^{-1})\max_{j\leq k\leq j+n-1}(\|f_k\|_{L^4}^2+v_{k,2,\delta}(f_k^2)+\|f_k\|_{L^1}+v_{k,q,\delta}(f_k))$.
\end{proposition}
Note then when $\max(q,p)=\infty$ then $u=1$ and then $\beta_{j,n}\leq C(1+\|S_{j,n}f\|_{L^2}^2)$ which is consistent with part (ii).
Like in Remark \ref{CLT Rem} we can prove a version of part (iii) without the martingale condition. However, we would have to replace $G_\ell$ by $M_\ell^2-\bbE[M_\ell^2]$, where $M_\ell$ is a reverse martingale difference which is cohomoologous to $f_\ell$, see Lemmata \ref{Sinai} and \ref{Mart lemm}. The proof of Proposition \ref{Mom prop} reveals that we can also get a version of (iii) with higher moments by applying successively an appropriate version of the Burkholder inequality \eqref{Burk}.  However, without some assumptions on the norms of $f_j$ the upper bounds we get are more complicated and they involve powers of expressions of the form $\sum_{\ell=j}^{n-1}\sum_{k\geq 0}\gamma^k\|f_{\ell-k}\|_{\ell-k,a,d,\delta}$ for appropriate $a,d$ and $0<\gamma<1$. Like in Remark \ref{CLT Rem} we believe that such expressions can be controlled under assumptions of the form $\|f_{j}\|_{j,a,d,\delta}\leq C\|f_k\|_{k,a,d,\delta}$ for $k\leq j$, but in order not to overload the paper we decided not to formulate such results.
Part (iv) is elementary but it allows to avoid using martingales at the expense of adding approximation coefficients to the upper bounds.

\begin{remark}
Using the main results in \cite{DavorYeor ASIP} and a block partition argument  we can get rates of order $O(\sigma_n^{1/2+\varepsilon})$ in the almost sure invariance principle (ASIP), that is we can couple $S_nf$ with a Brownian motion $B(t)$ such that $|S_nf-\bbE[S_nf]-B(\sigma_n^2)|=O(\sigma_n^{1/2+\varepsilon})$, almost surely. This implies the functional central limit theorem and the law of iterated logarithm, for example (see \cite{PS}). It also implies other limit laws like the Arcsine Law and the law of records, see \cite[Appendix C]{Sarig25}.
However, the rates  $O(\sigma_n^{1/2+\varepsilon})$ are suboptimal compared with the rates $O(\ln n)$ established recently for various stationary processes (see \cite{Korepanov Merlevede}). In our setup the optimal rates should be  $O(\ln\sigma_n)$, which seems  to require a different approach. Because of these reasons we decided only to remark on the above ASIP rates, and to address the problem of getting optimal rates elsewhere. 
\end{remark}

\subsection{Large and moderate deviations}
\subsubsection{Moderate deviations}
We begin with the following moderate deviations principle  with optimal scale.
\begin{theorem}\label{MDP}
Let \eqref{mix} hold with $p=\infty$ and some $1\leq q\leq p$. 
Assume that $\sup_j\|f_j\|_{j,\infty,\infty,\del}<\infty$, that
 $\mu_j(f_j)=0$ for all $j$ and that $\sig_n^2\geq cn$ for some $c>0$ and all $n$ large enough.
 Let $(a_n)$ be a sequence such that $a_n\to\infty$ such that $\lim_{n\to\infty}\frac{a_n}{\sqrt n}=\infty$
 but $a_n=o(n)$. Denote $s_n=a_n^2/n$.
Then for every Borel measurable set $\Gamma\subset\bbR$
$$
-\frac12\inf_{x\in \Gamma^o}x^2\leq \liminf_{n\to\infty}\frac{1}{s_n}\ln\bbP((S_{n}f/a_n)\in\Gamma)\leq \limsup_{n\to\infty}\frac{1}{s_n}\ln\bbP((S_{n}f/a_n)\in\Gamma)\leq -\frac12\inf_{x\in \overline{\Gamma}}x^2
$$
where $\Gamma^o$ is the interior of $\Gamma$ and $\overline{\Gamma}$ is it's closure. 
\end{theorem}

\begin{remark}
Other moderate type results can be proved using our methods. For instance using the martingale coboundary representation in Lemma \ref{Mart lemm} we can prove exponential concentration inequalities without growth assumptions on the variance, and using the method of cumulants and multiple correlation estimates  we can derive some moderate deviations type results when $\sig_n^2$ grows sub-linearly fast in $n$ (but faster than $n^\varepsilon$ for some $\varepsilon>0$) and other types of concentration inequalities.
However, in order not to overload the paper these results will not be formulated. Moreover, these result still seem to require some growth rates for $\sig_n^2$, and getting any type of  large deviations under the sole assumption that $\sigma_n\to\infty$ seem to require a different approach.
\end{remark}

\subsubsection{\textbf{Large deviations principles for Markov chains in random dynamical environments}}\label{LDP}
Let $(M,\cB,\bbP_0,\te)$ be an invertible ergodic probability preserving system. Let $\cD$ be a measurable space and let $\cX\subset M\times\cD$ be a measurable set such that its fibers
$\cX_\om=\{x\in\cD: (\om,x)\in\cX\}, \om\in M$ are measurable in $\om$.  For instance we can take $\cD$ to be a metric space and $\cX_\om$ to be random closed sets. Let $Q_\om(x,dy), x\in\cX_\om$  be a measurable collection of transition probabilities on $\cX_{\te\om}$. Let us define a Markov chain $(X_{\om,k})_{k\in\bbZ}$ with state spaces $\cX_{\te^k\om}$ by 
$$
\bbP(X_{\om,k+1}\in\Gamma|X_{\om,k}=x)=Q_{\te^k\om}(x,\Gamma).
$$
Define $\cY_{\om}$ to be the product $\prod_{k\in\bbZ}\cX_{\te^{k}\om}$. Let $T_\om:\cY_\om\to\cY_{\te\om}$ to be the left shift and denote by $\mu_\om$ the probability measure that $(X_{\om,k})_{k\in\bbZ}$  induces on $\cY_\om$. Denote 
$$
\varpi_{q,p}(n)=\text{ess-sup}_{\om\in\Om}\left(\sup_{j}\varpi_{q,p}(\cF_{-\infty,j,\om},\cF_{j+n,\infty,\om})\right).
$$
where $\cF_{k,\ell,\om}$ is the $\sig$-algebra generated by $X_{\om,m}$ for all finite $k\leq m\leq \ell$. In what follows we will always assume that $\varpi_{q,p}(n)\to$ for some $p,q$.
Let us take a measurable in $\om$ family of functions $f_\om:\cY_{\om}\to\bbR$ and let us consider random variables of the form 
$$
S_n^\om f=\sum_{j=0}^{k-1}f_{\te^j\om}\circ T_\om^j=\sum_{j=0}^{n-1}f_{\te^j\om}(...,X_{j-1,\om},X_{j,\om},X_{j+1,\om},...)
$$
where $T_\om^j=T_{\te^{j-1}\om}\cdots \circ T_{\te\om}\circ T_\om$. To address measurability of $f_\om$ with respect to $\om$ we may view $f_\om(x)$ as a restriction of a function $f(\om,x)$ on $M\times\cD^\bbZ$.

Let $\|\cdot\|_{\om,q,p,\delta}$ be the norm defined by
$$
\|f_\om\|_{\om,q,p,\delta}=\|f_\om\|_{L^p(\mu_\om)}+\sup_{r\geq 0}\delta^{-r}\left\|f_\om-\bbE[f_\om| X_{k,\om};|k|\leq r]\right\|_{L^q(\mu_\om)}.
$$
In Section \ref{RDS} we will prove the following result.
\begin{theorem}\label{VarRDS}
Suppose that   $\varpi_{q,p}(n)\to0$ for some $p\geq q\geq 2$ or that \eqref{mix2} holds for some $q>p\geq 2$. Moreover, let us assume that for some $d>2$ and $\delta>0$ we have
 $\om\to \|f_\om\|_{\om,q,p,\delta}\in L^d(M,P_0)$. Then there exists $\Sigma\geq0$ such that for $\bbP_0$-a.a. $\om$ we have
$$
\lim_{n\to\infty}\frac 1n\text{Var}_{\mu_\om}(S_n^\om f)=\Sigma^2.
$$
Moreover, $\Sigma=0$ if and only if there exist measurable functions $H_\om:\cY_\om\to\bbR$ such that for $\bbP_0$-a.a. $\om$,
$$
f_\om=\mu_\om(f_\om)+H_{\te\om}\circ T_\om-H_\om,\, \mu_\om-\text{ a.s}
$$
In the above case we must have $\|H_\om\|_{\om,a,p,\delta^{1/2-\eta}}\in L^d(M,P_0)$ for all $0<\eta<1/2$.
\end{theorem}

\begin{remark}
The case when $(X_j)$ is a stationary chain and $f_j=f$ does not depend on $j$ is included in the above setup by considering the case when $M$ is a singleton.    
\end{remark}
Now we are ready to formulate our local large deviations principle.
\begin{theorem}\label{LDP thm}
Suppose that $\varpi_{q,\infty}(n)\to 0$ for some $1\leq q\leq\infty$ and that $\text{ess-sup}\left(\|f_\om\|_{\om,\infty,\infty,\delta}\right)<\infty$  for some $\delta\in(0,1)$. 
If $\Sigma>0$ then there exists
$\varepsilon_0>0$ and a function $c:(-\varepsilon_0,\varepsilon_0)\to\bbR$ which is nonnegative, continuous, strictly
convex, vanishing only at $0$ and such that for $\bbP_0$-a.a. $\om$,
$$
\lim_{n\to\infty}\frac1n\ln\mu_\om(S_n^\om f-\mu_\om(S_n^\om f)>\varepsilon n)=-c(\varepsilon),\,\,\text{for all }\,\,\varepsilon\in(0,\varepsilon_0).
$$
\end{theorem}
We note that without the assumption that $\Sigma>0$ by Theorem \ref{VarRDS} the sums $S_n^\om f$ are uniformly bounded, and so for all $n$ large enough $\mu_\om(S_n^\om f-\mu_\om(S_n^\om f)>\varepsilon n)=0$. This means that formally we get the result with $c(\varepsilon)=\infty$. We also refer to Remark \ref{ConvRem} for a short discussion about large deviations principles for Markov chains with transition probabilities $Q_j(x,dy)$ that converge as $j\to\infty$ in an appropriate sense to a given transition probability $Q(x,dy)$.

\section{Two sided shifts: reduction to one sided shifts and related results}\label{Reduce Section}
\subsection{Reduction to arrays of functions under Assumption \ref{Ass2}}\label{RedSec1}
\subsubsection{A (coordinate-wise) re-centering procedure}
Let us fix some $N\in\bbN$ and let $0\leq j<N$.
Let us define 
$$
g_{j,(N)}=g_{j,(N)}(X_{j-[c\ln N]},X_{j-[c\ln N]+1},...)=\bbE\left[f_j|X_k, 
\cF_{j-[c\ln N],\infty}\right].
$$
Then $\bbE[g_{j,(N)}]=\bbE[f_j]$.
The ideas presented in this section is, for a fixed $N$, to consider $g_{j,(N)}, j<N$ as functions on $\cZ_{j-[c\ln N]}=\prod_{k\geq j-[c\ln N]}\cX_k$.
 This will reduce the problem to triangular arrays of functions that depend only on the present and the future (i.e. the reduction is to one sided shifts).  However, unlike \eqref{g bound} below, this has a certain affect of the approximation coefficients $v_{j-[c\ln N],s,\delta}(g_{j,(N)})$ since $f_j$ is centered around $X_j$ and not $X_{j-[c\ln n]}$. This issue will be addressed in Lemma \ref{v g} below.

Next, since conditional expectations contract $L^u$ norms, for all $u\geq 1$ we have 
\begin{equation}\label{g bound}
\|g_{j,(N)}\|_{L^u}\leq\|f_j\|_{L^u}.    
\end{equation}
Moreover, 
\begin{equation}\label{g approx}
\|f_j-g_{j,(N)}\|_{L^u}\leq v_{j,u,\delta}(f_j)\delta^{[c\ln N]}\leq v_{j,u,\delta}(f_j)N^{c\ln \delta}.
\end{equation}
Consequently, for all $j$ and $m$ such that $j+m<N$, with $S_{j,m,N}g=\sum_{k=j}^{j+m-1}g_{k,(N)}$,  we have
\begin{equation}\label{Sg approx}
\|S_{j,m}f-S_{j,m, N}g\|_{L^u}\leq N^{[c\ln N] \delta}\sum_{k=j}^{j+m-1}v_{k,u,\delta}(f_j).    
\end{equation}
We therefore get the following result. 
\begin{lemma}\label{EST LEMMA}
Under Assumption \ref{Ass2}, if $c=|\ln \delta|^{-1}(\zeta+1)$ then  for all $j$ and $m$ such that $j+m<N$,  
\begin{equation}\label{Sg approx2}
\|S_{j,m}f-S_{j,m,N}g\|_{L^s}\leq 2c_0+1
\end{equation}
where $c_0$ and $\zeta$ are specified in Assumption \ref{Ass2}. 
\end{lemma}
In Section \ref{Red CLT} we will see that this lemma is sufficient to deduce the optimal CLT rates for $S_Nf$ from the optimal CLT rates for $S_Ng=S_{0,N,N}g$.

Next let us obtain some estimates on $v_{j-[N\ln n],s,\delta}(g_{j,(N)})$.
\begin{lemma}\label{v g} 
Let $\eta\in(0,1)$. Then in the circumstances of Assumption \ref{Ass2} for every $0<w<1$ there are constants $C_w>0$ and $\delta_w\in(0,1)$ such that for all $0\leq j\leq N-1$ for  we have 
$$
v_{j-[c\ln N],s,\delta_w}(g_{j,(N)})=\sup_{r\geq1}\delta^{-\eta r}\|g_{j,(N)}-\bbE[g_{j,(n)}|\cF_{j-[c\ln N],j-[c\ln N]+r]}]\|_{L^s}\leq C_wN^{\zeta+w}.
$$
\end{lemma}

\begin{proof}[Proof of Lemma \ref{v g}]
Denote $\cF_{a,b}=\cF_{[a], [b]}$ for all $a,b$. Let $1>\beta\geq \delta$. 
Let $r\geq1$. If $j-[c\ln N]+r\geq j+\eta r$, namely $r(1-\eta)\geq [c\ln N]$ then
 $$
\|g_{j,(N)}-\bbE[g_{j,(N)}|\cF_{j-[c\ln N],j-[c\ln N]+r]}]\|_{L^s}\leq \|f_j-\bbE[f_j|\cF_{j-\eta r,j+\eta r}]\|_{L^s}\leq  v_{j,s,\delta}(f)\delta^{\eta r}\leq v_{j,s,\delta}(f)\beta^{r\eta}.
 $$
 On the other hand, if $j-[c\ln N]+r<j+\eta r$, then noting that in both cases $a\geq s$ we get
 $$
\|g_{j,(N)}-\bbE[g_{j,(N)}|\cF_{j-[c\ln N],j-[c\ln N]+r]}]\|_{L^s}\leq 2\|f_j\|_{L^s}\leq 2\|f_j\|_{L^s}\beta^{\eta r}\beta^{-\eta[c\ln N](1-\eta)}=\beta^{\eta r}\delta^{-\eta[c\ln N](1-\eta)}\left(\frac{\delta}{\beta}\right)^{-\frac{\eta [c\ln N]}{1-\eta}}
$$
$$
\leq 2\delta^{-\eta(1-\eta)}\|f_j\|_{L^s}\beta^{\eta r}N^{\frac{(1+\zeta)\eta}{(1-\eta)}}\left(\frac{\delta}{\beta}\right)^{\frac{\eta [c\ln N]}{1-\eta}}
 $$
 where we used that $\delta^{-c\ln N}=N^{(1+\zeta)}$. Next,
 let us take $\beta=\delta^{1-v}$ for $0\leq v<1$. Then
 $$
\left(\frac{\delta}{\beta}\right)^{\frac{\eta [c\ln N]}{1-\eta}}=\delta^{\frac{v\eta [c\ln N]}{1-\eta}}\leq \delta^{\frac{v\eta c\ln N}{1-\eta}}=N^{-(1+\zeta)\eta v}.
 $$
 Using that $\|f_j\|_{L^s}\leq CN^{\zeta}$ we get that when 
 $j-[c\ln N]+r<j+\eta r$, then 
 $$
\|g_{j,(n)}-\bbE[g_{j,(N)}|\cF_{j-[c\ln N],j-[c\ln N]+r]}]\|_{L^s}\leq C_\eta \beta^{r\eta}N^{\zeta+\frac{(1+\zeta)\eta}{1-\eta}-(1+\zeta)\eta v}=C_\eta \beta^{r\eta} N^{\zeta+(1+\zeta)\eta(\frac{1}{1-\eta}-v)}.
 $$
 Let $w>0$ and let $\eta$ small enough and $v$ close enough to $1$ so that $\frac{1}{1-\eta}-v<w$. Then,  when 
$j-[c\ln N]+r<j+\eta r$, 
 $$
\|g_{j,(n)}-\bbE[g_{j,(N)}|\cF_{j-[c\ln N],j-[c\ln N]+r]}]\|_{L^s}\leq C_\eta \beta^{r\eta} N^{\zeta+(1+\zeta)\eta w}\leq  N^{\zeta+w}
$$
assuming that $(1+\zeta)\eta<1$. This completes the proof of the lemma.
\end{proof}

Using Lemma \ref{v g} the strategy of the proof of Theorems \ref{BE} and \ref{ThWass} 
under Assumption \ref{Ass2} is to use the spectral approach with the norms $\|\cdot\|_{j,q,p,\delta^\eta}, j<N$ for a fixed sufficiently small $\eta\in(0,1)$. To overcome the problem that $\|g_{j,(N)}\|_{j-[c\ln N],q,p,\delta^\eta}$ is not bounded we replace  $g_{j,(N)}$ by $\tilde g_{j,(N)}=N^{-\zeta-w}g_{j,(N)}, w>0$. However, when $a<\infty$ in Assumption \ref{Ass2} this is still not enough to get the desired smoothness of the perturbations in the parameter $t$ of the perturbations of the operators $\cL_j$, since our approach of verifying it requires boundedness of the functions $\tilde g_{j,(N)}$ (note that these perturbations are given by $\cL_{j,t,(N)}(h)=\bbE[he^{it\tilde g_{j,(N)}}], t\in\bbR$). Because of that in Section \ref{TruncSec} we will first truncate $\tilde g_{j,N}$ in a certain way that ensures that the $L^\infty$ norm is of order $N^\te$ for some $\te>\zeta$. Then we take $w$ small enough and divide by $N^{\te}$ to get uniform boundedness in the $\|\cdot\|_{j,q,p,\delta^\eta}$ norms, which will allow us to get the desired smoothness.
When $a=\infty$ we can just use $\tilde g_{j,(N)}$ defined above that since
$\sig_N\geq c_1N^{\varepsilon+\zeta}$ and taking into account Lemma \ref{EST LEMMA} (ii) with $w<\varepsilon$ we still get that 
$$
\lim_{N\to\infty}\text{Var}\left(\sum_{j=0}^{N-1}\tilde g_{j,(N)}\right)=\infty
$$
which reduces the problem to a triangular array of functions which diverging variances. However, such normalization causes the rates to be $O(\sig_N^{-c(a,\zeta)})$ for some $c(a,\zeta)<\frac12$ such that $\lim_{a\infty, \zeta\to 0}c(a,\zeta)=\frac12$.

\subsubsection{A truncation argument}\label{TruncSec}
Let Assumption \ref{Ass2} hold with $a<\infty$ and let $b$ be defined by $1/k_0=1/a+1/b$ (note that $b=\frac{3k_0}{a-k_0}$). Let $\zeta,\varepsilon,c_0,c_1$ be like in that assumption.
Let us fix some $M>0$. Define a function $G_M:\bbR\to\bbR$ as follows. Set
$G_M(x)=x$ if $|x|\leq M$, set $G_m(x)=0$ if $|x|\geq 2M$ and on $[-2M,-M]$,\, let $G_M$ identify with the linear function connecting $(-2M,0)$ and $(-M,-M)$, while on $[M,2M]$ let it identify with the linear function connecting $(2M,0)$ and $(M,M)$. Then 
\begin{equation}\label{G M cond}
|G_M(x)-G_M(y)|\leq |x-y|\, \text{ and }\, |G_M(x)-x|\leq \bbI(|x|\geq M)|x|.    
\end{equation}
Let us take $M_j=(j+1)^{d}$ where $d=(b/a)(1+\zeta)+\zeta+\varepsilon/2-\te$ for some $0<\te<\varepsilon/2$. 
 Let $\bar g_{j,(N)}=G_{M_j}\circ g_{j,(N)}$ for $j<N$. Then for every $\eta\in(0,1)$
$$
v_{j-[c\ln N],s,\delta^\eta}(\bar g_{j,(N)})\leq v_{k,s,\delta}(g_{j,(N)})\leq CN^{\zeta+\frac{\eta(1+\zeta)}{1-\eta}}
$$
since $G_{M_j}$ is Lipschitz continuous with constant $1$. Note that
 $\|\bar g_{j,(N)}\|_{L^\infty}\leq M_j=(j+1)^d$. 
 Now, by the H\"older and the Markov inequalities and that $|G_M(x)-x|\leq \bbI(|x|\geq M)|x|$ we get that
$$
\|\bar g_{j,(N)}-g_{j,(N)}\|_{L^{k_0}}\leq \|g_{j,(N)}\bbI(g_{j,(N)}\geq M_j)\|_{L^3}\leq \|f_j\|_{L^a}\|f_j\|_{L^a}^{a/b}M_j^{-a/b}
$$
where we used that $\|g_{j,(N)}\|_{L^u}\leq \|f_j\|_{L^u}$ for all $u\geq1$.
Now, since $\|f_j\|_{L^a}\leq c_0(j+1)^{\zeta}$ and  $ad/b>1+\zeta+a\zeta/b$, using also \ref{Sg approx2} we get the following result.
\begin{lemma}
For all $j,m$ such that $j+m<N$ we have  
\begin{equation}\label{Sg approx 3}
 \|S_{j,m}f-S_{j,m}\bar g\|_{L^3}\leq c_0\sum_{j=1}^{n}j^{\zeta+a\zeta/b}j^{-ad/b}\leq C_1    
\end{equation}
for some constant $C_1$. 
\end{lemma}
Thus, as will be proven in Section \ref{Red CLT}, using the above Lemma it is enough to prove Theorems \ref{BE} and \ref{ThWass} for the sums $S_n\bar g=\sum_{j=0}^{n-1}\bar g_{j,(N)}$. Let us take $w<\varepsilon/4$  so that $\zeta+w<d$ and define 
$$
\tilde g_{j,(N)}=N^{-d} \bar g_{j,(N)}.
$$
Then there is a constant $C_2>0$ such that for all $N$ and $j<N$ we have 
$$
\sup_{j,N}\|\tilde g_{j,(N)}\|_{j-[c\ln N],\infty,s,\delta^{\eta}}<\infty
$$
Using \eqref{Sg approx 3} we see that there is constant $C_0>0$ such that for all $N$ large enough we have
$$
\text{Var}\left(\sum_{j=0}^{N-1}\tilde g_{j,N}\right)\geq C\sig_N^2 N^
{-2(b/a)(1+\zeta)-2\zeta-\varepsilon+2\te}\geq Cc_1 N^{2(b/a)(1+\zeta)+2\zeta+\varepsilon}\geq Cc_1 N^{2\te}\to \infty.
$$ 

\subsection{Reduction to arrays of functions under Assumption \ref{Ass1.1}}\label{RedSec2}
Let us fix some $n$. Let us define $B_{j,n}=B_j$ if $j<k_n$ and $B_{j,k_n}$ to be the union of $B_{k_n}$ and the part of $B_{k_n+1}$ that is contained in $[1,n]$. Then since we have uniform decay of correlations by taking $A$ large enough we still get
$$
\frac12 A\leq \|S_{B_{j,n}}f\|_{L^2}\leq 2A.
$$
Let $Y_{j,n}=S_{B_{j,n}}f$. Then for $j<k_n$ we have $Y_{j,n}=Y_j=S_{B_j}f$ and
$$
S_n=\sum_{j=1}^{k_n}Y_{j,n}.
$$
\subsubsection{A re-centering procedure}
Write $B_{j}=\{a_j,a_{j}+1,...,b_j\}$.
Let us take some $c>0$ and set $\bar Y_{j,n}=\bbE[Y_{j,n}|\cF_{a_{j-c\ln \sigma_n},\infty}]:=F_{j,n}(X_{j-[c\ln n]},X_{j-[c\ln\sigma_n+1},...)$. Then if $c$ is large enough, using that $a_{j-m}\leq a_j-m, m\geq0$ we have
$$
\sup_{j,n}\|\bar Y_{j,n}-Y_{j,n}\|_{L^p}\leq\sum_{a_j\leq \ell\leq b_j}\|f_\ell-\bbE[f_\ell|\cF_{a_j-[c\ln n],\infty}]\|_{L^p}\leq C\delta^{c\ln \sigma_n}\leq \sigma_n^{-2}.
$$
Let 
$$
\bar S_n=\sum_{j=0}^{n-1}\bar Y_{j,n}.
$$
Then 
\begin{equation}\label{Approx spec}
\sup_n\sum_{j=1}^{k_n}\|\bar Y_{j,n}-Y_{j,n}\|_{L^s}<\infty.   
\end{equation}
In particular,
\begin{equation}\label{In part}
\sup_n\|S_n-\bar S_n\|_{L^s}<\infty.    
\end{equation}
Thus, as will be proven later on, it is enough to obtain optimal CLT rates for $S_n$ by using rates for $\bar S_n$.

Next, using that $\sup_{\ell}\max_{B\subset B_\ell}\|S_{B}\|_{L^p}<\infty$ and the contraction of conditional expectations we see that 
\begin{equation}\label{ap}
 \sup_{j,n}\|\bar Y_{j,n}\|_{L^p}<\infty.   
\end{equation}
Next, we need the following result. Let us view $\bar Y_{j,n}$ as a function on the space $\prod_{k\geq a_{j-[c\ln \sig_n]}}\cX_k$. 
Let $\Upsilon_j=(X_{k})_{k\in B_j}$. Then we can view $\bar Y_{j,n}$ as a function of the path of $\Upsilon_m$, starting from $m=a_{j-[c\ln\sig_n]}$. Arguing like in the proof of Lemma \ref{v g} we get the following result.
\begin{lemma}\label{v g 2}  
Let $\eta\in(0,1)$. Then in the circumstances of Assumption \ref{Ass1.1} for every $0<w<1$ there are constants $C_w>0$ and $\delta_w\in(0,1)$ such that for all $0\leq j\leq n-2$ for  we have 
$$
v_{a_{j-[c\ln n]},s,\delta_w}(\bar Y_{j,n})=\sup_{r\geq1}\delta^{-\eta r}\|Y_{j,n}-\bbE[Y_{j,n}|\cF_{a_{j-[c\ln N]}},b_{j+[c\ln\sig_n]+r]}]\|_{L^s}\leq C_w\sig_n^{w}.
$$
\end{lemma}
This lemma shows that upon replacing the chain $(X_j)$ with the new chain $(\Upsilon_j)$ (which inherits the mixing properties of $(X_j)$) we can consider arrays of one sided functionals of  $(\Upsilon_j)$ centered at $a_{j-[c\ln\sig_n]}$.
\subsubsection{A truncation argument}
As before, let  $G_M:\bbR\to\bbR$ be defined as follows. Set
$G_M(x)=x$ if $|x|\leq M$, set $G_m(x)=0$ if $|x|\geq 2M$ and on $[-2M,-M]$,\, let $G_M$ identify with the linear function connecting $(-2M,0)$ and $(-M,-M)$, while on $[M,2M]$ let it identify with the linear function connecting $(2M,0)$ and $(M,M)$. 

Let us take some  $M_n>1$ and let $\tilde Y_{j,n}=G_{M_n}(\bar Y_{j,n})$. 
  the H\"older and the Markov inequalities and since $|G_M(x)-x|\leq \bbI(|x|\geq M)|x|$ we get that
$$
\|\tilde Y_{j,n}-\bar Y_{j,n}\|_{L^3}\leq \|\bar Y_{j,n}\bbI(\bar Y_{j,n}\geq M_n)\|_{L^3}\leq \|S_{B_{j,n}f}\|_{L^p}\|S_{B_{j,n}f}\|_{L^a}^{p/a}M_n^{-p/a}\leq CM_n^{-p/a}
$$
where $1/k_0=1/p+1/a$ for $a\leq p$.
Therefore, with $\tilde S_n=\sum_{j=1}^{k_n}\tilde Y_{j,n}$ we have
\begin{equation}\label{approx11}
\|S_n-\tilde S_n\|_{L^{k_0}}\leq C(1+\sigma_n^2M_n^{-p/a})=O(1) 
\end{equation}
assuming that $M_n=\sigma_{n}^{2a/p}$.  Define 
$$
\textbf{S}_n=\sig_n^{-2a/p}\tilde S_n.
$$
Then 
$$
\|\textbf{S}_n\|_{L^2}\geq C\sig_n^{1-2a/p}
$$
and the summands in $\textbf{S}_n$ are uniformly bounded in the $\|\cdot\|_{\cdot,\infty,s,\delta_w}$ norms, assuming that $w$ is small enough. Thus

\subsection{Sinai's lemma an related results}
A key tool in our proofs is to reduce all the limit theorems to the case when $f_j$ depends only on $X_j,X_{j+1},...$. 
Denote $\cZ_j=\cX_j\times\cX_{j+1}\cdots=\{(x_{j+k})_{k\geq0}: x_s\in\cX_s\}$.
For a measurable function $g:\cZ_j\to\bbR$ denote by
$\|g\|_{j,a,b,\delta}$ the norm of $g$  when viewing $g$ as a function on $\cY_j$ which depends only on the coordinates $x_{j+k}, k\geq0$.
Note that because of the Markov property,
$$
v_{j,a,\delta}(g)=\sup_{r\geq0}\delta^{-r}\|g(X_j,X_{j+1},...)|X_j,X_{j+1},...,X_{j+r}\|_{L^a},
$$
that is, there is no need in conditioning on $X_s$ for $s<j$.
Let $\pi_j:\cY_j\to\cZ_j$ be given by 
$$
\pi_j(y)=(y_{k+j})_{k\geq0}, \,y=(y_{j+k})_{k\in\bbZ}.
$$
Let $\tau_j:\cZ_j\to\cZ_{j+1}$ denote the left shift and set  $\tau_j^n=\tau_{j+n-1}\cdots \tau_{j+1}\circ S_j, n\in\bbN$.
The following result shows that we can reduce limit theorems for sums of the form $S_n=\sum_{j=0}^{n-1}f_j\circ T_j^n$, with $f_j:\cY_j\to\bbR$ to sums of the form $S_n=\sum_{j=0}^{n-1}g_j\circ \tau_j^n$ with  $g_j:\cZ_j\to\bbR$ is based on the following version of Sinai's Lemma.
\begin{lemma}\label{Sinai}
Let $f_j:\cY_j\to\bbR$ be such that $\sup_j\|f_j\|_{j,q,a,\delta}<\infty$ for some $a, q\geq 1$. Then there exist functions $u_j:\cY_j\to\bbR$ and $g_j:\cZ_j\to\bbR$ such that $\sup_j\|u_j\|_{j,a,a,\delta}\leq 2(1-\delta^{1/2})^{-1}\sup_{j}v_{j,a,\delta}(f_j)$,
and
$$
f_j=u_{j+1}\circ T_j-u_j+g_j\circ\pi_j.
$$
The function $g_j$ is given by 
\begin{equation}\label{gj}
g_j=\sum_{m=0}^\infty(\bbE[f_{j+m+1}|X_j,X_{j+1},...]-\bbE[f_{j+m+1}|X_{j+1},X_{j+2},...])+\bbE[f_j|X_j,X_{j+1},...]  
\end{equation}
and  we have $\sup_j\|g_j\|_{j,\min(a,q),a,\delta^{1/2}}\leq 4(1-\delta^{1/2})^{-1}\sup_j\|f_j\|_{j,q,a,\delta}.
$
\end{lemma}
\begin{remark}
It is clear that $g_j=f_j$ when $f_j$ depends only on the coordinates $x_{j+k},k\geq 0$. In that case it will follow from the proof of Lemma \ref{Sinai} that $u_j=0$.    
\end{remark}

\begin{proof}[Proof of Lemma \ref{Sinai}]
In the course of the proof we write $X_t=X_{[t]}$ for a real number $t$.
 Define $u_j:\cY_j\to\bbR$ by
 $$
u_j=\sum_{k=0}^\infty \left(f_{j+k}\circ T_j^k-\bbE[f_{j+k}\circ T_j^k|X_{j},X_{j+1},...]\right)
$$
$$
=\sum_{k=0}^\infty \left(f_{j+k}(...,X_{j+k-1},X_{j+k},X_{j+k+1},...)-\bbE[f_{j+k}(...,X_{j+k-1},X_{j+k},X_{j+k+1},...)|X_{j},X_{j+1},...]\right)
 $$
 Then 
 $$
\|u_j\|_{L^a}\leq 2\sum_{k\geq 0}v_{j+k,a,\delta}(f_{j+k})\delta^{k}
$$
$$
+\sum_{k\geq 0}\left\|\bbE[f_{j+k}\circ T_j^k|X_{j},...,X_{j+2k}]-\bbE[\bbE[f_{j+k}\circ T_j^k|X_{j},...,X_{j+2k}]|X_{j},X_{j+1},...]\right\|_{L^a}
$$
$$
=2\sum_{k\geq 0}v_{j+k,a,\delta}(f_{j+k})\delta^{k}\leq 2\sup_{m}v_{m,a,\delta}(f_m)(1-\delta)^{-1}
 $$
 where we used that 
 $$
 \bbE[\bbE[f_{j+k}\circ T_j^k|X_{j},...,X_{j+2k}]|X_{j},X_{j+1},...]=\bbE[f_{j+k}\circ T_j^k|X_{j},...,X_{j+2k}].
 $$
 Notice that
 $$
u_j-u_{j+1}\circ T_j=f_j+\sum_{k=0}^\infty\left(\bbE[f_{j+1+k}\circ T_j^k|X_{j+1},X_{j+2},...]-\bbE[f_{j+k}\circ T_{j+1}^k|X_{j},X_{j+1},...]\right)
 $$
 and so $u_j-u_{j+1}\circ T_j-f_j$ depends only on the coordinates with indexes $j+k,k\geq 0$. Set $g_j=f_j+u_{j+1}\circ T_j-u_j$.

 In order to complete the proof of the lemma it is enough to show that $\sup_j v_{j,a,\delta^{1/2}}(u_j)<\infty$. For that purpose we write
 $$
\left\|u_j-\bbE[u_j|X_{j-r},...,X_{j+r}]\right\|_{L^a}\leq\sum_{k=0}^{r/2}\left\|f_{j+k}\circ T_j^k-\bbE[f_{j+k}\circ T_j^k|X_{j-r},...,X_{j+r-1},X_{j+r}]\right\|_{L^a}
$$
$$
+\sum_{k=0}^{r/2}\left\|\bbE[f_{j+k}\circ T_j^k|X_{j},X_{j+1},...]-\bbE[\bbE[f_{j+k}\circ T_j^k|X_{j},X_{j+1},...]|X_{j-r},...,X_{j+r}]\right\|_{L^a}+2\sum_{k>r/2}v_{j+k,a,\delta}(f_{j+k})\delta^{k}.
 $$
 Next, for $k\leq r/2$ write 
 $$
  j-r=j+k-(r-k)\,\, \text{ and }\,\,j+r=j+k+(r-k).    
 $$
Then
 \begin{equation}\label{Thus}
 \left\|f_{j+k}\circ T_j^k-\bbE[f_{j+k}\circ T_j^k|X_{j-r},...,X_{j+r}]\right\|_{L^a}\leq v_{j+k,a,\delta}(f_{j+k})\delta^{r-k}    
 \end{equation}
 Next, write $f_{j,k,r}=\bbE[f_{j+k}\circ T_j^n|X_{j-r},...,X_{j+r}]$. Then by \eqref{Thus} and the contraction properties of conditional expectations,
\begin{equation}\label{UPy}
\left\|\bbE[f_{j+k}\circ T_j^k|X_{j},X_{j+1},...]-\bbE[\bbE[f_{j+k}\circ T_j^k|X_{j},X_{j+1},...]|X_{j-r},...,X_{j+r}]\right\|_{L^a}\leq 
2v_{j+k,a,\delta}(f_{j+k})\delta^{r-k}    
\end{equation}
$$
+\left\|\bbE[f_{j,k,r}|X_{j},X_{j+1},...]-\bbE[\bbE[f_{j,k,r}|X_{j},X_{j+1},...]|X_{j-r},...,X_{j+r}]\right\|_{L^a}.
$$
Notice that by the Markov property we have
$$
\bbE[\bbE[f_{j,k,r}|X_{j},X_{j+1},...]|X_{j-r},...,X_{j+r}]=\bbE[\bbE[f_{j,k,r}|X_{j},X_{j+1},...X_{j+r}]|X_{j-r},...,X_{j+r}]
$$
$$=
\bbE[f_{j,k,r}|X_{j},X_{j+1},...X_{j+r}].
$$
Using again the Markov property we see that 
$$
\bbE[f_{j,k,r}|X_{j},X_{j+1},...]=\bbE[f_{j,k,r}|X_{j},X_{j+1},...,X_{j+r}].
$$
Thus the second term on the right hand side of \eqref{UPy} vanishes. By combining the above estimates we conclude that
$$
\left\|u_j-\bbE[u_j|X_{j-r},...,X_{j+r}]\right\|_{L^a}\leq(1-\delta^{1/2})^{-1}\sup_{m}v_{m,a,\delta}(f_m)\delta^{r/2}.
$$
\end{proof}

\begin{remark}\label{Sinai Rem}
In Assumption \ref{Ass2} we allowed that $\|f_{j}\|_{j,a,s,\delta}=O((j+1)^\zeta)$ for some $0<\zeta<1$. Using that $(j+m)^\zeta\leq j^\zeta+m^{\zeta}$ and that $\sum_{k=0}^{r/2}(k+1)^\zeta \delta^{r-k}$ is of order $\delta^{(\frac12-\rho)^r}$ for all $\rho>0$
it is not hard to show that in this case the arguments in the proof of Lemma \ref{Sinai} yield that 
$\|g_j\|_{j,a,s,\delta^{1/3}}=O((j+1)^\zeta)$ and similarly $\|u_j\|_{j,s,s,\delta^{1/3}}=O((j+1)^\zeta)$.
\end{remark}

The following result shows that the functions $g_j$ from Lemma \ref{Sinai} satisfy a certain conditional regularity condition that will ensure that the the operators $h\to \bbE[h(X_j,X_{j+1},...)e^{itg_j(X_{j},X_{j+1},...)}|X_{j+1},X_{j+2},...], t\in\bbR$ are of class $C^k$ in the parameter $t$ when acting on the space of functions with finite $\|\cdot\|_{j,\infty,\infty,\delta^{1/2}}$.
 
\begin{proposition}\label{Special Prop}
Let $f_j:\cY_j\to\bbR$ be such that $\sup_j\|f_j\|_{j,q,a,\delta}<\infty$ for some $q,a\geq 1$.
Let $k\in\bbN$. Suppose that 
there is a constant $C>0$ such that
\begin{equation}\label{1 cond1}
\bbE[|f_j(...,X_{j-1},X_j,X_{j+1},....)|^k|X_{j+1},X_{j+2},...]\leq C
\end{equation}
almost surely. 
Moreover, assume that for all $m\geq 0$ and $r\geq m$ we have,
\begin{equation}\label{1 cond2}
\bbE[|f_{j+m}-F_{j+m,r}|^k|X_{j},X_{j+1},...]\leq C\delta^{rk}
\end{equation}
where $F_{s,r}$ is an $\cF_{s-r,s+r}$ measurable function. Namely, let Assumption \ref{DomAss} be in force.
Let $g_j$ be the functions from Lemma \ref{Sinai}. Then there is a constant $C_1>0$ such that almost surely we have
$$
\bbE[|g_j(X_j,X_{j+1},....)|^k|X_{j+1},X_{j+2},...]\leq C_1
$$
and 
$$
\sup_r\delta^{-rk/2}\bbE[|g_j-g_{j,r}|^k|X_{j+1},X_{j+2},...]\leq C_1 
$$
where $g_{j,r}=\bbE[g_j|X_j,...,X_{j+r}]$.
\end{proposition}
\begin{proof}
Recall that  $g_j$ is given by
$$
g_j=\sum_{m=0}^\infty(\bbE[f_{j+m+1}|X_j,X_{j+1},...]-\bbE[f_{j+m+1}|X_{j+1},X_{j+2},...])+\bbE[f_j|X_j,X_{j+1},...].
$$
 Denote 
$$
D_{j+m+1}=\bbE[f_{j+m+1}|X_j,X_{j+1},...]-\bbE[f_{j+m+1}|X_{j+1},X_{j+2},...]
$$
and 
$$
\bar D_{j+m+1,m}=\bbE[F_{j+m+1,m}|X_j,X_{j+1},...]-\bbE[F_{j+m+1,m}|X_{j+1},X_{j+2},...].
$$
Then  $\bar D_{j+m+1,m}=0$ since $F_{j+m+1,m}$ is a function of $X_{j+1},X_{j+2},...$. Therefore, 
\begin{equation}\label{D1}
\left(\bbE[|D_{j+m+1}|^k|X_{j+1},X_{j+2},...]\right)^{1/k}=
\left(\bbE[|D_{j+m+1}-\bar D_{j+m+1,m}|^k|X_{j+1},X_{j+2},...]\right)^{1/k}\leq 2C\delta^m
\end{equation}
where the last inequality uses \eqref{1 cond2}.
Thus there is a constant $A_k>0$ such that
$$
\bbE[|g_j|^k|X_{j+1},X_{j+2},....]\leq A_k\left(\sum_{m=1}^\infty\delta^m+C\right)^k
$$
where $C$ comes from \eqref{1 cond1}.

Next, let $r\in\bbN$. Denote 
$$
D_{j+m+1,r}=\bbE[D_{j+m+1}|X_{j},...,X_{j+r}]=\bbE\left[\left(\bbE[f_{j+m+1}|X_{j+1},X_{j+2},...]-\bbE[f_{j+m+1}|X_{j},X_{j+1},...]\right)|X_{j},...,X_{j+r}\right].
$$
Then  by the conditional Jensen inequality
$$
|D_{j+m+1,r}|^k\leq \bbE[|D_{j+m+1}|^k|X_{j},...,X_{j+r}]
$$
and so by \eqref{D1},
$$
\bbE[|D_{j+m+1,r}|^k|X_{j+1},X_{j+2},...]\leq 
\bbE[\bbE[|D_{j+m+1}|^k|X_{j},...,X_{j+r}]|X_{j+1},X_{j+2},...]
$$
$$=
\bbE[\bbE[|D_{j+m+1}|^k|X_{j+1},...,X_{j+r}]|X_{j+1},X_{j+2},...]=
\bbE[\bbE[|D_{j+m+1}|^k|X_{j+1},X_{j+2},...]|X_{j+1},...,X_{j+r}]
\leq (2C\delta^m)^k
$$
Thus,
$$
\cD_{r,k}:=\sum_{m\geq [r/2]-1}\left(\bbE[|D_{j+m+1}|^k|X_{j+1},X_{j+2},....]\right)^{1/k}+
\sum_{m\geq [r/2]-1}\left(\bbE[|D_{j+m+1,r}|^k|X_{j+1},X_{j+2},....]\right)^{1/k}\leq C'\delta^{r/2}.
$$
On the other hand, if $-1\leq m<[r/2]-1$ then 
$$
|\bbE[f_{j+m+1}|X_j,X_{j+1},...]-\bbE[f_{j+m+1}|X_j,X_{j+1},...,X_{j+r}]|
$$
$$
\leq|\bbE[F_{j+m+1,r-m-1}|X_j,X_{j+1},...]-\bbE[F_{j+m+1,r-m-1}|X_j,X_{j+1},...,X_{j+r}]|
$$
$$
+\bbE[|f_{j+m+1}-F_{j+m+1,r-m-1}||X_j,X_{j+1},...]|=\bbE[|f_{j+m+1}-F_{j+m+1,r-m-1}||X_j,X_{j+1},...]
$$
where the last equality uses the Markov property.
Therefore, using again the conditional Jensen inequality we see that
$$
\left(\bbE\left[|\bbE[f_{j+m+1}|X_j,X_{j+1},...]-\bbE[f_{j+m+1}|X_j,X_{j+1},...,X_{j+r}]|^k\right|X_{j},X_{j+1},...]\right)^{1/k}
$$
$$
\leq\left(|\bbE[|f_{j+m+1}-F_{j+m+1,r-m-1}|^k|X_j,X_{j+1},...]\right)^{1/k}\leq C\delta^{r-m}
$$
where the last inequality uses \eqref{1 cond2}.
Similarly
$$
\left(\bbE\left[|\bbE[f_{j+m+1}|X_j,X_{j+1},...]-\bbE[f_{j+m+1}|X_j,X_{j+1},...,X_{j+r}]|^k\right|X_{j+1},X_{j+2},...]\right)^{1/k}\leq C\delta^{r-m}.
$$
By combining the above estimates and using the triangle inequality in $L^k$ with respect to  the conditional measure on $X_{j+1},X_{j+2},...$  we see that there exists a constant $C_0>0$ such that almost surely we have
$$
\left(\bbE[|g_{j}-g_{j,r}|^k|X_{j+1},X_{j+2},...]\right)^{1/k}\leq C_0\delta^{r/2}.
$$
\end{proof}

\subsection{Reduction of optimal CLT rates  from two sided functionals to one sided ones}\label{Red CLT}
Here we explain how to derive Theorems \ref{BE} and \ref{Var them} for the sums $S_nf$ from the corresponding results for either $S_ng=\sum_{j=0}^{n-1}g_j(X_j,X_{j+1},...)$ or $S_ng=\sum_{j=0}^{n-1}\bar g_{j,(n)}(X_{j-[c\ln n]},X_{j-[c\ln n]+1},...)$, where the functions $g_j$ are given in Lemma \ref{Sinai} and  $\bar g_{j,(n)}=G_{M_j}(\bbE[f_j|\cF_{j-[c\ln n],\infty}])$, with $M_j=(j+1)^d$ and $d=\frac{3(1-\zeta)}{a-3}+\zeta+\varepsilon-\te$ where $a,\zeta$ and $\varepsilon$ are specified in Assumption \ref{Ass2} and $\te$ is an arbitrary number such that $0<\te<\varepsilon$ (in the sequel we will take $\te$ close to $\varepsilon$). The same arguments will show how to reduce the results under Assumption \ref{Ass1.1}, but we decided to skip the details which are left for the reader.

\begin{proposition}\label{RedProp}
Assume that $\sup_j\|f_j\|_{j,a,a,\delta}<\infty$ for some $a\geq2$ and $\delta\in(0,1)$. 
If Theorems \ref{BE} and \ref{ThWass} hold for  $S_ng$ with rate $\sig_n^{-(1-u)}$ for some $u<1$
then the hold for $S_nf$ and with the same rate.    
\end{proposition}
\begin{proof}
Note that it is enough to prove the proposition when $\bbE[f_j]=0$. Note also that in this case \eqref{gj} we also have $\bbE[g_j]=0$ and $\bbE[g_{j,(n)}]=0$. Now,  by Lemma either Lemma \ref{Sinai} or Lemma \ref{EST LEMMA},
\begin{equation}\label{App}
A=\sup_n\|S_nf-S_ng\|_{L^a}<\infty.
\end{equation}

Next, note that part (ii) of Theorem \ref{BE} is a direct consequence of part (i).  Theorem \ref{BE} (iii) 
also follows from Theorem \ref{BE}(i). Indeed, 
for every random variable $W$ with distribution function $F$ and a function $h$ 
satisfying $H_s(h)<\infty$
we have
$$
E[h(W)]-h(\infty)
=-E\left[\int_{W}^{\infty}h'(x)dx\right]=-\int_{-\infty}^{\infty}h'(x)P(W\leq x)dx= -\int_{-\infty}^{\infty}h'(x)F(x)dx.
$$
To show that Theorem \ref{BE}(i) for $S_n g$ implies Theorem \ref{BE}(i) for  $S_n f$,  
let 
$$
F_n(t)\!=\!\bbP\left(\frac{S_n f}{\sig_n}\!\leq\! t\right),
\quad G_n(t)=\bbP\left(\frac{S_n g}{\kappa_n}\leq t\right)
\text{ where }\kappa_n=\|S_n g\|_{L^2}
\text{ and }\sig_n=\|S_nf\|_{L^2}.$$ 
By \eqref{App}  and the triangle inequality, $|\sig_n-\kappa_n|\leq \|S_nf-S_ng\|_{L^2}\leq A$. To complete the proof
 fix $s\geq 0$ and assume that Theorem \eqref{BE} (i) holds for $S_ng$ with that $s$. 
 Let $\rho=\rho_n(t)$ be given by $\rho^a=\delta_n\sigma_n^{-(a-1)}(1+|t|^s)$ for some positive sequence $\delta_n=\delta_n(t)$ which is bounded and bounded away from the origin and which will be specified latter (it will follow that we can take $\delta_n(t)$ to either $2$ or $\frac14$). Then,
$$
F_n(t)\leq G_n((t+\varepsilon)\sig_n/\kappa_n)+\bbP(|S_n f-S_n g|>\sig_n\varepsilon):=I_1+I_2.
$$
Now, by the Markov inequality we have 
$$
I_1=\bbP(|S_n f-S_n g|^a>\sig_n\rho^a)\leq \|S_nf-S_ng\|_{L^a}^a\sig_n^{-a}\rho^{-a}\leq C_1A^a(1+|t|^s)^{-1}\sigma_n^{-1}
$$
for some constant $C_1>0$.
Next, by the validity of Theorem \ref{BE} (i) for $S_ng$ with rate $\sig_n^{-(1-u)}$, for all $n$ large enough we have 
$$
G_n((t+\rho)\sig_n/\kappa_n)\leq C_s\sig_n^{-(1-u)}\left(1+\left|\frac{(t+\rho)\sig_n}{\kappa_n}\right|^s\right)^{-1}+\Phi((t+\rho)\sig_n/\kappa_n)).
$$
where we used that $\sig_n/\kappa_n\to 1$. Next, we claim that for all $n$ large enough and all $t$ we have we can choose $\frac14\leq \delta_n=\delta_n(t)\leq 2$ such that for all $t$,
$$
\sig_n^{-1}\left(1+\left|\frac{(t+\rho)\sig_n}{\kappa_n}\right|^s\right)^{-1}\leq B_s(1+|t|^s)^{-1}\sig_n^{-1}  
$$
for some constant $B_s>0$ which does not depend on $n$ and $t$. Indeed, since $\sig_n/\kappa_n\to 1$ the above estimate is equivalent to 
\begin{equation}\label{eq}
 1\leq B'_s\left(\frac1{1+|t|^s}+\left|\frac{t}{1+|t|}+\delta_n\sigma_n^{-(a-1)}(1+|t|)^{s-1}\right|^s\right).   
\end{equation}
Let $K>0$ be such that $\frac12\leq |\frac{t}{1+|t|}|\leq \frac32$ whenever $|t|\geq K$. When $|t|\leq K$ we just take $\delta_n(t)=1$ and then \eqref{eq} holds with some constant.
To show that the above estimate holds when $|t|>K$, if $\sigma_n^{-(a-1)}(1+|t|)^{s-1}\geq 1$ and $|t|\geq K$ then by taking $\delta_n(t)=2$ we see that
$$
\left|\frac{t}{1+|t|}+\delta_n\sigma_n^{-(a-1)}(1+|t|)^{s-1}\right|\geq \frac12
$$
and so \eqref{eq} holds. On the other hand, if $\sigma_n^{-(a-1)}(1+|t|)^{s-1}<1$ we take $\delta_n(t)=\frac14$ which together with $|\frac{t}{1+|t|}|\geq \frac12$ yields
$$
\left|\frac{t}{1+|t|}+\delta_n\sigma_n^{-(a-1)}(1+|t|)^{s-1}\right|\geq \frac14
$$
yielding \eqref{eq}.

Finally, let us write $(t+\rho)\sig_n/\kappa_n=t+\frac{\rho\sig_n}{\kappa_n}+\frac{t(\sig_n-\kappa_n)}{\kappa_n}$. Using that 
$|\Phi(x+\delta)-\Phi(x)|\leq C\delta e^{-x^2/2}$ for every $x$ and $\delta>0$ for some absolute constant $C>0$ and that $\frac{t(\sig_n-\kappa_n)}{\kappa_n}=O(t\sig_n^{-1})$ we see that 
$$
\Phi((t+\rho)\sig_n/\kappa_n))\leq\Phi(t)+\left|\frac{\rho\sig_n}{\kappa_n}+\frac{t(\sig_n-\kappa_n)}{\kappa_n}\right|e^{-ct^2}
$$
for some $c>0$ (and all $n$ large enough). Noticing that $\rho\leq C_0\sigma_n^{-1}(1+|t|^s)$ we get that the above right hand side does not exceed $C'_s(1+|t|)^{-s}\sig_n^{-1}$ for some constant $C'_s>0$. Combining the above estimates we see that 
$$
F_n(t)\leq  \Phi(t)+C''_s(1+|t|^s)^{-1}\sig_n^{-(1-u)}
$$
for some constant $C''_s>0$. A similar argument shows that 
$$
F_n(t)\geq \Phi(t)-C''_s(1+|t|^s)^{-1}\sig_n^{-(1-u)}.
$$

Finally to deduce Theorem \ref{ThWass} for $S_n f$ from the corresponding result for $S_ng$  let us take some $b<a-1$. By Theorem \ref{ThWass}, we can couple $S_ng$ with a standard normal random variable $Z$ 
so that $\|S_n g/\sig_n-Z\|_{L^b}\leq C\sig_n^{-(1-u)}$. Now by Berkes--Philipp Lemma \cite[Lemma A.1]{BP},  we can also couple all three random variables $S_n g$, $S_nf$ and $Z$ so that \eqref{App} still holds under the new probability law. 
\end{proof}

\section{Examples and applications}\label{AppSec}
Linear, Garch, things from statistics, random matrices, random Lyponov exponents, random operators, things from dynamics that can be modeled by non-stationary Bernoulli shifts

\subsection{Products of positive matrices and other operators}
Let us begin with a more abstract description. Let $(X_j)_{j\in\bbZ}$ be a Markov chain satisfying \eqref{mix} with some $p\geq1$. Let $\cX_j$ be the state space of $X_j$.
Let $B_j$ be (possibly random) Banach spaces of functions on some space equipped with a norm $\|\cdot\|_{B_j}$ satisfying $\|g\|_{B_j}\geq \sup|g|$. 
Let $A_j(X_j)$ be a bounded linear operator from $B_j$ to $B_{j+1}$ such that for some constant $C>0$ we have $\|A_j(X_j)\|_{B_j\to B_{j+1}}\leq C$. Define 
$$
\textbf{A}_j^n=A_{j+n-1}(X_{j+n-1})\cdots A_{j+1}(X_{j+1})A_j(X_j).
$$
We assume that there are (possibly random) Birkhoff cones $\cC_j\subset B_j$ and $n_0\in\bbN$ such that for all $j$ we have $A(X_j)\cC_j\subset\cC_{j+1}$ (almost surely) and the projective diameter of $\textbf{A}_j^{n_0}\cC_j$ inside $\cC_{j+n_0}$
does not exceed some constant $d_0<\infty$ which is independent of $j$ (see \cite[Appendix A]{HK} for an overview of projective metrics and cones). Finally, let us assume that the cones are regenerating in the sense of \cite[Section 5]{Rug}, namely  there exist $r\in\bbN$ and $C>0$ such that every $g\in B_j$ can be written as 
$$
g=\sum_{k=1}^r g_k,\,\, g_k\in\cC_j, \text{ and }\sum_{k=1}^r\|g_k\|_{B_j}\leq C\|g\|_{B_j}.
$$
Then by the arguments of \cite[Ch.4]{HK}
we obtain the following random Perron frobenious theorem. There are random variables $\la_j=\la_j(...,X_{j-1},X_j,X_{j+1},...)$ random vectors $h_j=h_j(...,X_{j-1},X_j,X_{j+1},...)\in B_j$ and random functionals $\nu_j=\nu_j(...,X_{j-1},X_j,X_{j+1},...)\in B_j^*$ and constants $C>0$ and $\delta\in(0,1)$ such that, for all $j$ and $n$, almost surely we have
\begin{equation}\label{RRPF}
\left\|\frac{\textbf{A}_j^n}{\la_{j,n}}-\nu_j\otimes h_{j+n}\right\|_{B_{j+n}}\leq C\delta^n    
\end{equation}
where $\la_{j,n}=\prod_{k=j}^{j+n-1}\la_k$ and $(\nu_j\otimes h_{j+n})(g)=\nu_j(g)h_{j+n}$. Moreover, $A_j(X_j)h_j=\la_jh_{j+1}, (A_j(X_j))^*\nu_{j+1}=\la_j\nu_j$ and $\nu_j(h_j)=1$. It also follows from the arguments in the proof (or directly from \eqref{RRPF}) that $h_j,\nu_j,\la_j$ can be approximated in $L^\infty$ exponentially fast in $r$ by variables taking values in appropriate spaces that depend on $X_{j+k}$ for $|k|\leq r$. Therefore, by taking logarithms we conclude that we get limit theorems for 
$$
\ln\|\textbf{A}_0^n\|\text{ and }\ln\left(\mu_{n}(\textbf{A}_0^n g)\right)
$$
where $g\in B_0$ and $\mu_n\in B_n^*$, $\sup_n\|\nu_n\|<\infty$ and we assume  that $g\in\cC_j$ and $\mu_n\in\cC_n^*$ (the dual is the set of functionals which take positive values on the cone).  
\begin{example}[Random functions with positive entries]
Each $A_j(X_j)$ is a random matrix of dimension $d$ with positive entries which are uniformly bounded and uniformly bounded away from the origin. Then all the conditions hold with $\cC_j$ being the first quadrant 
Our results sharpens the CLT in \cite{KestFurs61} in the Markov case.
\end{example}

\begin{example}[Random transfer operators]
Each $A_j(X_j)$ is the random transfer operator associated with a random expanding map $T_{X_j}$ satisfying the conditions of \cite[Ch.5]{HK}. That is 
$$
(A_j(X_j))g(x)=\sum_{y\in(T_{X_j})^{-1}\{x\}}e^{\phi_{X_j}(y)}g(y)
$$
for a H\"older continuous random functions $\phi_{X_j}$ (which are uniformly H\"older continuous) and H\"older continuous functions $g$. 
Then by \cite[Theorem]{HK} there are random cones satisfying the above conditions and \eqref{RRPF} holds.
\end{example}

\subsection{Random Lyapunov exponents}\label{Lyp Sec}
Let $d>1$ and let $A$ be a hyperbolic matrix with distinct eigenvalues $\la_1,...,\la_d$. Suppose that for some $k<d$ we have $\la_1<\la_2<...<\la_k<1<\la_{k+1}<...<\la_{d}$. Let $h_j$ be the corresponding eigenvalues.

Now, let $(A_j)$ be a sequence of matrices such that $\sup_j \|A_j-A\|\leq \ve$. Then, if $\ve$ is small enough there are numbers $\la_{j,1}<\la_{j,2}<...<\la_{j,k}<1<\la_{j, k+1}<...<\la_{j,d}$
and vectors $h_{j,i}$ such that 
$$
A_{j}h_{j,i}=\la_{j,i}h_{j+1,i}.
$$
Moreover, $\sup_j|\la_{j,i}-\la_i|$ and 
$\sup_j\|h_{j,i}-h_{i}\|$ converge to $0$ as $\ve\to 0$.

Now, the sequence $(A_j)$ is uniformly hyperbolic and the sequences 
$(\la_{1,j})_j,...,(\la_{d,j})_j$ can be viewed as its sequential Lyapunov exponents. Moreover,  
the one dimensional  spaces $H_{i,j}=\text{span}\{h_{i,j}\}$ can be viewed as its sequential Lyapunov spaces.
Next,  $\la_{i,j}$ and $h_{i,j}$ 
can be approximated exponentially fast in $n$ by functions of 
$$(A_{j-n},A_{j-n+1},...,A_j,A_{j+1},...,A_{j+n}),$$ uniformly in $j$. 

Finally, let us consider a sufficiently fast mixing Markov chain $(X_j)$ and let us take random matrices of the form $A_j=A_j(X_j)$ such that 
$$
\sup_{j}\|A_j-A\|_{L^\infty}\leq \varepsilon.
$$
Then if $\varepsilon$ is small enough the random variables $\la_{i,j}$ and $h_{i,j}$ can be approximated exponentially fast by functions of  $X_{j+k}, |k|\leq r$ as $r\to\infty$.

\subsection{Linear processes}
Let $h_{j}:\cX_j\to\bbR$ be measurable functions, let $(a_k)$ be a sequence of numbers and define 
$$
f_j(...,X_{j-1},X_j,X_{j+1},...)=\sum_{k\in\bbZ}a_kg_{j-k}(X_{j-k}),
$$
assuming that the above series converges. Note that 
$$
\|f_j\|_{L^p}\leq \sum_{k}|a_k|\|g_{j-k}\|_{L^p}
$$
and so if we assume that $\sup_k\|g_k(X_k)\|_{L^p}<\infty$ and that the series $\sum_{k}|a_k|$ converges we get that $\sup_j\|f_j\|_{L^p}<\infty$. Notice that 
$$
\|f_j-\bbE[f_j|X_{k}; |k-j|\leq r]\|_{L^p}\leq \sum_{|k-j|>r}|a_k|\|g_{j-k}\|_{L^p}\leq C\sum_{|k-j|>r}|a_k|
$$
Thus, if  $\sum_{|k-j|>r}|a_k|=O(\delta^r)$ for some $\delta\in(0,1)$ then we conclude that 
$$
\sup_j\|f_j\|_{j,p,p,\delta}<\infty.
$$
\subsection{Iterated random functions driven by inhomogeneous  Markov chains}\label{Iter}
 Iterated random functions (cf.
[15]) are an important class of processes. Many nonlinear models like ARCH,
bilinear and threshold autoregressive models fit into this framework.
We refer to [15 in JiraK be AoP paper] for a survey on such processes, where the case of iid $X_j$ is considered.
Here we describe a non-stationary version of such processes which are based on Markov chains $X_j$ instead of iid sequences.
We believe that considering such processes driven by inhomogeneous Markov chains rather than iid variables could be useful for practitioners.

\subsubsection{Processes with initial condition}
Let $G_k:\bbR\times\cX_k\to\bbR$ be measurable functions. Let
$L_k(x_k)$ denote the Lipschitz  constant of the function $G_k(\cdot,x_k)$.
Define a process recursively by setting $Y_0=y_0$ to be a constant then  setting $Y_k=G_k(Y_{k-1},X_k), k\geq 1$. 
 Notice that for $k\geq1$,
$$
Y_k=G_{k,X_k}\circ G_{k-1,X_{k-1}}\circ\cdots\circ G_{1,X_1}(y_0):=f_k(X_1,...,X_{k-1},X_k).
$$
where $G_{s,X_s}(y)=G_s(y,X_s)$. This fits our model of functions $f_k$ that depend on the entire path of a two sided Markov chain $(X_j)_{j\in\bbN}$ (note that one can always extend $X_j$ to a two sided sequence simple by considering iid copies of $X_0$, say, which are also independent of $X_j, j\geq0$). Namely, by abusing the notation we may write 
$$
f_k(X_1,...,X_{k-1},X_k)=f_k(...,X_{k-1},X_{k},X_{k+1},...)
$$
where the dependence is only on $X_1,...,X_{k-1},X_k$.
\begin{lemma}
Suppose that there are uniformly bounded sets $K_i(x_i)\subset\bbR$ such that 
 $G_{i,X_i}(K_i(X_i))\subset K_{i+1}(X_{i+1})$ almost surely for all $i$. Assume also that $K_1:=\cap_{x_1}K_1(x_1)\not=\emptyset$. Then if we start with   $y_0\in K_1$ then 
 $$
\sup_k\|Y_k\|_{L^\infty}<\infty.
 $$
 In particular this is the case when the functions $G_{i}$ are uniformly bounded.
\end{lemma}

\begin{lemma}
Let us assume that for some $p,q\geq 1$ we have  
\begin{equation}\label{ExpCond}
 \|L_k(X_k)\cdots L_{k-m+1}(X_{k-m+1})\|_{L^p}\leq C\delta^m   
\end{equation}
for all $k$ and $m<k$ and that $C_0:=\sup_s\|G(y_0,X_s)\|_{L^q}<\infty$. Let $a$ defined by $1/a=1/p+1/q$. Then 
$$
\sup_k\|Y_k\|_{L^a}<\infty.
$$
\end{lemma}
\begin{proof}
First, by \cite[Corollary 5.3]{[15]} we have 
$$
|Y_k-y_0|\leq |G_{k,X_k}(y_0)-y_0|+L_{k}(X_k)|G_{k-1,X_{k-1}}(y_0)-y_0|+
L_{k}(X_k)L_{k-1}(X_{k-1})|G_{k-2,X_{k-2}}(y_0)-y_0|
$$
$$
+...+L_{k}(X_k)L_{k-1}(X_{k-1})\cdots L_{2}(X_2)|G_{1,X_1}(y_0)-y_0|.
$$
 Then for $a$ defined by $1/a=1/p+1/q$ we get that 
$$
\|Y_k-y_0\|_{L^p}\leq (C_0+|y_0|)C(1-\delta)^{-1}<\infty.
$$
Thus, $\sup_k\|Y_k\|_{L^a}<\infty$. 
\end{proof}

Next, we need 
\begin{lemma}
Let $b$ be given by $1/b=1/a+1/p$.  Then under the assumptions of the previous Lemma we have 
$$
\sup_{k}\sup_r\delta^{-r}\|Y_k-\bbE[Y_k|X_{k},...,X_{k-r}]\|_{L^b}<\infty.
$$
\end{lemma}
\begin{proof}
Let $r\in\bbN$. We claim that 
$$
\|Y_k-\bbE[Y_k|X_{k},...,X_{k-r}]\|_{L^a}\leq C_1\delta^r
$$
for some constant $C_1$. If $r\geq k-1$ then there is nothing to prove since $Y_k$ depends only on $X_1,...,X_k$. Suppose that $r<k-1$ and let us take arbitrary points $x_1,...,x_{k-r-1}$ with $x_i\in\cX_i$. Then by the minimization property of conditional expectations we have
$$
\|Y_k-\bbE[Y_k|X_{k},...,X_{k-r}]\|_{L^b}\leq \left\|Y_k-G_{k,X_k}\circ G_{k-1,X_{k-1}}\circ\cdots\circ G_{k-r,X_{k-r}}\circ G_{k-r-1,x_{k-r-1}}\circ\cdots\circ G_{1,x_1}(y_0)\right\|_{L^b}
$$
$$
\leq C\delta^{r}\left\|Y_{k-r-1}-G_{k-r-1,x_{k-r-1}}\circ\cdots\circ G_{1,x_1}(y_0)\right\|_{L^a}\leq C\delta^r\left(\|Y_{k-r-1}\|_{L^a}+|G_{k-r-1,x_{k-r-1}}\circ\cdots\circ G_{1,x_1}(y_0)|\right)
$$
where the second inequality uses the H\"older inequality.
Notice next that by applying again  \cite[Corollary 5.3]{[15]} we see that 
$$
|G_{k-r-1,x_{k-r-1}}\circ\cdots\circ G_{1,x_1}(y_0)|\leq |G_{k-r-1,x_{k-r-1}}(y_0)-y_0|+L_{k-r-1}(x_{k-r-1})|G_{k-r-2,x_{k-r-2}}(y_0)-y_0|
$$
$$
+\dots+L_{k-r-1}(x_{k-r-1})\cdots L_2(x_2)|G_{1,x_1}(y_0)-y_0|.
$$
Since the $L^a$ norm of the expression on the right hand side above upper bound with respect to the distribution of $(X_{k-r-1},....,X_1)$ is bounded by some constant $A$ we can always choose points $x_{j}, 1\leq j\leq k-r-1$ such that the above upper bound does not exceed $A$. Thus, 
$$
\sup_{k}\sup_r\delta^{-r}\|Y_k-\bbE[Y_k|X_{k},...,X_{k-r}]\|_{L^b}<\infty.
$$
\end{proof}

Finally, let us discuss when  the condition $\|L_k(X_k)\cdots L_{k+m-1}(X_{k+m-1})\|_{L^p}\leq C\delta^m$ (i.e. condition \eqref{ExpCond}) holds. Clearly, it holds when   $\sup_k\|\bbE[|L_k(X_k)|^p|X_{k-1}]\|_{L^\infty}<1$, and in particular when $X_k$'s are independent and $\sup_k\|L_k(X_k)\|_{L^p}<1$ or when simply $\sup_k\|L_k(X_k)\|_{L^\infty}<1$ (and then we can take $p=\infty$).

Another example for finite $p$'s we have in mind is as follows. Let $\psi_U(1)$ be the first order upper $\psi$-mixing coefficient of the chain $(X_j)$, namely $\psi_U(1)$ is the smallest number such that 
$$
\bbP(A\cap B)-\bbP(A)\bbP(B)\leq \psi_U(1)\bbP(A)\bbP(B)
$$
for all $s\geq 0$ and measurable sets $A\in\sigma\{X_j:j\leq s\}$ and $B\in\sig\{X_{j}: j>s\}$. Note that in the notations of \cite[Eq. (1.6) and Eq. (2.2)]{BradMix} we have $\psi_U(1)=\psi^*(1)-1$.
\begin{lemma}\label{ProdLemm}
Suppose that  $\psi_U(1)<\infty$ (in particular we can just assume that the chain is $\psi$-mixing which implies \eqref{mix} with all $1\leq p\leq \infty$). Let $p\geq1$ and $0<\beta<1$ be numbers satisfying $\varepsilon:=\beta^p(1+\psi_U(1))<1$, and assume that $\sup_j\|L_j(X_j)\|_{L^p}\leq\beta$.
Then for all $k$ and $m<k$ we have 
$$
\left\|\prod_{j=k-m+1}^k L_k(X_k)\right\|_{L^p}\leq \left(\varepsilon^{1/p}\right)^m
$$
and so \eqref{ExpCond} holds with $\delta=\varepsilon^{1/p}$.
\end{lemma}
\begin{proof}
First, by \cite[Lemma 60]{Advances} we have 
$$
\bbE\left[\prod_{j=k-m+1}^k |L_k(X_k)|^p\right]\leq (1+\psi_U(1))^{m}\prod_{j=k-m+1}^{k}\bbE[|L_j(X_j)|^p].
$$
Next, since $\sup_j\|L_j(X_j)\|_{L^p}\leq\beta$, 
$$
\prod_{j=k-m+1}^k\bbE\left[ |L_k(X_k)|^p\right]\leq \beta^{mp}
$$
and so 
$$
\bbE\left[\prod_{j=k-m+1}^k |L_k(X_k)|^p\right]\leq\left(\beta^p(1+\psi_U(1))\right)^m.
$$
Now, recalling that $\varepsilon=\beta^p(1+\psi_U(1))<1$ we get that 
$$
\left\|\prod_{j=k-m+1}^k L_k(X_k)\right\|_{L^p}\leq \left(\varepsilon^{1/p}\right)^m.
$$
\end{proof}



\subsubsection{Processes without initial condition}
Note the the processes $Y_k$ defined in the previous section can never be stationary since $Y_0=y_0$ and $Y_k$ depends on $X_1,...,X_k$.  
Here we considering a related class of recursive sequences  which will be stationary when the chain $(X_j)$ is stationary and the functions $G_k$ conicide.  
Let $G_k:\bbR\times\cX_k\to\bbR$ be like in the previous section. We define recursively $Y_k=G_k(Y_{k-1},X_k)=G_{k,X_k}(Y_{k-1})$. Then there is a measurable functions $f_k$ on $\prod_{j\leq k}\cX_j$ such that 
$$
Y_k=f_k(...,X_{k-1},X_k).
$$
This fits our general framework by either considering functions $f_k$ which depend only the the coordinates $x_j, j\leq k$. 

\begin{lemma}
Let us assume that (see Lemma \ref{ProdLemm}) for some $p,q\geq 1$ there exist $C>0$ and $\delta\in(0,1)$ such that  
\begin{equation}\label{ExpCond11}
 \|L_k(X_k)\cdots L_{k-m+1}(X_{k-m+1})\|_{L^p}\leq C\delta^m   
\end{equation}
for all $k$ and $m<k$ and that $\sup_s\|G(y_0,X_s)\|_{L^q}<\infty$ for some $y_0\in\bbR$. Let $a$ defined by $1/a=1/p+1/q$. Then the above process is well defined and
$$
\sup_k\|Y_k\|_{L^a}<\infty.
$$
\end{lemma}
\begin{proof}
Let us take $k\in\bbZ$ and $1\leq n\leq m$
Notice that 
$$
\left|G_{k,X_k}\circ\cdots\circ G_{k-m,X_{k-m}}(y_0)-G_{k,X_k}\circ\cdots\circ G_{k-m,X_{k-n}}(y_0)\right|
$$
$$
\leq \left(\prod_{j=k-n}^{k}L_j(X_j)\right)\left|G_{k-n-1,X_{k-n-1}}\circ\cdots\circ G_{k-m,X_{k-m}}(y_0)\right|.
$$
Now, by applying again \cite[Corollary 5.3]{[15]} we see that 
$$
\left|G_{k-n-1,X_{k-n-1}}\circ\cdots\circ G_{k-m,X_{k-m}}(y_0)\right|\leq |G_{k-n-1,X_{k-n-1}}-y_0|+L_{k-n-1}|G_{k-n-2,X_{k-2,n}}-y_0|
$$
$$
+...+L_{k-n-1}(X_{k-n-1})\cdots L_{k-m}(X_{k-m})|G_{k-m,X_{k-m}}-y_0|.
$$
Combining the above estimates and using the H\"older ineuqlaity we get that 
$$
\left\|G_{k,X_k}\circ\cdots\circ G_{k-m,X_{k-m}}(y_0)-G_{k,X_k}\circ\cdots\circ G_{k-m,X_{k-n}}(y_0)\right\|_{L^p}\leq C\delta^n.
$$
Thus the sequence 
$$
A_n=G_{k,X_k}\circ\cdots\circ G_{k-n,X_{k-n}}(y_0)
$$
is Cauchy in $L^p$ and thus hence as a limit denoted by $Y_k$. To show that $\sup_k\|Y_k\|_{L^p}<\infty$ we use the above estimates with $n=1$ and take the limit as $m\to\infty$ to get 
$
\|Y_k\|_{L^p}\leq C\delta.
$
\end{proof}

\begin{lemma}
Suppose that there are bounded sets $K_i$ such that $G_{i,x_i}(K_i)\subset K_{i+1}$ for all $i$ and $x$. Suppose also that $K:=\cap K_i\not=0$. Then by taking $y_0\in K$ we get that there is a solution $Y_k$ such that 
$$
\sup_{k}\|Y_k\|_{L^\infty}<\infty.
$$
In particular this is the case when the functions $G_{i}$ are uniformly bounded.
 \end{lemma}

\begin{lemma}
Let $b$ be given by $1/b=1/a+1/p$.  Then under the assumptions of the previous Lemma we have 
$$
\sup_{k}\sup_r\delta^{-r}\|Y_k-\bbE[Y_k|X_{k},...,X_{k-r}]\|_{L^b}<\infty.
$$
\end{lemma}
\begin{proof}
Let $r\in\bbN$. Then
$$
Y_k=G_{k,X_k}\circ\cdots\circ G_{k-r,X_{k-r}}(Y_{k-r})
$$
and so
$$
\|Y_k-\bbE[Y_k|X_{k},...,X_{k-r}]\|_{L^b}\leq \left\|Y_k-G_{k,X_k}\circ G_{k-1,X_{k-1}}\circ\cdots\circ G_{k-r,X_{k-r}}(y_0)\right\|_{L^b}
$$
$$
\leq C\delta^{r}\|Y_k-y_0\|_{L^a}\leq C_1\delta^r
$$
for some constant $C_1>0$.
\end{proof}

\subsubsection{The case of a random environment}
Let $(M,\cB,\bbP_0,\te)$ be an ergodic probability preserving system with $\te$ being invertible. Let $(X_{\om,n})_{n\in\bbZ}, \om\in M$ be a Markov chain in the random environment $(M,\cB,\bbP_0,\te)$. We consider functions measurable $G_\om:\bbR\times\cX_\om\to\bbR$ and define 
$$
Y_{\om,k}=G_{\te^k\om}(Y_{\om,k-1},X_{\om,k}).
$$
Let $L_\om(x)$  denote the Lipschitz  constant of the function $G_\om(\cdot,x)$.  
Then if we assume that for $\bbP$-a.a. $\om$ we have 
$$
\left\|\prod_{j=0}^{n-1}L_{\te^{-j}\om}(X_{\om,k})\right\|_{L^p}\leq C\delta^n
$$
and $\sup_k\|G_{\te^k\om}(y_0,X_{\om,k})\|_{L^q}<\infty$ for some $C>0$, $y_0\in\bbR$ and $\delta\in(0,1)$ we get that 
$$
Y_{\om,k}=f_{\te^k\om}(...,X_{k-1,\om},X_{k,\om})
$$
namely we are in the setup of Section \ref{RDS} and the variance of $\sum_{j=0}^{n-1}Y_{\om,j}$ either grows linearly fast for $\bbP$-a.a. $\om$ or it is bounded for $\bbP$-a.a. $\om$.
\subsection{Application to $\text{GARCH}(\mathfrak p, \mathfrak q)$ sequences}
Assume that $X_j$ are real valued and have zero mean and that $\sup_j\|X_j\|_{L^p}<\infty$ for some $p\geq2$. Let $Y_k=X_kL_k$
where $L_k$ is defined in recursion by
$$
L_k^2=\mu+\alpha_1L_{k-1}^2+...+\alpha_{\mathfrak p}L^2_{k-\mathfrak p}+\beta_1X_{k-1}^2+...+\beta_{\mathfrak q}X^2_{k-\mathfrak q}
$$
with $\mu,\alpha_i,\beta_j\in\bbR, \mu>0$.  We refer to \cite[Example 3.5]{Jirak} for more references  and motivation for considering such processes. 
We assume here that with $r=\max(\mathfrak p,\mathfrak q)$,
$$
\gamma_C=\sum_{i=1}^r\|\alpha_i+\beta_iX_i^2\|_{L^2}<1.
$$
Now, as explained in \cite[Example 3.5]{Jirak}  we have
$$
Y_k=\sqrt{\mu}X_k\left(1+\sum_{n=1}^\infty\sum_{1\leq l_1,...,l_n\leq r}\prod_{i=1}^n(\al_{i_i}+\beta_{l_i}X^2_{k-l_1-...-l_n})\right).
$$
Arguing like in \cite[Example 3.5]{Jirak} one can show that 
$$
\sup_kv_{k,p,\delta}(Y_k)<\infty
$$
for some $\delta\in(0,1)$. Indeed, in \cite{Jirak} only the case when $X_j$ are iid was considered, which led to a similar statement which is suitable to the case of Bernoulli shifts. However, taking a careful look at the arguments shows that what can be done is to approximate exponentially fast in the above sense. Thus we generalize the results of Jirak to GARCH processes generated by inhomogenuous Markov chains, where already the case of independent and not identically distributed random variables $X_j$ seems to be a new result.

\subsection{Applications to dynamical systems: limit theorems for H\"older on average observables}\label{DSY}
In \cite{DolgHaf PTRF 2} 
we described a general method to obtain CLT rates 
for a wide class of expanding or hyperbolic maps and unifomly H\"older continuous functions. In this section we will explain how the methods in this paper can also provide similar results, and in some cases new results that do not follow from the latter papers since here we can consider functions $f_j$ which are only H\"older on average.
\subsection{Non-uniformly expanding maps via Korepanov's semi-conjugacy}

Let us begin with the setup of \cite{Korepanov}. Korepanov considered non-uniformly expanding maps $T:M\to M$ on a metric space $M$ which has a reference probability measure, and the system $(M,T)$ admits a tower extension. The reason this is relevant to our work is that he essentially proved the following theorem.

\begin{theorem}
Let $T:M\to M$ be the class of non-uniformly expanding maps considered in \cite{Korepanov}. Then there exists a two sided Bernoulli shift $(X,\sigma)$ which is semi conjugated with $T$. Moreover, if $(X_j)$ is the underlying iid sequence then for every H\"older continuous function $g:M\to\bbR$ the function $f=g\circ\pi$ (where $\pi$ ois the semi conjugation) satisfies
$$
\sup_{j}v_{j,p,\delta^{1/p}}(f\circ T^j)\leq \|g\|_{\text{Holder}}
$$
for some constant $\delta<1$.
\end{theorem}
Using this theorem the problem reduces to our setup (we can consider uniformly bounded functions which are H\"older continuous on average). 
\subsection{Subshifts of finite type and H\"older on average observables}
Another example which is relevant to our setup is when working with measure of maximal entropy of a subshift of finite type $T$. Let us briefly recall the definition of a subshift of finite type. Let $\cA$ be a finite set and let $(A_{i,j})_{i,j\in\cA}$ be a matrix with $0-1$ entries such that $A^M$ has only positive entries for some $M$. Let $\Sigma=\{(x_i)\in\cA^{\bbN}: A_{x_i,x_{i+1}}=1\}$ and let $\sigma:\Sigma\to\Sigma$ be the left shift. Let $\mu$ be the unique measure of maximal entropy (see \cite{Bowen}).
Then (see \cite{Bowen}) when viewed as random variables $(X_j)_{j\geq0}$ whose path is distributed according  to $\mu$ the coordinates $X_j$ form a $\psi$-mixing Markov chain, and so we can prove optimal CLT rates for partial sums of the form $\sum_{j=0}^{n-1}f_j\circ T^j$ for uniformly bounded functions $f_j$ which are only H\"older continuous on average, that is under the assumption that 
$$
\sup_jv_{j,s,\delta}(f_j)<\infty
$$
for some $s$ and $\delta$. Indeed, Assumption \ref{Ass1} is in force. Note that we can also consider Markov measures on Gibbs-Markov maps \cite{Jon} since also in that setup the coordinates are $\psi$-mixing exponentially fast. 
When the functions are not uniformly bounded but instead are bounded in some $L^p$ norm then we get rates $\sig_n^{1-u}$ where $u\to 0$ as $p\to\infty$.
In fact, even when $f_j=f$ does not depend on $j$ these results seem to be new. Moreover, we can also consider non-stationary SFT $(T_j)$ like in \cite{DolgHaf PTRF 2} since also in that case the coordinates are $\psi$-mixing (see \cite{DolgHaf PTRF 2,Nonlin}). This allows us to prove optimal CLT rates for non-stationary  Markovian piecewise expanding intervals maps and H\"older on average functions $f_j$, see \cite[Section 4]{DolgHaf PTRF 2}.

Another application is to Gibbs Markov maps considered in \cite{Jon}.

By considering symbolic representations we derive the following corollary.
\begin{corollary}
Let $T:M\to M$ be an Anosov map and let $\mu$ be the unique measure of maximal entropy. Let $f_j$ be uniformly bounded functions which are uniformly H\"older on average, that is for every $j$ there exists a measurable function $C_j:M\to\bbR$ such that $\sup_j\int|C_j(x)|^sd\mu(x)<\infty$ and 
$$
|f_j(x)-f_j(y)|\leq (C_j(x)+C_j(y))(\text{dist}(x,y))^\eta
$$
where $\eta\in(0,1]$. 
Then Assumption \ref{Ass1} is in force when lifting this system to the SFT and therefore all the results in Theorem \ref{BE} and \ref{ThWass} hold for $\sum_{j=0}^{n-1}f_j\circ T^j$ when viewed as random variables on the space $(M,\mu)$.
\end{corollary}
\begin{remark}
 When considering non stationary SFT (see \cite{DolgHaf PTRF 2}) we can get results for Markov measures, and so  also for small perturbations of Anosov maps, see \cite[Appendix C]{DolgHaf PTRF 2}.   
\end{remark}

\section{A sequential spectral gap and perturbation theory}
\subsection{A Perron-Frobenius theorem for the transfer operators}
Denote by $\cB_{j,p,a,\delta,+}$ the space of all functions $g$ on $\cZ_j$ such that $\|g\|_{j,p,a,\delta}<\infty$. Then $\cB_{j,p,a,\delta}$ is a Banach space. 
Let us denote by $\ka_j$ the probability law of $(X_j,X_{j+1},...)$. For $g\in L^1(\ka_j)$ define 
$$
\cL_jg(x_{j+1},x_{j+2},...)=\bbE[g(X_{j},X_{j+1},...)|X_{j+1}=x_{j+1},X_{j+2}=x_{j+2},...]=\int g(y,x_{j+1},x_{j+2},...)P_j(dy,x_{j+1})
$$
where $P_j(\cdot,z)$ is the measure given by $P_j(A,z)=\bbP(X_j\in A|X_{j+1}=z)$.
Then the following duality relation holds:
\begin{equation}\label{Dual}
\int g\cdot (f\circ T_j)\,d\ka_j=\int (\mathcal L_j g) f\,d\ka_{j+1}    
\end{equation}
for all functions $g\in L^1(\ka_j)$ and $f\in L^\infty(\ka_{j+1})$. Define
$$
\mathcal L_j^n=\mathcal L_{j+n-1}\circ\cdots\circ \mathcal L_{j+1}\circ \mathcal L_j.
$$
Then 
$$
\mathcal L_j^ng(x_{j+n},x_{j+n+1},...)=\bbE[g(X_{j},X_{j+1},...)|X_{j+n}=x_{j+n},X_{j+n+1}=x_{j+n+1},...].
$$
\begin{theorem}\label{RPF}
Suppose $\varpi_{q_0,p_0}(n)\to 0$ for some $1\leq q_0,p_0\leq \infty$.
Denote by $\textbf{1}$ the constant function taking the value $1$, regardless of its domain. Then for every $\delta\in(0,1)$ there exist a constants $A>0$  such that  for every $j\in\bbZ, n\in\bbN$ and $g\in\cB_{j,q_0,p_0,\delta,+}$,
$$
\|\mathcal L_j^n g-\ka_j(g)\textbf{1}\|_{j+n,p_0,p_0,\delta}\leq 
A\left(v_{j,p_0,\delta}(g)\delta^{n/2}+\|g\|_{L^{q_0}(\ka_j)}\varpi_{q_0,p_0}([n/2])\right)\leq
A\|g\|_{j,q_0,p_0,\delta}\left(\delta^{n/2}+\varpi_{q_0,p_0}([n/2])\right).
$$
If also $q_0\leq p_0$ then there exists a constant $\gamma\in(0,1)$ such that for every $j\in\bbZ, n\in\bbN$ and $g\in\cB_{j,q_0,p_0,\delta,+}$,
$$
\|\mathcal L_j^n g-\ka_j(g)\textbf{1}\|_{j+n,p_0,p_0,\delta}\leq A\|g\|_{j,q_0,p_0,\delta}\gamma^n.
$$
The constants  $A$ and $\gamma$ depend only on $\delta$ and $p_0$ and $q_0$, while the dependence on $q,p$ is through the sequence $\varpi_{q_0,p_0}(n)$ in \eqref{mix}.
\end{theorem}
Note that the theorem shows that the operator norm of $(\mathcal L_j^n-\ka_j):\cB_{j,q_0,p_0,\delta}\to \cB_{j+n,p_0,p_0,\delta}$ does not exceed either $A(\delta^{n/2}+\varpi_{q_0,p_0}([n/2]))$ or
$A\gamma^n$. 
If $\varpi_{q_0,p_0}(n)$ decays exponentially fast we get that the first estimate also provides exponential rates. In the special case when $X_j$ are independent (or are $m$-dependent) we get that for all $n\geq 1$ (or $n\geq 2m+1$),
$$
\|\cL_j^n g-\ka_j(g)\|_{j+n,p_0,p_0,\delta}\leq A\delta^{n/2}v_{j,p_0,\delta}(g).
$$
If $p\geq q$ then we automatically get exponential decay, and it is immediate that the operator norm  when viewed as map from $\cB_{j,p_0,p_0,\delta,+}$ to $\cB_{j+n,p_0,p_0,\delta,+}$ or from  $\cB_{j,q_0,p_0,\delta,+}$ to $\cB_{j+n,q_0,p_0,\delta,+}$  does not exceed  $A\gamma^n$. 
\begin{proof}[Proof of Theorem \ref{RPF}]
For each $m$ let $g_{m}=g_m(X_j,...,X_{j+m})=\bbE[g(X_{j},X_{j+1},...)|X_{j},X_{j+1},...,X_{j+m}]$. Then 
$$
\|g-g_{j,[n/2]}\|_{L^{p_0}(\ka_j)}\leq v_{j,p_0,\delta}(g)\delta^{[n/2]}
$$
and so by the contraction property of conditional expectations,
$$
\left\|\mathcal L_j^n g-\mathcal L_j^n g_{[n/2]}\right\|_{L^{p_0}(\ka_{j+n})}\leq v_{j,p_0,\delta}(g)\delta^{[n/2]}.
$$
where we view $g_{[n/2]}$ as a function on $\cY_j$ which depends only on finitely many coordinates. We also have 
$$
|\ka_j(g)-\ka_j(g_{[n/2]})|\leq v_{j,p_0,\delta}(g)\delta^{[n/2]}.
$$
Thus,
$$
\left\|\mathcal L_j^n g-\mu_j(g)\right\|_{L^{p_0}(\ka_{j+n})}\leq 2 v_{j,p_0,\delta}(g)\delta^{[n/2]}+\left\|\mathcal L_j^n g_{[n/2]}-\ka_j(g_{[n/2]})\right\|_{L^{p_0}(\ka_{j+n})}.
$$
Now by \eqref{mix} the last term on the above right hand side does not exceed $\|g\|_{L^{q_0}(\ka_j)}\varpi_{q_0,p_0}([n/2])$ and so 
$$
\|\mathcal L_j^n g-\ka_j(g)\|_{L^{p_0}(\ka_{j+n})}\leq A'\left(\delta^{n/2}v_{j,p_0,\delta}(g)+\varpi_{q_0,p_0}([n/2])\|g\|_{L^{q_0}(\kappa_j)}\right)
$$
for some constant $A'$. To estimate $v_{j+n,p_0,\delta}(\mathcal L_j^n g-\ka_j(g))$, notice that 
$$
v_{j+n,p_0,\delta}(\mathcal L_j^n g-\ka_j(g))=v_{j+n,p_0,\delta}(\mathcal L_j^n g).
$$
Now, we have 
$$
\|\mathcal L_j^ng-\bbE[\mathcal L_j^ng|X_{j+n+1},X_{j+n+2},...X_{j+n+r}]\|_{L^{p_0}(\ka_{j+n})}
$$
$$
=\left\|\bbE[g(X_{j},X_{j+1},...)|X_{j+n},X_{j+n+1},...]-\bbE[g(X_{j},X_{j+1},...)|X_{j+n},X_{j+n+1},...X_{j+n+r}]\right\|_{L^{p_0}(\ka_{j+n})}
$$
$$
\leq 2 v_{j,p_0,\delta}(g)\delta^{n+r}+\Big\|\bbE[g_{n+r}(X_{j},X_{j+1},...X_{j+n+r})|X_{j+n},X_{j+n+1},...]
$$
$$
-\bbE[g_{n+r}(X_{j},X_{j+1},...X_{n+r})|X_{j+n},X_{j+n+1},...X_{j+n+r}]\Big\|_{L^{p_0}(\ka_{j+n})}.
$$
Finally we note that due to the Markov property we have 
$$
\bbE[g_{n+r}(X_{j},X_{j+1},...X_{j+n+r})|X_{j+n},X_{j+n+1},...]=\bbE[g_{n+r}(X_{j},X_{j+1},...X_{j+n+r})|X_{j+n},X_{j+n+1},...X_{j+n+r}]
$$
and so 
$$
\left\|\bbE[g(X_{j},X_{j+1},...)|X_{j+n},X_{j+n+1},...]-\bbE[g(X_{j},X_{j+1},...)|X_{j+n},X_{j+n+1},...X_{j+n+r}]\right\|_{L^{p_0}(\ka_{j+n})}
\leq 2 v_{j,p_0,\delta}(g)\delta^{n+r}.
$$
Hence,
$$
\sup_r \delta^{-r}\|\mathcal L_j^n-\bbE[\mathcal L_j^n|X_{j+n+1},X_{j+n+2},...X_{j+n+r}]\|_{L^{p_0}(\ka_{j+n})}\leq 2 v_{j,p_0,\delta}(g)\delta^{n}.  
$$
We thus conclude that there exists a constant $A_0=A_0(\delta)$ such that
\begin{equation}\label{Ab}
 \|\mathcal L_j^ng-\ka_j(g)\|_{j+n,p_0,p_0,\delta}\leq A_0\left(v_{j,p_0,\delta}(g)\delta^{n/2}+\|g\|_{L^{q_0}(\ka_j)}\varpi_{q_0,p_0}([n/2])\right).   
\end{equation}
Next, let us assume that $q\leq p$.
 Denote $D_j=\cL_j-\ka_j\textbf{1}$ and $D_{j}^n=D_{j+n-1}\circ\cdots\circ D_{j+1}\circ D_j$. 
 Then, using that $\ka_{j+1}(\cL_j g)=\ka_j(g)$ and $\cL_j(\textbf{1})=\textbf{1}$ we have 
$$
D_j^ng=\mathcal L_j^ng-\ka_j(g)\textbf{1}.
$$
 Now, by \eqref{Ab} there is a constant $B_0>1$ such that for all $j$ and $n$ we have 
$$
\|D_{j,n}g\|_{j+n,p_0,p_0,\delta}\leq B_0\|g\|_{j,q_0,p_0,\delta}.
$$

Next, let us take $n_0$ large enough such that $A_0\left(\delta^{n/2}+\varpi_{q_0,p_0}([n/2])\right)<1/2$. Let us denote the operator norm of $D_{j}^n:\cB_{j,p_0,p_0,\delta}\to\cB_{j+n,p_0,p_0,\delta}$ simply by $\|D_{j}^n\|$. 
Then by \eqref{Ab} and since $p_0\geq q_0$, for all $n\geq n_0$ we have
$$
\|D_{j,n}\|\leq \frac12.
$$
We conclude that if $n=kn_0+s$ for some $k\in\bbN$ and $0\leq s<n_0$ then
$$ 
\|\mathcal L_j^n(g)-\ka_j(g)\textbf{1}\|_{j+n,p_0,p_0,\delta}\leq \left(\prod_{m=1}^{k-1}\left\|D_{j+s+mn_0}^{n_0}\right\|\right)\left\|D_j^{n_0+s}g\right\|_{j+n_0+s,q_0,q_0,\delta}\leq 2^{-k}\|g\|_{j,q_0,p_0,\delta}.
$$
Now, for $n\leq n_0$ we have 
$$
\|\mathcal L_j^n(g)-\ka_j(g)\textbf{1}\|_{j+n,p_0,p_0\delta}=\|D_j^n g\|_{j+n,p_0,p_0,\delta}\leq B_0\|g\|_{j,q_0,p_0,\delta}
$$
and so for all $n\geq 1$ we have 
$$
\|\mathcal L_j^n(g)-\ka_j(g)\textbf{1}\|_{j+n,p_0,p_0\delta}\leq 2B_0 2^{-n/n_0}\|g\|_{j,q_0,p_0,\delta}.
$$
\end{proof}

\subsection{Small complex  perturbations}
Next, given a triangular array $g_{j,n}:\cZ_j\to\bbR, j\leq n$ of functions and $t\in\bbR$ we define 
$$
\cL_{j,t,(n)} h(x)=\bbE[e^{it g_{j,N}(X_j,X_{j+1},,,)}g(X_j,X_{j+1},...)|(X_{j+1},X_{j+2},...)=x]=\int e^{itg_{j,n}(y,x)}h(y,x)P_j(dy,x_{j+1})
$$
and 
$$
\cL_{j,t,(n)}^{m}=\cL_{j+n-1,t,(n)}\circ\cdots\circ\cL_{j+1,t,(n)}\circ\cL_{j,t,(n)}.
$$
Denote 
$$
S_{j,m}g=S_{j,m,(n)}g=\sum_{k=j}^{j+m-1}g_{k,n}(X_k,X_{k+1},...).
$$
When $g_{j,n}=g_j$ does not depend on $n$ we drop the subscript $n$ and write $\cL_{j,t,(n)}=\cL_{j,t}$ and $\cL_{j,t,(n)}^m\cL_{j,t}^m$.
Then by \eqref{Dual} and induction on $m$ we have the following result.
\begin{lemma}\label{CharLemma}
For all $j\in\bbZ$, $n\in\bbN$, $t\in\bbR$ and $h\in L^1(\kappa_j)$ we have $\mathcal L_j^{t,m,(n)}h=\mathcal L_j^m(he^{itS_{j,m}g})$ and 
$$
\kappa_j(e^{itS_{j,m}g})=\kappa_{j+m}(\mathcal L_{j,t,(n)}^m\textbf{1}).
$$
\end{lemma}

\subsection{Smoothness of the perturbation with respect to the parameter under one of Assumptions \ref{Ass1}. \ref{Ass2}, \ref{Ass1.1} or \ref{DomAss}}
\begin{proposition}
(i) Under Assumption \ref{Ass1.1}  the operators $\cL_{j,t}$ with $g_j=f_j$ are of class $C^\infty$ in $t$ with uniformly bounded norms in both $j$ and $t\in[-1,1]$ when viewed as linear maps between $\cB_{j,q,p,\delta}$ to $\cB_{j+1,q,p,\delta}$. 
\vskip0.1cm
(ii) Under Assumption \ref{Ass2} the operators generated by the triangular array $n^{-d}\tilde g_{j,n}$ constructed in Section \ref{RedSec1} are  of class $C^\infty$ in $t$ with uniformly bounded norms in both $j,n$ and $t\in[-1,1]$ when viewed as linear maps between  between $\cB_{j,q,p,\delta'}$ to $\cB_{j+1,q,p,\delta'}$ (for some $\delta'<1$ close enough to $1$).  
\vskip0.1cm
(iii) Under Assumption \ref{Ass1.1} the operators generated by the triangular array $n^{-2a/p}\tilde Y_{j,n}$ constructed in Section \ref{RedSec1} are  of class $C^\infty$ in $t$ with uniformly bounded norms in both $j,n$ and $t\in[-1,1]$ when viewed as linear maps between  between $\cB_{j,q,p,\delta'}$ to $\cB_{j+1,q,p,\delta'}$ (for some $\delta'<1$ close enough to $1$)..  
\vskip0.1cm
(iv) Under Assumption \ref{DomAss} the operators $\cL_{j,t}$ generated by the functions $g_j$ from Lemma \ref{Sinai} are of class $C^3$ in $t$, uniformly in $j$ when viewed as linear operators between $\cB_{j,\infty,\infty,\delta^{1/2}}$ to $\cB_{j+1,\infty,\infty,\delta^{1/2}}$.
\end{proposition}

\begin{proof}
(i) Let $h\in\cB_{j,q,p,\delta}$ be such that $\|h\|_{j,q,p,\delta}\leq1$. Then for every $k$ we have 
$$
\|\mathcal L_{j,t}(f_j^kh)\|_{L^q}\leq \sup_{j}\|f_j\|_{L^\infty}\|\cL_j(|h|)\|_{L^q}\leq C\|h\|_{L^q}\leq C
$$
for some constant $C$. Next, let us take some $r\geq1$. Then by the minimization and contraction properties of conditional expectations,
$$
\left\|\mathcal L_{j,t}(f_j^kh)-\bbE[\mathcal L_{j,t}(f_j^kh)|\cF_{j+1-r-1},\cF_{j+r+1}]\right\|_{L^p}\leq
\leq\|\mathcal L_{j}(e^{itf_j}f_j^kh)-L_{j}(e^{itf_{j,r}}f_{j,r}^kh_{r})\|_{L^p} 
$$
$$
\leq\|e^{itf_j}f_j^kh-e^{itf_{j,r}}f_{j,r}^kh_{r}\|_{L^p}\leq\|f_j\|_{L^\infty}^k\|(e^{itf_j}-e^{itf_{j,r}})h\|_{L^p}+\|(f_{j}^k-f_{j,r}^k)h\|_{L^p}+\|f_{j,r}^k(h-h_r)\|_{L^p}:=I
$$
where $f_{j,r}=\bbE[f_j|\cF_{j-r,j+r}]$ and $h_r=\bbE[h|\cF_{j-r,j+r}]$. Now, since $\sup_j\|f_j\|_{L^\infty}<\infty$ by the mean value theorem we have $|e^{itf_j}-e^{itf_{j,r}}|\leq |f_j-f_{j,r}|$. Since $1/p=1/q+1/q$ we conclude that there is a constant $C_k$ such that
$$
I\leq C_k(\|h\|_{L^q}\|f_j-f_{j,r}\|_{L^s}+\|h-h_r\|_{L^p})\leq C'_k\delta^r. 
$$
Therefore $t\to\cL_{j,t}$ is of class $C^\infty$ and the operator norms are uniforly bounded in $j$ and $t\in[-1,1]$.
\vskip0.1cm
(ii)+(iii) These results are proved similarly to (i) since the reduction is to triangular arrays of functions with uniformly bounded $\|\cdot\|_{\cdot,\infty,s,\delta'}$-norms with $\delta'$ close enough to 1.
\vskip0.1cm
(iv) In view of Proposition \ref{Special Prop} it is clear that under Assumption \ref{DomAss} the operators $\cL_{j,t}$ corresponding to $g_j$ are of class $C^3$ in the sense described in part (iv).  

\end{proof}
\subsection{A complex Perron Frobenius theorem}
Under one of Assumptions \ref{Ass1}, \ref{Ass2} and \ref{Ass1.1} denote $B_j=\cB_{j,q,p,\delta'}$ (where under Assumption \ref{Ass1} we have $\delta'=\delta$). Under Assumption \ref{DomAss} denote $B_j=\cB_{j,\infty,\infty,\delta^{1/2}}$. 
Next, by applying \cite[Theorem D.2]{DolgHaf PTRF 2} we get the following corollary of Theorem \ref{RPF}.
\begin{corollary}\label{Cor11}
There exists $0<\delta_0<1$ such that for every $t\in\bbR$ with $|t|\leq\delta_0$ there are $\lambda_j(t)\in\bbC\setminus\{0\}$, $h_j^{(t)}\in b_J$ and $\ka_j^{(t)}\in B_j^*$ such that $\mu_j^{(t)}(\textbf{1})=\mu_j^{(t)}(h_j^{(t)})=1$, 
 $\lambda_j(0)=1$, $h_j^{(0)}=\textbf{1}$, $\ka_j^{(0)}=\ka_j$
and
\begin{equation}\label{UP}
 \mathcal L_{j,t} h_j^{(t)}=\lambda_j(t)h_{j+1}^{(t)},\, (\mathcal L_{j,t})^*\ka_{j+1}^{(t)}=\lambda_j(t)\ka_{j}^{(t)}.   
\end{equation}
Moreover, $t\to\lambda_j(t)$, $t\to h_j^{(t)}$ and $t\to\mu_j^{(t)}$ are $C^3$ functions of $t$ with uniformly (over $t$ and $j$) bounded $C^3$ norm (under one of Assumptions \ref{Ass1}, \ref{Ass2} or \ref{Ass1.1} they are $C^\infty$). 
Finally,  there are $C_1>0, \delta_1\in(0,1)$ such that for every $g\in B_j$ and all $n$,
\begin{equation}\label{Exp}
 \left\|\mathcal L_j^{t,n}g-\lambda_{j,n}(t)\ka_j^{(t)}(g)h_{j+n}^{(t)}\right\|_{B_{j+n}}\leq C_1\|g\|_{B_j}\delta_1^n   
\end{equation}
where $\lambda_{j,n}(t)=\prod_{k=j}^{j+n-1}\lambda_k(t)$. 
\end{corollary}
We note that the above formulation is for sequences of operators instead of arrays like in the circumstances of Assumptions \ref{Ass2} or \ref{Ass1.1}. However, the result also holds for arrays by considering the operators themselves as the parameters and by setting $\tilde g_{j,n}=\tilde Y_{j,n}=0$ for $j>n$.


\section{Limit theorems for one sided Markov shifts: proofs}

 \subsection{A martingale coboundary representations, the asymptotic behavior of the variance }
  \begin{lemma}\label{Mart lemm}
  Let $g_j:\cZ_j\to\bbR$ be measurable functions. Let \eqref{mix2} hold\footnote{Recall that by Theorem \ref{RPF} \eqref{mix} is equivalent to \eqref{mix2} when $p\geq q$} with some $1\leq q,p\leq\infty$ and suppose that $G:=\sup_j\|g_j\|_{j,q,p,\delta}<\infty$.
Then there are functions $M_j=M_j(g)$ and $h_j=h_j(f)$ on $\cZ_j$ such that almost surely we have
\begin{equation}
\label{MartCob}
g_j(X_j,X_{j+1},...)-\bbE[g_j(X_j,X_{j+1},...)]=M_{j}(X_j,X_{j+1},...)+h_{j+1}(X_{j+1},X_{j+2},...)-h_j(X_j,X_{j+1},...).
\end{equation}
Moreover, $\sup_j\|h_j\|_{j,p,p,\del}<\infty$, $\sup_j\|M_j\|_{j,q,p,\delta}<\infty$ and $M_j(X_j,X_{j+1},...)$  is a reverse martingale difference  with respect to the reverse filtration $\cF_{j,\infty}=\sigma\{X_k: k\geq j\}$.
\end{lemma}

\begin{proof}
Denote 
$\tilde g_j(X_j,X_{j+1},...)=g_j(X_j,X_{j+1},...)-\bbE[g_j(X_j,X_{j+1},...)]$.
Set 
\begin{equation}
\label{DefUj}
h_j=\sum_{k=1}^{\infty}\cL_{j-k}^k\tilde g_{j-k}=\sum_{k=1}^\infty \bbE[\tilde g_{j-k}|\cF_{j,\infty}]
\end{equation}
where for $s<0$ we set $g_s=0$.
Then by Theorem \ref{RPF}, 
$$
\|h_j\|_{j,p,p,\delta}\leq 2A\sum_{k=1}^\infty \gamma^k\|g_{j-k}\|_{j-k,q,p,\delta}\leq 2A(1-\gamma)^{-1}G.
$$
 Set $M_{j}=\tilde g_j+h_j-h_{j+1}\circ T_j$, namely 
$$
M_j(X_j,X_{j+1},...)=\tilde g_j(X_j,X_{j+1},...)+h_j(X_j,X_{j+1},...)-h_{j+1}(X_{j+1},X_{j+2},...).
$$
 It remains to show that $M_{j}(X_j,X_{j+1},...)$ is indeed a reverse martingale difference. To prove that, using that $h_{j+1}$ is measurable with respect to $\cF_{j+1,\infty}$ we have
$$
\bbE[M_j|\cF_{j+1,\infty}]=\bbE[\tilde g_j|\cF_{j+1,\infty}]+\bbE[h_j|\cF_{j+1,\infty}]-h_{j+1}
=\bbE[\tilde g_j|\cF_{j+1,\infty}]+\sum_{k=1}^\infty \bbE[\tilde g_{j-k}|\cF_{j+1,\infty}]-\sum_{k=1}^\infty\bbE[\tilde g_{j+1-k}|\cF_{j+1,\infty}]=0.
$$
\end{proof}

\begin{remark}\label{Mart Rem}
In Assumption \ref{Ass2} we allowed that $\|f_{j}\|_{j,a,s,\delta}=O((j+1)^\zeta)$ for some $0<\zeta<1$. Using that $(j+m)^\zeta\leq j^\zeta+m^{\zeta}$ and that $\sum_{k=0}^{r/2}(k+1)^\zeta \delta^{r-k}$ is of order $\delta^{(\frac12-\rho)^r}$ for all $\rho>0$
it is not hard to show that in this case the arguments in the proof of Lemma \ref{Sinai} yield that 
$\|M_j\|_{j,a,s,\delta^{1/3}}=O((j+1)^\zeta)$ and similarly $\|h_j\|_{j,s,s,\delta^{1/3}}=O((j+1)^\zeta)$.  
\end{remark}

\begin{proof}[Proof of Theorem \ref{Var them}]
First, by Lemma \ref{Sinai} it is enough to prove the theorem for the measurable functions $g_j:\cZ_j\to\bbR$ described there instead of $f_j$.
Let us show that conditions (1)-(3) are equivalent.
By Lemma   \ref{Mart lemm} we can write
$$
\tilde g_j=M_j+h_{j+1}-h_j
$$ 
with $h_j$ and $M_j$ with the properties described in  Lemma   \ref{Mart lemm}, except that in general the sum of the variances of $M_j$ might not converge. Notice now that since $p\geq2$ we have
\begin{equation}\label{M apprx}
\|S_ng-S_nM\|_{L^2}\leq \|S_ng-S_nM\|_{L^p}\leq 2\sup_{j}\|h_j\|_{L^p}<\infty
\end{equation}
where $S_nM=\sum_{j=0}^{n-1}M_j(X_j,X_{j+1},...)$.  

Now assume (1), and let $n_k$ be an increasing sequence such that $n_k\to\infty$ and $\sig_{n_k}\!\!=\!\!\|S_{n_k}g\|_{L^2}\!\!\leq\!\! C$ for some constant $C>0$. Then by \eqref{M apprx}, 
$
B:=\sup_k\|S_{n_k}M\|_{L^2}<\infty.
$
However, since $M_j(X_j,X_{j+1},...)$ is a reverse martingale, we have
$$
\sum_{j=0}^{n_k-1}\text{Var}(M_j)=\|S_{n_k}M\|_{L^2}^2\leq B^2.
$$
Now, since  $V_n:=\|S_nM\|_{L^2}^2=\sum_{j=0}^{n-1}\text{Var}(M_j)$ is increasing we conclude that the summability condition in (3) holds. This shows that (1) implies (3). 

Next,  (2) clearly implies (1). 
Thus, to complete the proof it is enough to show that (3) implies (2), but this also follows from \eqref{M apprx} since the latter yields
$
\|S_{n}g\|_{L^2}^2\leq (V_n+U)^2<\infty.
$

Finally, the proof of the last statement proceeds like the proof of \cite[Theorem 3.5]{BDH}, with minor modifications. 
\end{proof}

\subsection{Quadratic variation and moment estimates}
Recall that the (unconditioned)  quadratic variation difference of the reverse  martingale difference $M_{j}(X_j,X_{j+1},...)$ is given by 
$
Q_j=Q_j(M):=M_{j}^2
$ 
Henceforth we denote $Q_j=M_j^2$ and let
$$
S_{j,n}f=\sum_{k=j}^{j+n-1}f_k \circ T_j^k.
$$
$S_{j,n}g$, $S_{j,n}M$ and $S_{j,n}Q$ are defined similarly. Denote 
$$
G_j=Q_{j}-\bbE[Q_{j}(X_j,X_{j+1},...)]=M_j^2(X_j,X_{j+1},...)-\bbE[M_j^2(X_j,X_{j+1},...)].
$$

\begin{proposition}\label{Thm 4.1'}
Let \eqref{mix} or \eqref{mix2} hold  some $1\leq q, p$, $p\geq2$. Denote $a=\max(q,p)$. 
Assume that $\sup_j\|G_j\|_{j,q,p,\delta^{1/2}}<\infty$ (which by Lemmata
\ref{Sinai} and \ref{Mart lemm} is always the case when $\sup_{j}\|f_j\|_{j,2a,2a,\delta}<\infty$).
Let $u$ be the conjugate exponent of $p$. 
 Then there is a constant $C$ such that for all $j\in\bbZ$ and $n\in\bbN$ have
$$
\text{Var}(S_{j,n} Q)\leq C\left(\sum_{j\leq \ell<j+n}\left(\bbE[(G_\ell)^2]+\|G_\ell\|_{L^u}\right)\right).
$$
When $p=\infty$ (so $u=1$) we have
$$
\text{Var}(S_{j,n} Q)\leq C(1+\text{Var}(S_{j,n} f)).
$$
\end{proposition}
\begin{proof}
First, to simplify the notation let us assume that $j=0$. 
 The argument below is similar to 
 the first part of the proof of \cite[Theorem 4.1]{CR}. First, we write
$$
\bbE[(S_n G)^2] \leq 2
\sum_{0\leq \ell <n}\sum_{0\leq k\leq \ell}\left|\bbE\big[(G_k\circ T_0^k)\cdot (G_\ell\circ T_0^\ell)\big]\right|
=\sum_{k=0}^{n-1}\bbE[(G_k)^2]
+2
\sum_{0\leq \ell <n}\sum_{0\leq k<\ell} \left| \bbE[G_{\ell} \cdot \cL_{k}^{\ell-k}G_{k}]\right|:=I_1+I_2.
$$
Next, by Theorem \ref{RPF}, we have
$$
I_2\leq C_0\sum_{0\leq \ell <n}\sum_{0\leq k<\ell}\|G_\ell\|_{L^u} \|G_{k}\|_{k,q,p,\delta}\gamma^{\ell-k}= C_0\sum_{0\leq \ell<n}\|G_\ell\|_{L^u}\left(\sum_{0\leq k< \ell}\|G_{k}\|_{k,q,p,\delta^{1/2}}\gamma^{\ell-k}\right)
$$
$$
\leq c_0\sum_{0\leq \ell<n}\|G_\ell\|_{L^u}\leq 2c_0\sum_{0\leq \ell<n}\|Q_\ell\|_{L^u}.
$$
for some constant $c_0$ (the first inequality of the last line   uses that $\sup_j\|G_j\|_{j,q,p,\delta^{1/2}}<\infty$). This finishes the proof of the first estimate.

Note that when $p=\infty$ then $u=1$ and $a=\infty$ so that $C_0:=\sup_{j}\|G_{j}\|_{j,\infty,\infty,\delta^{1/2}}<\infty$ and so the above bound yields
$$
I_1\leq 2C_0\sum_{0\leq \ell<n}\bbE[Q_j]+2C_0c_0\sum_{0\leq \ell<n}\bbE[Q_\ell]=
2C_0(1+c_0)\sum_{0\leq \ell<n}\bbE[Q_\ell].
$$
Finally, recall that  
$
\bbE[Q_\ell]\!\!=\!\!\bbE\left[(M_{\ell}(X_\ell,X_{\ell+1},...))^2\right]
$
and, because of the orthogonality property,
$$
\sum_{0\leq \ell<n}\bbE\left[(M_{\ell}(X_\ell,X_{\ell+1},...))^2\right]=\text{Var}(S_{0,n} M).
$$
 Now the second estimate follows from \eqref{M apprx} together with Lemma \ref{Sinai}.
\end{proof}

\subsection{Proof of of Proposition \ref{Mom prop} (i)-(iii)}\label{MomSec1}

To simplify the notation, we will only  prove the proposition when $j=0$.  Moreover, by replacing $f_j$  with $f_j-\bbE[f_j]$ we can and will assume that $\bbE[S_{n}f]=\bbE[S_ng]=0$ for all $n$.
First, let us prove Proposition \ref{Mom prop} (i). By Lemmata \ref{Sinai} and \ref{Mart lemm}, we have 
$$
\|S_nf\|_{L^q}\leq C_q+\|S_n M\|_{L^q}
$$
for some constant $C_q>0$.
Recall the following version of Burkholder's inequality for martingales (see \cite[Theorem 2.12]{PelBook}). Let $\mathfrak{d}_1,....,\mathfrak{d}_n$ be a martingale difference with respect to a filtration $(\cG_j)_{j=1}^n$ on a probability space. Let $D_n=\mathfrak{d}_1+\mathfrak{d}_2+...+\mathfrak{d}_n$ and
$E_n=\mathfrak{d}_1^2+\mathfrak{d}_2^2+...+\mathfrak{d}_n^2$.
Then, for every $s\geq 2$ there are constants $c_s,C_s>0$ depending only on $s$ such that 
\begin{equation}\label{Burk}
 c_p\|E_n\|_{L^{s/2}}^{1/2}\leq \|D_n\|_{L^s}\leq C_p\|E_n\|_{L^{s/2}}^{1/2}. 
\end{equation}
Now, applying \eqref{Burk} with the reverse martingale $(M_j)$ we conclude that 
$$
\|S_n M\|_{L^q}\leq\left(\sum_{j=0}^{n-1}\|M_j\|_{L^q}^2\right)^{1/2}\leq \sum_{j=0}^{n-1}\|M_j\|_{L^q}\leq A_qn
$$
for some constant $A_q$.

Next, let us prove Proposition \ref{Mom prop} (ii). Henceforth we denote $\|\cdot\|_q=\|\cdot\|_{L^q}$. Note that by Lemma \ref{Sinai} it is enough to prove the claim for one sided functionals $g_j$,
Notice also that it is enough to prove the claim for $b$ of the form $b=2^m$ for some $m$. 
We use induction on $m$, with induction hypothesis being that the claim is true with $b=2^m$ and all sequences $(g_j)$ with $\sup_j\|g_j\|_{j,\infty,\infty,\delta^{1/2}}$ 

For  $m=1$ the result  is trivial. 
Suppose that the statement is true for some $m\geq 1$. In order to estimate $\| S_ng\|_{2^{m+1}}$ we first use that by Lemma \ref{Mart lemm},
$$
\|S_ng\|_{2^{m+1}}\leq C+\|S_n M\|_{2^{m+1}}
$$
for some constant $C$, since actually
$\|S_ng-S_nM\|_{L^\infty}$ is bounded in $n$. So it suffices to show that
 \begin{equation}\label{en}
\|S_n M\|_{2^{m+1}}\leq C(1+\|S_n g\|_{2})
\end{equation}
 for an appropriate constant $C$.
 
Applying \eqref{Burk}  with   the (reverse) martingale difference $(M_j)$ we see that 
\begin{equation}\label{Burk1}
\|S_n M\|_{2^{m+1}}\leq a_m\|S_nQ\|_{2^m}^{1/2}
\end{equation}
where $S_n Q$ and
$a_m$ 
depends only on $m$. Applying the induction hypothesis with the sequence of functions $\tilde Q_j=Q_j-\bbE[Q_j]$ which also satisfies $\sup_{j}\|\tilde Q_j\|_{j,\infty,\infty,\delta^{1/2}}<\infty$
 we see that there is a constant $R_m>0$ depending only of $m$ and the constants in the formulation of Proposition \ref{Mom prop} such that 
 $$
 \|S_n\tilde Q\|_{2^m}\leq R_m(1+\|S_n\tilde Q\|_{2}).
$$
 Since $\bbE[S_nQ]=\text{Var}(S_n M)$, 
 Proposition \ref{Thm 4.1'} 
 gives
$$
\|S_n Q\|_{2^m}\leq  \|S_n\tilde Q\|_{2^m}+\bbE[S_nQ]\leq R_m\left(1+C(1+\text{Var}(S_n
g))\right)+ \text{Var}(S_n M)
$$
$$
\leq R'_m(1+\text{Var}(S_n g))+\text{Var}(S_n M)
$$
for some other constant $R_m'$.
Using that  $\sup_n\|S_n g-S_nM\|_{L^\infty}<\infty$ we see that there is a constant 
$C>0$ such that
$
\text{Var}(S_n M)\!\!\leq\!\! C\!
\left(1+\text{Var}(S_n g)\right).
$
 Thus, there is a constant $R_m''>0$ such that  
$$
 \|S_n\tilde Q\|_{2^m}\leq R_m''(1+\text{Var}(S_n g)).
$$
Now \eqref{en} follows from \eqref{Burk1}, completing the proof of the Proposition \ref{Mom prop} (ii).

Now let us prove Proposition \ref{Mom prop} (iii). Let us first focus on the reverse martingale case.
We begin similarly to the proof of part (ii). Applying \eqref{Burk}  with   the (reverse) martingale difference $(M_j)$ we see that 
\begin{equation}\label{Burk11}
\|S_n M\|_{4}\leq a_4\|S_nQ\|_{2}^{1/2}\leq a_4\left(\bbE[S_nQ]\right)^{1/2}+a_4\|S_nQ-\bbE[S_nQ]\|_{2}^{1/2}.
\end{equation}
Notice that $\bbE[S_nQ]=\bbE[(S_nM)^2]\leq C+\|S_nf\|_{L^2}^2$. 
Now Proposition \ref{Mom prop} (iii) follows from Proposition \ref{Thm 4.1'}.

 Next, let us assume that $f_j=f_j(...,X_{j-1},X_j)$ is a forward martingale difference with respect to the filtration $\cF_{-\infty,j}$. We first fix $n$ and for $k\geq 0$ define $Z_{k}=Z_{k,n}=X_{n-k}$, while for $k<0$ we take an iid sequence $(Y_k)_{k<0}$ which is independent of the chain $(X_j)$ and set $Z_k=Y_k, k<0$. 
Let us define $\bar f_{j,n}(Z_j,Z_{j+1},...)=f_{j-n}(...,X_{j-n-1},X_{j-n}), j<n$. 
Then 
$$
\sum_{j=0}^{n-1}f_j(...,X_{j-1},X_j)=\sum_{j=0}^{n-1}\bar f_{j,n}(Z_{j,n},Z_{j+1,n},...).
$$
Notice also that \eqref{mix} holds for the chain $Z_{j}$, uniformly in $n$ (recall that in this case $q\geq p$). Moreover, notice that 
$$
\bbE[\bar f_{j,n}(Z_{j,n},Z_{j+1,n},...)|Z_{j+1,n},....]=\bbE[f_{j-n}(...,X_{j-n-1},X_{j-n})|X_{n-j-1},X_{n-j-2},...]=0
$$
since $f_j$ is a forward martingale. Namely, $\bar f_{j,n}, j<n$ is a triangular array of reversed martingales and so the problem reduces to the case of a reverse martingales.

\subsection{A direct fourth moment estimate-proof of Proposition \ref{Mom prop} (iv)}\label{MomSec2}
Proposition \ref{Mom prop} (iv) follows by expanding 
$$
\bbE[(S_{j,n}f)^4=\sum_{\ell=j}^{j+n-1}\bbE[f_\ell^4]+C_1\sum_{j\leq m<\ell<j+n}\bbE[f_m^2 f_\ell^2]+C_2\sum_{j\leq m<\ell<j+n}\bbE[f_m f_\ell^3]+C_2\sum_{j\leq m<\ell<j+n}\bbE[f_m^3 f_\ell]
$$
for some constants $C_1,C_2>0$
and using the following simple result with $F_j\in\{f_j,f_j^2,f_j^3\}$.
\begin{lemma}
Let $F_j:\cY_j\to\bbR$ be measurable functions. Then for all $j\geq0$   and $k>0$ and $p,q\geq1$ and conjugate exponents $(p_0,q_0)$ and $(p_1,q_1)$  we have 
$$
\left|\mu_j(F_j\cdot(F_{j+k}\circ T_j^k))-\mu_j(F_j)\mu_{j+k}(F_{j+k})\right|
$$
$$
\leq 2\delta^{k/4}(\|F_j\|_{L^{p_0}}v_{j+k,q_0,\delta}(F_{j+k})+\|F_{j+k}\|_{L^{p_1}}v_{j,q_1,\delta}(F_{j}))+2\varpi_{q,p}([k/2])\|F_{j}\|_{L^q}\|F_{j+k}\|_{L^v}
$$
where $v$ is the conjugate exponent of $p$.
\end{lemma}

\subsection{Proof of Theorem \ref{CLT}}
In the circumstances of Theorem \ref{CLT} (i) we will obtain optimal CLT rates later, so let us focus on Theorem \ref{CLT} (ii). Let us first assume that $f_j$ is a reversed martingale difference with respect to $\cF_{j,\infty}$. Then 
 $f_j=M_j$. Since $f_j$ satisfies the Lindeberg condition by a reversed version a Theorem of Brown \cite{Brown} to prove Theorem \ref{CLT} (ii) it is enough to show that
\begin{equation}\label{Veri}
\lim_{n\to\infty}\left(\sig_n^{-2}\sum_{j=0}^{n-1}\bbE[M_j^2|\cF_{j+1,\infty}]\right)=1    
\end{equation}
in probability. Let $D_j=M_j^2-\bbE[M_j^2|\cF_{j+1,\infty}]=G_j-\bbE[G_j|\cF_{j+1,\infty}]$, $G_j=M_j^2-\mu_j(M_j^2)$. Then $Q_j$ is by itself a reverse martingale difference and so
$$
\left\|\sum_{j=0}^{n-1}D_j\right\|_{L^2}^2=\sum_{j=0}^{n-1}\|D_j\|_{L^2}^2=
\sum_{j=0}^{n-1}\|G_j\|_{L^2}^2=o(\sigma_n^4)
$$
where the last inequality uses \eqref{Special condition}. Thus, 
$$
\lim_{n\to\infty}\left\|\sig_n^{-2}\sum_{j=0}^{n-1}D_j\right\|_{L^2}=0
$$
and so in order to prove \eqref{Veri} it is enough to prove that 
\begin{equation}\label{Veri1}
\lim_{n\to\infty}\left(\sig_n^{-2}\sum_{j=0}^{n-1} M_j^2\right)=1    
\end{equation}
in probability. To prove that we notice that for all $\varepsilon>0$,
$$
\bbP\left(\left|\sig_n^{-2}\sum_{j=0}^{n-1} M_j^2-1\right|\geq\varepsilon\right)
=\bbP\left(\left|\sum_{j=0}^{n-1}Q_j\right|\geq\sig_n^2\varepsilon\right)\leq\frac{\bbE[|S_nQ|^2]}{\sigma_n^4\varepsilon} 
$$
and that by Proposition \eqref{Thm 4.1'} under \eqref{Special condition} we have $\bbE[|S_nG|^2]=o(\sig_n^4)$.

Next, let us assume that $f_j=f_j(...,X_{j-1},X_j)$ is a forward martingale difference with respect to the filtration $\cF_{-\infty,j}$. For a fixed $n$ let $\bar f_{j,n},j<n$ and $Z_{j,n}, j\in\bbZ$ be like at the end of the proof of Proposition \ref{Mom prop}. 
Then 
$$
\sum_{j=0}^{n-1}f_j(...,X_{j-1},X_j)=\sum_{j=0}^{n-1}\bar f_{j,n}(Z_{j,n},Z_{j+1,n},...).
$$
and $\bar f_{j,n}, j<n$ is a triangular array of reversed martingales. Now the result follows by the arguments in the case of reversed martingales applied for a fixed $n$ to the above functions and the chain and using a version of the theorem by Brown for arrays of forward martingales.  
 
\subsection{The sequential pressure function and its approximation properties}
Henceforth we assume that the operators $\cL_{j,t}$ are of class $C^k$ in $t$ and that $\bbE[f_j(...,X_{j-1},X_j,X_{j+1},...)]=0$ for all $j$. Then by \eqref{gj}, $\bbE[g_j(X_j,X_{j+1},...)]=0$ for all $j$.

Next since $|\lambda(t)-1|\leq C|t|$ for some $C>0$ for $t$ small enough we can develop a $C^3$ branch $\Pi_j(t)$ of $\lambda(t)$ such that $\Pi_j(0)=0$ and $\Pi_j(t)$ is uniformly bounded.
\begin{lemma}\label{Approx Lemma}
There exist $r_0>0,C_0>0$ and $n_0$ such that on $t\in[-r_0,r_0]$ for all $n\geq n_0$ we can develop a  branch $\Lambda_{j,n}(t)$ of $\ln\mu_{j}(e^{itS_{j,n}g
})=\bbE[e^{it S_{j,n}g}]$ such that for $s=0,1,2,3,...,k$ we have
$$
\left|\Lambda_{j,n}^{(s)}(t)-\sum_{k=j}^{j+n-1}\Pi_k^{(s)}(t)\right|\leq C_0
$$
\end{lemma}
\begin{proof}
Let us define $\Pi_{j}^{(t)}(g)=\mathcal L_j^{(it)}\mu_j^{(t)}(g)h_{j+1}^{(t)}$ and $E_{j}^{(t)}=\mathcal L_j^{it}-\Pi_{j}^{(t)}$. Then, $t\to E_j^{(t)}$ is of class $C^3$ and by \eqref{UP} we have 
$$
\mathcal L_{j+1}^{it}\circ\Pi_{j}^{(t)}=\Pi_{j+1}^{(t)}\circ \mathcal L_j^{(it)}=\Pi_{j+1}^{(t)}\circ\Pi_{j}^{(t)}.
$$
Therefore, if follows by induction on $n$ that
$$
\mathcal L_j^{it,n}-\lambda_{j,n}(t)\mu_j^{(t)}h_{j+n}^{(t)}=\mathcal L_j^{it,n}-   \Pi_j^{it,n}=E_{j+n-1}^{(t)}\circ \cdots \circ E_{j+1}^{(t)}\circ E_{j}^{(t)}:=E_{j}^{t,n}.
$$
Note that by \eqref{Exp},
$$
\|E_{j}^{t,n}\|_{j+n}\leq C_1\delta_1^n.
$$
 Define $\bar E_j^{(t)}=E_j^{(t)}/\lambda_j(t)$. Then by taking $t$ small enough we can ensure that for all $k$ and $m$,
\begin{equation}\label{UP1}
\|\bar E_{k}^{m,t}\|_{k+m}\leq C\delta_2^m    
\end{equation}
where $C>0$ and $\delta_2\in(0,1)$ are constants.

Using the above notations, we have
$$
\mu_j(e^{itS_{j,n}g})=\mu_{j+n}(\mathcal L_j^{it,n}\textbf{1})=e^{\sum_{k=j}^{j+n-1}\Pi_k(t)}\left(1+(\mu_{j+n}(h_{j+n}^{(t)}-1)+\mu_{j+n}(\bar E_{j}^{t,n}\textbf{1})\right)
$$
Notice that $U_{j+n}(t):=\mu_{j+n}(h_{j+n}^{(t)})-1=O(t)$.
Thus for $n$ large enough and $t$ close enough to $0$ we can develop a branch of $\Lambda_{j,n}(t)=\log\mu_j(e^{itS_{j,n}g})$ and 
$$
\Lambda_{j,n}(t)=\sum_{k=j}^{j+n-1}\Pi_k(t)+\ln\left(1+U_{j+n}(t)+\mu_{j+n}(\bar E_{j}^{t,n}\textbf{1})\right).
$$
Notice that the first $k$ derivatives of $t\to \bar E_{j}^{t,n}$ are also of order $O(\delta_2^n)$ (by the differentiation rule for derivatives of products). Hence for $s=0,1,2,3,...,k$ we have 
$$
\left|\Lambda_{j,n}^{(s)}(t)-\sum_{k=j}^{j+n-1}\Pi_k^{(s)}(t)\right|\leq C_0
$$
for some constant $C_0>0$ that might depend on $k$,
and the proof of the lemma is complete.
\end{proof}

\subsubsection{Proof of Theorems \ref{BE} and \ref{ThWass} when the variance grows linearly fast}
First, by Proposition \ref{RedProp} it is enough to prove Theorems \ref{BE} and \ref{ThWass} for the sums $S_ng$. 
Now, by applying Lemma \ref{Approx Lemma} with $j=0$ we see that there exist $C_k,\del_k>0$ such that for all $s\leq k$ and $t\in[-\delta_k,\delta_k]$ we have 
$$
|\Lambda_{j,n}^{(s)}(t)|\leq C_kn.
$$
Now for $n$ large enough we have $\sig_n^2\geq cn$ and so 
$$
|\Lambda_{j,n}^{(s)}(t)|\leq C_k'\sig_n^2.
$$
for some constant $C'_k$. Thus Theorems \ref{BE} and \ref{ThWass} follow from \cite[Theorem 5]{H} and \cite[Theorem 9]{H} and \cite[Corollary 11]{H}.


\subsection{Proof of Theorems \ref{BE} and \ref{ThWass}}\label{Sec Log}
As before, by Proposition \ref{RedProp} it is enough to prove both theorems for $S_ng$ (or $\tilde S_n$ in the notations of Section \ref{RedSec2}).
Let  $\Lambda_{j,n}(t)$ be the branch of $\ln\bbE[e^{itS_{j,n}g
}]$ from Lemma \ref{Approx Lemma}. Note also that by rescaling under one of Assumptions \ref{Ass2} or \ref{Ass1.1} we can always assume that $u=0$ in Theorems \ref{BE} and \ref{ThWass}.
Recall that we assumed that $\bbE[f_j]=0$ which implies that $\bbE[g_j]=0$.
By applying again \cite[Theorem 5]{H} and \cite[Theorem 9]{H} and \cite[Corollary 11]{H}. 
Theorems \ref{BE} and \ref{ThWass} for $S_ng$ will follow from the following result.
\begin{proposition}\label{Growth Prop}
In the circumstances of  Theorems \ref{BE} and \ref{ThWass} there are constants
 $\del_k>0$ and $C_k>0$ such that for all $s\leq k$ we have
\begin{equation}\label{Add}
\sup_{t\in[-\del_k,\del_k]}|\Lambda^{(s)}_{0,n}(t)|\leq C_3\sig_n^2.    
\end{equation}
\end{proposition}
\begin{remark}\label{k remark}
Notice that it is enough to prove \eqref{Add} with $s=k$. Indeed,  if $|t|\leq \delta_3$ then
$$
\Lambda^{(k-1)}_{0,n}(t)=\int_{0}^t\Lambda^{(k)}_{0,n}(x)dx=O(\sig_n^2)
$$
and similarly $\Lambda^{(s)}_{0,n}(t)=O(\sig_n^2)$ for $s<k-1$.
\end{remark}

Set 
\begin{equation}\label{pi def}
\Pi_{j,n}(t)=\sum_{\ell=j}^{j+n-1}\Pi_\ell(t).
\end{equation}

\begin{lemma}\label{L35}
Let $B$ be a constant and let $s\geq 2$. Then if $B$ is
sufficiently large there are constants $D$ and $r_0$ depending only on $B$ and $s$ so that
for every $t\in[-r_0,r_0]$ and each $j,n$ such that $B\leq \text{Var}(S_{j,n}g)\leq 2B$ we have
$$
|\Pi_{j,n}^{(k_0)}(t)|\leq D(1+\bbE[|S_{j,n}f-\bbE[S_{j,n}f]|^{k_0}])
$$
\end{lemma}
\begin{proof}
Applying \cite[Lemma 43]{DolgHaf PTRF 1} with $S=S_{j,n}g=S_{j,n}g-\bbE[S_{j,n}g]$ we see that  there is $r=r(B)$ such that if $t\in[-r,r]$ then
\begin{equation}\label{Then}
 |\Lambda_{j,n}^{(k_0)}(t)|\leq D_0\bbE[|S|^{k_0}]   
\end{equation}
for some constant $D_0$. Now the result follows since $\|S_{j,n}f-S_{j,n}g\|_{L^{k_0}}\leq C$ for some constant $C$ (see \eqref{approx11}, \eqref{Sg approx 3} and Lemma \ref{Sinai}).
\end{proof}

\begin{proof}[Proof of Proposition \ref{Growth Prop}]
Since $\sig_n=\|S_n\|_{L^2(m_0)}\to\infty$, using the martingale coboundary representation from Lemma \ref{Mart lemm}, given $B>0$ large enough we can decompose 
$\{0,...,n\!-\!1\!\}$ into a disjoint union of intervals $I_1,...,I_{k_n}$ in $\bbZ$ so that $I_j$ is to the left of $I_{j+1}$  and 
\begin{equation}\label{Var block}
B\leq \text{Var}(S_{I_j}g)\leq 2B
\end{equation}
where $S_Ig=\sum_{j\in I}g_j(X_j,X_{j+1},...)$ for every interval $I$. Under Assumption \ref{Ass1.1} we can just work with such a decomposition which is given there.
Now, by Lemma \ref{Mart lemm} there is a constant $A>0$ independent of $B$ such that 
$
\left|\|S_ng-\bbE[S_ng]\|_{L^2}-\left(\sum_{\ell=1}^{n-1}\text{Var}(M_\ell)\right)^{1/2}\right|\leq A
$
and for each $j$ we have
$
\left|\|S_{I_j}g-\bbE[S_{I_j}g]\|_{L^2}-\left(\sum_{k\in I_j}\text{Var}M_k\right)^{1/2}\right|\leq A.
$

Hence, if we also assume that $B>(4A)^2$ then it follows that 
\begin{equation}\label{kn prop}
C_1\leq k_n/\sig_n^2\leq C_2
\end{equation}
for some constants $C_1,C_2>0$ which depend only on $B$. Again, under Assumption \ref{Ass1.1} this is also guaranteed.
Next, let $\Pi_{I}(t)=\sum_{k\in I}\Pi_k(t)$. Then
by Lemma \ref{L35} there are constants $r_0>0$ and $D_0$ such that
$$
\sup_{t\in[-r_0,r_0]}\left|\Pi_{I_j}^{(k_0)}(t)\right|\leq D_0(1+\bbE[|S_{I_j}f|^{k_0}]).
$$
Hence,
\begin{equation}\label{pii bound}
\sup_{t\in[-r_k,r_k]}\left|\Pi_{0,n}^{(k_0)}(t)\right|\leq D_0 k_n+\sum_{j=1}^{k_n}\bbE[|S_{I_j}f|^{k_0}]=O(\sig_n^2)
\end{equation}
where we used Assumption \ref{Stand MomAss}.
Combining this  with Lemma \ref{Approx Lemma} and taking into account that $\sig_n\to\infty$ we see that 
$$
\sup_{t\in[-r_0,r_0]}\left|\Lambda_n^{(k_0)}(t)\right|\leq \tilde D\sig_n^2
$$
for some constant $\tilde D$, and the proof of the proposition is complete. 
\end{proof}

\section{Large deviations}
For a complex number $z$ let us define 
$$
L_{j,z}(h)=\cL_j(e^{zg_j}h)
$$
where $g_j$ are the functions from Lemma \ref{Sinai}. Denote $B_j=\cB_{j,\infty,\infty,\delta^{1/2}}$. 
Then, since $\sup_j\|g_j\|_{j,\infty,\infty,\delta^{1/2}}<\infty$ we see that $L_{j,z}$ are uniformly analytic in $z$.
By applying \cite[Theorem D.2]{DolgHaf PTRF 2} we get the following corollary of Theorem \ref{RPF}. 
\begin{corollary}\label{Cor1}
There exists $0<\delta_0<1$ such that for every $z\in\bbC$ with $|z|\leq\delta_0$ there are $\lambda_j(z)\in\bbC\setminus\{0\}$, $h_j^{(z)}\in b_J$ and $\ka_j^{(z)}\in B_j^*$ such that $\mu_j^{(z)}(\textbf{1})=\mu_j^{(z)}(h_j^{(z)})=1$, 
 $\lambda_j(0)=1$, $h_j^{(0)}=\textbf{1}$, $\ka_j^{(0)}=\ka_j$
and
\begin{equation}\label{UP11}
 L_{j,z} h_j^{(z)}=\lambda_j(z)h_{j+1}^{(z)},\, (\mathcal L_{j,z})^*\ka_{j+1}^{(z)}=\lambda_j(z)\ka_{j}^{(z)}.   
\end{equation}
Moreover, $z\to\lambda_j(z)$, $z\to h_j^{(z)}$ and $z\to\mu_j^{(z)}$ are analytic functions of $z$ with uniformly (over $z$ and $j$) bounded derivatives.
Finally,  there are $C_1>0, \delta_1\in(0,1)$ such that for every $g\in B_j$ and all $n$,
\begin{equation}\label{Exp11}
 \left\|L_j^{z,n}g-\lambda_{j,n}(z)\ka_j^{(z)}(g)h_{j+n}^{(z)}\right\|_{B_{j+n}}\leq C_1\|g\|_{B_j}\delta_1^n
\end{equation}
where $\lambda_{j,n}(z)=\prod_{k=j}^{j+n-1}\lambda_k(z)$. 
\end{corollary}
Arguing like in the previous section we can prove the following result.
\begin{lemma}\label{Approx Lemma1}
There exist $r_0>0,C_0>0$ and $n_0$ such that for every complex number $z$ with $|z|\leq r_0$ for all $n\geq n_0$ we can develop a  branch $\Lambda_{j,n}(z)$ of $\ln\mu_{j}(e^{zS_{j,n}f})$ such that for $s=0,1,2,3$ we have
$$
\left|\Lambda_{j,n}^{(s)}(z)-\sum_{k=j}^{j+n-1}\Pi_k^{(s)}(z)\right|\leq C_0.
$$
\end{lemma}
Relying on the above corollary and lemma the proof of Theorems \ref{LDP thm} and \ref{MDP} is standard and it is based  on the Gartner Ellis theorem (see \cite{LD}). Indeed, in the case of random environment we have $\la_j(z)=\la_{\te^\omega}(z)$ and so by the mean ergodic theorem
$$
\lim_{n\to\infty}\frac1n\ln\bbE[e^{tS_n}]=\int\ln\la_\omega(t)d\mathbb P(\omega).
$$
Now, notice that (see the arguments in \cite[Ch.5]{HK}) $\int\ln\la_\omega(t)d\mathbb P(\omega)=1-\frac {t^2\Sigma^2}{2}+O(t^3)$ where 
$$
\Sigma^2=\lim_{n\to\infty}\frac1n\text{Var}(S_n^\omega f)>0.
$$
This completes the proof of Theorem \ref{LDP thm}.

The proof of Theorem \ref{MDP} proceeds similarly to the proof of \cite[Theorem]{Nonlin} and it is based on Taylor expansion of order $2$ of the functions $\ln(\la_j(z))$ around the origin. The exact details are left for the reader.

\section{Special cases with linearly fast growing variances}\label{LinSec}
\subsection{Markov shifts in random dynamical environment}\label{RDS}

Let $\cZ_\om=\prod_{j\geq 0}\cX_{\te^j\om}$ and $\cY_\om=\prod_{j\in\bbZ}\cX_{\te^j\om}$. Let $\pi_\om:\cY_\om\to\cZ_\om$ be the natural projection. Let us first formulate a version of Lemma \ref{Sinai} that allows some growth rates. 
\begin{lemma}\label{Sinai1}
Let $f_\om:\cY_\om\to\bbR$ be random measurable functions such that $\om\to \|f_\om\|_{\om,q,a,\delta}\in L^d(\bbP_0)$ for some $d,a, q\geq 1$. Then there exist random functions $u_\om:\cY_\om\to\bbR$ and $g_\om:\cZ_j\to\bbR$ such that $|u_\om\|_{\om,a,a,\delta^{12/-\eta}}\in L^d(P_0)$, $\|g_\om\|_{\om,\min(a,q),a,\delta^{1/2-\eta}}\in L^d(P_0)$, when $\eta$ is an arbitrarily small positive number.
and
$$
f_\om=u_{\te\om}\circ T_\om-u_\om+g_\om\circ\pi_\om.
$$
\end{lemma}
\begin{proof}
The proof of this lemma proceeds similarly to the proof of Lemma \ref{Sinai}. Let us prove a  brief explanation. We define 
$$
u_\om=\sum_{k=0}^\infty \left(f_{\te^{k\om}}\circ T_\om^k-\bbE[f_{\te^k\om}\circ T_\om^k|X_{\om,0},X_{\om,1},...]\right)
 $$
 Then like in the proof of Lemma \ref{Sinai} we get that
 $$
\|u_\om\|_{L^a}\leq 2\sum_{k\geq 0}v_{\te^k\om,a,\delta}(f_{\te^k\om})\delta^{k}.
$$
Now, since $Q(\om)=v_{\om,a,\delta}(f_{\om})\in L^d(P_0)$ we see that 
$$
\left\|\sup_{k\geq 0}(k+1)^{-2} v_{\te^k\om,a,\delta}(f_{\te^k\om})\right\|_{L^d}\leq \|Q\|_{L^d(\bbP_0)}\sum_{k\geq 1}k^{-2}.
$$
Thus, there exists a random variable $A(\om)\in L^d(\bbP_0)$ such that 
$$
v_{\te^k\om,a,\delta}(f_{\te^k\om})\leq A(\om)(k+1)^2, k\geq0.
$$
Hence $\om\in \|u_\om\|_{L^a}\in L^d(\bbP)$.

Next, like in the proof of Lemma \ref{Sinai},
 $$
\left\|u_\om-\bbE[u_\om|X_{\om,-r},...,X_{\om,r}]\right\|_{L^a}\leq 2\sum_{k>r/2}v_{\te^k,a,\delta}(f_{\te^k\om})\delta^{k}+3\sum_{k=0}^{r/2}v_{\te^k\om,a,\delta}\delta^{r-k}\leq CA(\om)\delta^{(\frac12-\eta)r}.
 $$
\end{proof}

\begin{proof}[Proof of Theorem \ref{VarRDS}]
Since $d>2$ by Lemma \ref{Sinai1} it is enough to consider the case when $f_\om$ is actually a function on $\cZ_\om$. Let $\cL_\om$ denote the transfer operator of $\tau_\om$ with respect to the measures $\ka_\om$ and $\kappa_{\te\om}$, where $\ka_\om$ is the law of $(X_{\om,k})_{k\geq0}$.
Let $\chi_\om=\sum_{k=1}^{\infty}\cL_{\te^{-k}\om}^k\tilde f_{\te^{-k}\om}$ where 
$$
\cL_\om^k=\cL_{\te^{k-1}\om}\circ\cdots\circ\cL_{\te\om}\circ\cL_\om
$$
and $\tilde f_\om=f_\om-\ka_\om(f_\om)$. Then $\om\to \|\chi_\om\|_{\om,q,p,\delta}\in L^d(P_0)$.
Indeed, by Applying Theorem \ref{RPF} and taking into account that $\varpi_{\om,q,p}(n)=\varpi_{q,p}(n)=O(\gamma^n)$ for some $\gamma\in(0,1)$ we get that  
$$
L_{\om,k}:=\|\cL_{\te^{-k}\om}^k\tilde f_{\te^{-k}\om}\|_{\om,p,p,\delta}\leq C\gamma^k\|f_{\te^{-k}\om}\|_{\te^{-k}\om,p,q,\delta}
$$
and so 
$$
\|L_{\om,k}\|_{L^d}\leq C'\gamma^k
$$
for some constant $C'$. Next, like in the proof of Lemma \ref{Mart lemm}
$$
\tilde f_\om=M_\om+\chi_{\te\om}\circ\tau_\om-\chi_\om
$$
where $M_{\te^k\om}\om(X_{k,\om},X_{k+1,\om},...)$ is a revers martingale difference. 
Note that $\om\to\|M_\om\|_{\om,p,q,\delta}\in L^{d}(P_0)$. Since $\|\chi_\om\|_{L^2(\ka_\om)}\in L^2(P_0)$ by the mean ergodic theorem we have $\|\chi_{\te^k\om}\|_{L^2(\ka_{\te^k\om})}=o((k+1)^{1/2})$ and so 
$$
\text{Var}(S_n^\om f)=o(n)+\text{Var}(S_n^\om M)=o(n)+\sum_{j=0}^{n-1}\ka_{\te^j\om}[(M_{\te^j\om})^2].
$$
Since $\om\to\|M_{\om}\|_{L^2(\kappa_\om)}\in L^2(P_0)$  it follows by the mean ergodic theorem that 
$$
\lim_{n\to\infty}\frac1n\text{Var}(S_n^\om f)=\int_{M}\int_{\cZ_\om}|M_\om(x)|^2d\kappa_\om(x)\, d\bbP_0(\om):=\Sigma^2.
$$
Now, $\Sigma^2=0$ if and only if $M_\om=0$ for $\bbP_0$-a.a $\om$. The equivalence between the representation $\tilde f_\om=\chi_{\te\om}\circ\tau_\om-\chi_\om$ to the more general representation $f_\om=H_{\te\om}\circ T_\om-H_\om$ is obtained using the same arguments like in the proof of \cite[Theorem 4.1]{BDH}.
\end{proof}

\subsection{Small perturbations of homogeneous  Markov shifts}\label{Pert}
Suppose that there is a measurable space $\cX$ such that $\cX_j=\cX$ for all $j$, that all the state space of $X_j$  coincide. 
Suppose also that for every $j$ and $x\in\cX$ there is a measure $P_j(\cdot,x)$ on $\cX$ such that $\bbP(X_{j}\in A|X_{j+1}=x)=P_j(A,x)$. Then 
$$
\cL_j g(x)=\int g(y,x)P_j(dy,x_{j+1}), x=(x_j,x_{j+1},...).
$$
Then the proof of Theorem \ref{RPF} when $p=\infty$ proceeds the same with the norm $\|\cdot\|_{j,\infty,\infty,\delta}$ also work with the following variation $\tilde v_{j,\infty,\delta}$ of $v_{j,\infty,\delta}$ which is given by 
$$
\tilde v_{\infty,\delta}(g)=\sup_r\delta^{-r}\inf_{G}\sup_{x=(x_k)_{k\geq0}\in\cX^\bbN}\left|g(x)-G(x_0,x_{1},...,x_{r})\right|
$$
where the supremum is taking over all measurable functions $G:\cX^{r+1}\to\bbR$. We can also replace the $L^\infty$ norm by the usual supremum norm. That is, we consider the following norm instead of $\|\cdot\|_{j,\infty,\delta}$,
$$
\|g\|_{1,\infty,\delta}=\sup_{y\in\cX}|g(y)|+\tilde v_{\infty,\delta}(g).
$$

Next, let $Z_j$ be an homogeneous\footnote{Namely, $\bbP(Z_{j+1}\in \Gamma|Z_j=x)=\bbP(Z_1\in\Gamma|Z_0=x)$ for all $j$, $x$ and a measurable set $\Gamma$ on the common state space of $Z_j$.} Markov chain satisfying \eqref{mix}. 
Let us denote by $T$ the corresponding left shift. Let us suppose that there is a family of measures $P(\cdot,x),x\in\cX$ on $\cX$ such that 
$$
Lg(x)=\bbE[g(Z_0,Z_1,...)|(Z_1,Z_2,...)=x]=\int g(y,x)P(dy,x).
$$

Let us assume that 
Let $f:\cY_0\to\bbR$ be a measurable function such that $\|f\|_{1,\infty,\delta}<\infty$ for some $\delta$. We assume that 
$$
\sigma_f^2:=\lim_{n\to\infty}\frac1n\text{Var}(\cS_nf)>0
$$
where $\cS_nf=\sum_{j=0}^{n-1}f\circ T^j$. This limit exists by Theorem \ref{VarRDS} in the case when $M$ is a singelton.
 Let $\varepsilon>0$. We assume that 
$$
\sup_{j}\sup_{z\in\cX}\sup_{A\in\cG}\left|\bbP(X_j\in A|X_{j+1}=z)-\bbP(Z_0\in A|Z_1=z)\right|\leq \varepsilon.
$$
Next, let us take a sequence of measurable functions $f_j:\cY_0\to\bbR$ such that 
$$
\sup_j\left(\sup_{y\in\cY_0}|f_j(y)-f(y)|+\tilde v_{\infty,\delta}(f_j-f)\right)\leq\varepsilon
$$
\begin{theorem}
There exist $\varepsilon_0,c>0$ and $m_0\in\bbN$ such that if $\varepsilon\leq\varepsilon_0$  then for all $n\geq n_0$ we have $\text{Var}(S_nf)\geq cn$.   
\end{theorem}

\begin{proof}
 There exists an absolute constant $C>0$ such that for every $j$ and every complex $\zeta$ such that $|\zeta|\leq 1$ we have
$$
\|\cL_{j,\zeta}g-L_{\zeta}g\|_{\infty}\leq C\|g\|_\infty\varepsilon.
$$
Now by applying Lemma \ref{Approx Lemma1} in the homogeneous setting and denoting the pressure function simply by $\Pi(t)$, we see that
$$
|\text{Var}(\cS_n f)-n\Pi''(0)|\leq C
$$
and similarly by Lemma \ref{Approx Lemma1} applied with the inhomogeneous chain,
$$
\left|\text{Var}(S_n f)-\sum_{j=0}^{n-1}\Pi_j''(0)\right|\leq C.
$$
By applying \cite[Theorem D.2]{DolgHaf PTRF 2} with the parametrized family of operators we see that if $\varepsilon$ is small enough then 
$$
\sup_{j}|\Pi''(0)-\Pi_j''(0)|\leq \delta(\varepsilon)\to 0\,\text{ as }\,\varepsilon\to 0.
$$
Thus for $\varepsilon$ small enough and $n$ large enough we have 
$$
\text{Var}(S_n f)\geq \frac n2\sig_f^2.
$$
\end{proof}

\begin{remark}\label{ConvRem}
 The proof reveals that when 
 $$
 \lim_{j\to\infty}\left(\sup_{y\in\cY_0}|f_j(y)-f(y)|+\tilde v_{\infty,\delta}(f_j-f)\right)=0
 $$   
 and 
 $$
\lim_{j\to\infty}\sup_{z\in\cX}\sup_{A\in\cG}\left|\bbP(X_j\in A|X_{j+1}=z)-\bbP(Z_0\in A|Z_1=z)\right|=0
 $$
 then 
 $$
\lim_{n\to\infty}\frac1n\text{Var}(S_nf)=\sig_f^2.
 $$
 Indeed, one can omit the first $j$ summands for $j$ large enough and then repeat the arguments with an arbitrarily small $\varepsilon$. Moreover, in this case also Theorem \ref{LDP} holds for $S_nf$. Indeed, it follows that 
 $$
\lim_{n\to\infty}\left|\frac1n\sum_{j=0}^{n-1}\ln\bar\la_j(t)-\ln\bar\la(t)\right|=0.
 $$
\end{remark}
\subsubsection{A coupling approach}
Another approach to considering small perturbations of a given homogeneous chain $Z=(Z_j)_{j\geq 0}$ with state space $\cZ$ passes through coupling. Fix some $\delta\in(0,1)$. Let us an inhomogeneous chain $X=(X_j)_{j\ge q0}$ and let $p,q\geq 1$. We define 
$$
d_{p,q}(X,Z)=\inf_{\ka\in\cC(X,Z)}\sup_{\|g\|_{j,p,q,\delta}}\sup_{j}\left\|\bbE[g(X_j,X_{j+1},...)|X_{j+1},X_{j+2},...]-\bbE[g(Z_0,Z_{1},...)|Z_{1},Z_{2},...]\right\|_{j+1,p,q,\delta}
$$
where $\cC(X,Z)$ is the set of all couplings of $X$ and $Z$ and the norm of $g$ is with respect to the measure $\frac12(k_1+\kappa_2)$ where $\kappa_i,i=1,2$ are the marginals of $\kappa$. Moreover, the approximation coefficients are defined using the $\sig$-algebras $\tilde\cF_{m,n}=\sigma\{(X_j,Z_j): m\leq j\leq n\}$.
We can also define 
$$
d_{p,q}((f_j),f)=\inf_{\ka\in\cC(X,Z)}\sup_j\|f_j(X)-f(Z)\|_{j,p,q,\delta}
$$
Then the arguments in the previous section yield:
\begin{lemma}
There exists $\varepsilon>0$ such that if $d_{3,\infty}(X,Z)\leq\varepsilon$ and $d_{3,\infty}((f_j),f)\leq\varepsilon$ then 
$$
\liminf_{n\to\infty}\frac1n\text{Var}(S_nf)\geq \frac12\lim_{n\to\infty}\frac1n\text{Var}(\cS_nf).
$$
\end{lemma}
Indeed, in this setting we use Assumption \ref{Ass1} and perturb the transfer operator of $Z$ with respect to the $\|\cdot\|_{\cdot,\infty,3,\delta}$ norms.
\section*{Appendix A. A detailed discussion on conditions \eqref{1 cond} and \eqref{2 cond}}\label{Discussion}

Assumption \ref{DomAss} is less explicit that Assumptions \ref{Ass1} and \ref{Ass2}, and so we decided to include a detailed discussion when it holds beyond the trivial case that $\sup_{j\geq 0}\|f_j\|_{j,\infty,\infty,\delta}<\infty$.

Condition \ref{1 cond} means that the functions $f_j$ are dominated by functions of the ``present`` $X_j$ and the ``past`` $X_k, k<j$ in an appropriate sense. Indeed the condition holds if there is probability measure $\nu_j$ on 
$\cX_{-\infty,j}=\prod_{k\leq j}\cX_k$ such that
for all measurable sets $A\subset\cX_{-\infty,j}$ we have 
\begin{equation}\label{DomCond}
\bbP((...,X_{j-1},X_{j})\in A|X_{j+1}=z)\leq C\nu_j(A)    
\end{equation}
for a.a. $z$ with respect to the law of $X_{j+1}$ and $B_j:\cX_{-\infty, j}\to\bbR$ such that $\sup_j\|B_j\|_{L^k(\nu_j)}<\infty$ and for $\nu_j$-a.a. $x=(x_{k+j})_{k\in\bbZ}$ and all $r\geq1$,
$$
|f_j(x)|\leq B_j(...,x_{j-1},x_j).
$$
Then case when $B_j$ are uniformly bounded corresponds to $\sup_{j}\|f_j\|_{L^\infty(\mu_j)}<\infty$. Similarly, under \eqref{DomCond} condition \eqref{2 cond} holds if there are functions $B_{s,m}:\cX_{-\infty,s-m}$ such that $\sup_{s,m}\|B_{s,m}\|_{L^k(\nu_{s-m})}<\infty$ and for all $s\geq0$, $r\geq1$ and $0\leq m\leq r$, for $\mu_s$-a.a. $x$ we have
$$
|f_s(x)-F_{s,r}(x)|\leq \delta^r B_{s,m}(....,x_{s-m-1},x_{s-m})
$$
for some functions $F_{s,r}$.
Like before, the case when $B_{s,m}$ are uniformly bounded corresponds to $\sup_sv_{j,\infty,\delta}(f_s)<\infty$. Another interesting situation when condition \eqref{2 cond} holds is when  all $\cX_j$ are metric spaces and $f_j$ are locally H\"older continuous functions as discussed in Remark \ref{Rem Hold}. Namely, we assume 
$$
|f_s(x)-f_s(y)|\leq A_s(x)\left(\rho_s(x,y)\right)^\alpha
$$
for all $x$ and $y$ in $\cY_s$, where $\rho_s$ is defined in \eqref{rho metric}, $\al\in(0,1]$ is a constant and $A_s:\cY_s\to\bbR$ is a measurable function. 
Let $F_{s,r}=f_s(a,X_{s-r},...,X_{s+r},b)$ for  arbitrary points $a\in\prod_{\ell<s-r}\cX_\ell$ and $b\in\prod_{\ell>s+r}\cX_\ell$. The case when $A_s$ are uniformly bounded corresponds to the case when $\sup_{j}v_{j,\infty,\delta}(f_j)<\infty$ where $\delta=2^{-\alpha}$. Then
$$
|f_{s}(x)-F_{s,r}(x)|\leq A_s(x)\delta^{r}
$$
since $x=(x_{s+j})_{j\in\bbZ}$ and $(a,x_{s-r},...,x_{s+r},b)$ identify on the coordinates indexed by $s+j$ with $|j|\leq r$ and $\delta=2^{-\alpha}$. Taking $s=j+m$ with $m\leq r$ and writing $f_s=f_{s}(...,X_{s-1},X_{s},X_{s+1},..)$ and similarly for $F_{s,r}$ we see that
$$
\bbE[|f_{j+m}-F_{j+m,r}|^k|X_j,X_{j+1},...]\leq \delta^{rk}\bbE[|A_{j+m}(...,X_{j+m-1},X_{j+m},X_{j+m+1},...)|^k|X_{j},X_{j+1},...]
$$
and so under \eqref{DomCond} we have 
$$
\bbE[|f_{j+m}-F_{j+m,r}|^k|X_j,X_{j+1},...]\leq C\delta^{rk}\int |A_{j+m}(x,X_j,X_{j+1},...)|^kd\nu_{j-1}(x).
$$
Thus \eqref{2 cond} holds when $A_{s}(...,x_{s-1},x_{s},x_{s+1},...)\leq B_s(...,x_{-2},x_{-1})$ for $s\geq0$ and $\sup_{s,j\geq0}\|B_s\|_{L^k(\nu_{j})}<\infty$ where we view $B_s$ as a function on $\cX_{-\infty,j}$. Note that if the measures $\nu_j$ are consistent (i.e. the restriction of $\nu_{j+1}$ to $\cX_{-\infty,j}$ is equivalent to $\nu_{j}$) then we can just assume  $\sup_{s\geq0}\|B_s\|_{L^k(\nu_{0})}<\infty$. Note also that we can exploit the restriction that $m\leq r$ and assume instead that there are points $a_r\in\cX_{-\infty,-r}$ such that $\|B_s(a_r,\cdot)\|_{L^k(\nu_j)}\leq C2^{\beta r}$ for some $0<\beta<\alpha$. Indeed, in that case we can replace $\delta=2^{-\alpha}$ by $\delta'=2^{-(\alpha-\beta)}$. Similar conditions can be imposed.

Condition \eqref{DomCond} holds, for instance, when there exists a constant $C>0$ such that for all $j$ for almost all $a\in\cX_j$ with respect to the law of $X_j$ for every measurable set $\Gamma\in\cX_{j+1}$ we have 
$$
\bbP(X_{j+1}\in\Gamma|X_j=a)\leq C\bbP(X_{j+1}\in\Gamma).
$$
In this case we can take $\nu_j$ to be the law of $(X_m)_{m\leq j}$,  and so the above consistency condition holds.
In particular, condition \eqref{DomCond} is satisfied when there are probability measures $\nu_j$ on $\cX_j$ and constants $C_1,C_2>0$ such that for every measurable set $\Gamma\subset\cX_{j+1}$ for a.a. $x\in\cX_j$ with respect to the law of $X_j$ we have 
\begin{equation}\label{Doeblin}
C_1\nu_{j+1}(\Gamma)\leq \bbP(X_{j+1}\in\Gamma|X_j=x)\leq C_2\nu_{j+1}(\Gamma).    
\end{equation}
This is the, so called, two sided Doeblin conditions which ensures that the chain is exponentially fast $\psi$-mixing  and then \eqref{mix} holds with every $p$ and $q$.

\end{document}